\documentclass[a4paper,custommargin,times,numbered,print,index]{UCUPThesis}

\ifsetCustomMargin
  \RequirePackage[left=28mm,right=28mm,top=35mm,bottom=30mm]{geometry}
  \setFancyHdr 
\fi

\ifsetCustomFont
  \RequirePackage{helvet}
\fi

\usepackage{pdfpages}

\RequirePackage[labelsep=space,tableposition=top]{caption}

\usepackage{subcaption}

\usepackage{booktabs} 
\usepackage{multirow}

\usepackage{amssymb}
\usepackage{graphicx}
\usepackage{amsfonts}
\usepackage{amsmath}
\usepackage{mathtools}
\usepackage[all]{xy}

\usepackage{amsthm}
\usepackage{mathrsfs}
\usepackage{bbm}

\newcommand\imCMsym[4][\mathord]{%
  \DeclareFontFamily{U} {#2}{}
  \DeclareFontShape{U}{#2}{m}{n}{
    <-6> #25
    <6-7> #26
    <7-8> #27
    <8-9> #28
    <9-10> #29
    <10-12> #210
    <12-> #212}{}
  \DeclareSymbolFont{CM#2} {U} {#2}{m}{n}
  \DeclareMathSymbol{#4}{#1}{CM#2}{#3}
}
\newcommand\alsoimCMsym[4][\mathord]{\DeclareMathSymbol{#4}{#1}{CM#2}{#3}}

\imCMsym{cmmi}{124}{\CMjmath}
\imCMsym[\mathop]{cmsy}{113}{\CMamalg}
\imCMsym[\mathop]{cmex}{96}{\CMcoprod}
\alsoimCMsym[\mathop]{cmex}{97}{\CMbigcoprod}

\newcommand{\Asterisk}{\mathop{\scalebox{1.5}{\raisebox{-0.2ex}{$\ast$}}}}

\newtheorem{theo}[subsection]{Theorem}
\newtheorem{lem}[subsection]{Lemma}
\newtheorem{prop}[subsection]{Proposition}
\newtheorem{coro}[subsection]{Corollary}
\newtheorem{defi}[subsection]{Definition}
\newtheorem{rem}[subsection]{Remark}
\newtheorem{exem}[subsection]{Example}

\let\oldexem\exem
\renewcommand{\exem}{\oldexem\normalfont}

\let\olddefi\defi
\renewcommand{\defi}{\olddefi\normalfont}
\let\oldrem\rem
\renewcommand{\rem}{\oldrem\upshape}

\def\mathfrakdef#1{\expandafter\def\csname#1\endcsname{{\mathfrak#1}}}
\def\mathrmdef#1{\expandafter\def\csname#1\endcsname{{\rm#1}}}
\def\mathsfdef#1{\expandafter\def\csname#1\endcsname{{\rm\sf#1}}}
\def\mathcaldef#1{\expandafter\def\csname#1\endcsname{{\mathcal#1}}}
\mathrmdef{id}\mathrmdef{op}\mathrmdef{coop}\mathrmdef{co}\mathrmdef{tco}
\mathsfdef{Alg} \mathrmdef{Ps} \mathsfdef{Set} \mathsfdef{CAT}\mathcaldef{Ran}\mathcaldef{Lan}\mathrmdef{bi}
\mathrmdef{j}\mathrmdef{h}\mathrmdef{obj}\mathsfdef{CoAlg}\mathrmdef{Id}\mathsfdef{Top}\mathrmdef{id}
\mathsfdef{Cat}\mathsfdef{cat}\mathsfdef{SET}
\def\alg{\mathfrak{alg}}

\def\tttt{\mathfrak{t}}
\def\pppp{\mathfrak{p}}

\def\coalg{\mathfrak{coalg}}
\def\aaa{\mathfrak{A}}
\def\bbb{\mathfrak{B}}
\def\ccc{\mathfrak{C}}

\def\AAA{\mathbb{A}}

\def\DDDD{\mathcal{D}}
\def\AAAA{\mathcal{A}}
\def\BBBB{\mathcal{B}}

\def\YYYY{\mathcal{Y}}

\def\PPPPP{\mathcal{P}}
\def\TTTTT{\mathcal{T}}
\def\SSSSS{\mathcal{S}}

\mathsfdef{Grph}\mathsfdef{Top}\mathsfdef{Cmp}
\mathsfdef{cmp}\mathcaldef{F}\mathsfdef{Der}\mathrmdef{comp}\mathrmdef{t}
\mathsfdef{RGrph}\mathsfdef{Rgrph}\mathrmdef{obj}
\mathcaldef{I}\mathcaldef{L}\mathsfdef{grph}\mathsfdef{Gr}\mathsfdef{Prd}
\mathcaldef{U}\mathcaldef{P}\mathsfdef{gr}\mathcaldef{M}\mathrmdef{J}\mathfrakdef{i}\mathfrakdef{u}
\mathcaldef{R}\mathcaldef{E}\mathsfdef{Rcmp}\mathsfdef{Bicat}\mathsfdef{RCmp}\mathsfdef{Group}\mathsfdef{AbGroup}\mathsfdef{BICAT}
\def\G{\mathfrak{G}}

\def\aaa{\mathfrak{A}}

\mathsfdef{Mod}\mathsfdef{Ring}\mathcaldef{Desc}
\mathcaldef{RAN}\mathrmdef{m}\mathrmdef{n}\mathrmdef{i}
\mathrmdef{suc}\mathrmdef{Suc}

\def\aaa{\mathfrak{A}}
\def\bbb{\mathfrak{B}}
\def\ccc{\mathfrak{C}}

\def\AAA{\mathcal{A}}

\def\hhh {\mathfrak{H}}

\ifuseCustomBib
   \RequirePackage[square, sort, numbers, authoryear]{natbib} 

\fi

\title{Pseudomonads and Descent}

\author{Fernando Lucatelli Nunes}

\crest{\includegraphics[width=0.48\textwidth]{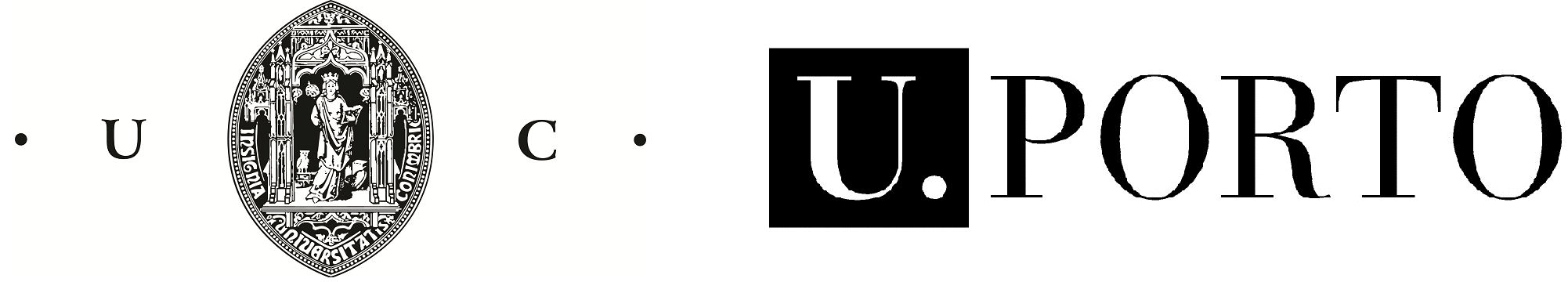}}

\degree{UC|UP Joint PhD Program in Mathematics \\[3mm] Programa Inter-Universit\'{a}rio de Doutoramento em Matem\'{a}tica \\[15mm]PhD Thesis | Tese de Doutoramento}

\degreedate{September 2017}

\ifdefineAbstract
\fi

\ifdefineChapter
\fi

\begin{document}



\frontmatter

\begin{titlepage}
  \vspace*{30mm}
  \maketitle
\end{titlepage}


\begin{acknowledgements}

The first year of my PhD studies was supported by a research grant  
of CMUC, Centre for Mathematics of the
University of Coimbra, under the project Pest-C/MAT/UI0324/2011, co-funded by FCT, FEDER and COMPETE.

The research was supported by CNPq, National Council for Scientific and Technological Development -- 
Brazil (245328/2012-2),  and by the CMUC -- UID/MAT/00324/2013, funded by FCT/MCTES and 
co-funded by the European Regional 
Development Fund through the Partnership Agreement PT2020.

I would like to thank the members of the CMUC for contributing to create 
such an inspiring, helpful and productive atmosphere which surely positively influenced me towards my work.
I am specially grateful to my supervisor Maria Manuel Clementino for
her patience, encouragement, insightful lessons and useful pieces of advice.

\end{acknowledgements}

\begin{abstract}

This thesis consists of one introductory chapter and four 
 single-authored papers written during my PhD studies, with minor adaptations. 
The original contributions of the papers are mainly within the study of 
pseudomonads and descent objects, including applications 
to descent theory, commutativity of weighted bilimits,
coherence and (presentations of) categorical structures.

In Chapter \ref{Introducao PhD},  we give a glance of the scope of our work and 
briefly describe elements of the original contributions of each paper, 
including some connections between them. We also give
a brief exposition of our main setting, which is $2$-dimensional category theory.
In this direction: (1) we give an exposition on the doctrinal adjunction, focusing 
on the Beck-Chevalley condition
as used in Chapter 3, (2) we apply the results of Chapter 5
in a generalized setting of the formal theory of monads and (3) we apply the biadjoint triangle 
theorem of Chapter 4 to study (pseudo)exponentiable pseudocoalgebras.

Chapter 2 corresponds to the paper 
\textit{Freely generated n-categories, coinserters and presentations of low dimensional categories},
\textit{DMUC 17-20} or \textit{arXiv:1704.04474}. 
We introduce and study presentations of categorical structures induced by  
$(n+1)$-computads and groupoidal computads. In this context, we introduce the notion of deficiency and presentations of groupoids via computads. 
We compare the resulting notions with those induced by monads together with a finite measure of objects. In particular,
we find our notions to generalize the usual ones. One important feature of this paper is that we show
that several freely generated structures are naturally given by
 coinserters. After recalling how the category freely generated by a graph $G $ internal to $\Set $
is given by the coinserter of $G$, we introduce higher icons and present the definitions of $n$-computads
via internal graphs of the $2$-category $n\Cat $ of $n$-categories, $n$-functors and $n$-icons. Within this setting,
we show that 
the \textit{$n$-category freely generated by an $n$-computad is also given by a coinserter}. 
Analogously, we demonstrate that the \textit{geometric realization} of a graph $G$ 
consists of a left adjoint functor $\F _ {\Top _1}: \grph\to \Top $ given objectwise by the 
\textit{topological coinserter}. Furthermore, as a fundamental tool to study presentation of
thin and locally thin
categorical structures,  
we give a detailed construction of a $2$-dimensional analogue of   $\F _ {\Top _1} $, denoted by $\F _ {\Top _2}: \cmp\to \Top$.
In the case of group presentations, 
$\F _ {\Top _2}$ formalizes the Lyndon-van Kampen diagrams. Finally, 
we sketch a construction of the $3$-dimensional version $\F _ {\Top _3}$ which
associates a $3$-dimensional CW-complex to each $3$-computad.

Chapter 3  corresponds to the article 
\textit{Pseudo-Kan Extensions and descent theory}, \textit{arXiv:1606.04999} under review.
We develop and employ results on idempotent pseudomonads to get theorems on the general setting of
descent theory, which, in our perspective, is the study of the image of pseudomonadic pseudofunctors. 
After giving a direct approach to prove 
an analogue of Fubini's Theorem
for weighted bilimits and constructing pointwise pseudo-Kan extensions,  we employ the results  
on pseudomonadic pseudofunctors to get theorems 
on commutativity of bilimits.
In order to use these results as the main framework to deal with
classical descent theory in the context of \cite{MR1466540}, we prove that the 
descent category (object) of a 
pseudocosimplicial category (object) is its
conical bilimit. We use, then, this formal approach of commutativity of bilimits 
to (1) recast classical theorems of descent theory,
(2) prove generalizations of such theorems and (3) get new results of descent theory. In this direction,
we give formal proofs of transfer theorems, embedding theorems, a pseudopullback theorem, 
a Galois Theorem and 
the B\'{e}nabou-Roubaud Theorem. We also apply the pseudopullback theorem to 
detect effective descent morphisms
in suitable categories of enriched categories in terms of (the embedding in) internal categories.

Chapter 4 corresponds to the 
article \textit{On Biadjoint Triangles}, published in \textit{Theory and Applications of Categories, Vol 31, N. 9} (2016).
The main contributions are the biadjoint triangle theorems, which have many 
applications in 
$2$-dimensional category theory. 
Examples of which are
given in this same paper: reproving the \textit{Pseudomonadicity characterization} of
\cite{MR1916482}, improving results on the \textit{$2$-monadic approach to coherence} 
of \cite{MR985657, MR1007911, MR1935980}, improving results on \textit{lifting of biadjoints} of 
\cite{MR1007911} and introducing the suitable concept of \textit{pointwise pseudo-Kan extension}.

Chapter 5 corresponds to the 
article \textit{On lifting of biadjoints and lax algebras}, to appear in  
\textit{Categories and General Algebraic Structures with Applications}. 
It can be seen as a complement of the precedent 
chapter,  since it gives further theorems on lifting of biadjoints 
provided that we can describe the categories
of morphisms of a certain domain in terms of weighted (bi)limits. 
This approach, together with 
results on lax descent objects and lax algebras, allows us to get results of lifting of biadjoints
involving (full) sub-$2$-categories of the $2$-category of lax algebras. As a consequence, we complete our 
treatment of the $2$-monadic approach to coherence via biadjoint triangle theorems.

\cleardoublepage
\setsinglecolumn
\chapter*{\centering \Large Resumo}
\thispagestyle{empty}

Esta tese consiste em um 
cap\'{i}tulo introdut\'{o}rio e 
quatro artigos de autoria única, escritos durante os meus estudos de doutoramento.
As contribui\c{c}\~{o}es originais dos artigos est\~{a}o principalmente 
dentro do contexto do estudo de pseudomónadas e objetos de descida,
com aplica\c{c}\~{o}es à teoria da descida, comutatividade de bilimites ponderados, 
coerência e apresenta\c{c}\~{o}es de estruturas categoriais.

No Cap\'{i}tulo \ref{Introducao PhD}, introduzimos aspectos do escopo do 
trabalho e descrevemos alguns elementos das contribui\c{c}\~{o}es originais de cada
artigo, incluindo interrela\c{c}\~{o}es entre elas. 
Damos tamb\'{e}m uma exposi\c{c}\~{a}o básica sobre o principal assunto da tese, 
nomeadamente, teoria das categorias de dimens\~{a}o $2$.
Nesse sentido, (1) introduzimos adjunção doutrinal, focando na condição de Beck-Chevalley,
com a perspectiva adotada no Capítulo 3, (2)
aplicamos resultados do Capítulo 5 em um contexto
generalizado da teoria formal das mónadas e (3) aplicamos o teorema de triângulos biadjuntos
do Capítulo 4 para estudar pseudocoalgebras (pseudo)exponenciáveis.

O Cap\'{i}tulo 2 corresponde ao artigo \textit{Freely generated n-categories, coinserters and presentations of low dimensional categories},
\textit{DMUC 17-20} ou \textit{arXiv:1704.04474}. Neste trabalho, introduzimos e estudamos apresenta\c{c}\~{o}es de estruturas categoriais induzidas 
por $(n+1) $-computadas e computadas grupoidais. Introduzimos a no\c{c}\~{a}o de deficiência de grupóides via computadas.
Comparamos, ent\~{a}o, as no\c{c}\~{o}es resultantes com as no\c{c}\~{o}es induzidas por  mónadas junto com medidas finitas de objetos. Em particular, concluimos que 
nossas no\c{c}\~{o}es generalizam as no\c{c}\~{o}es cl\'{a}ssicas. Outras contribui\c{c}ões do artigo consistiram em mostrar que as propriedades universais
 de v\'{a}rias estuturas livremente geradas  
podem ser descritas por coinserções. Come\c{c}amos por relembrar que as categorias livremente geradas s\~{a}o dadas por coinserções de grafos e, ent\~{a}o, introduzimos 
 \textit{icons} de dimens\~{a}o alta e apresentamos as defini\c{c}\~{o}es de $n$-computadas via grafos internos da $2$-categoria  $n\Cat $ de $n$-categorias, $n$-functores e $n$-\textit{icons}.
Nesse caso, mostramos que a $n$-categoria livremente gerada por uma $n$-computada \'{e} a sua coinserção. 
Analogamente, demonstramos que a realização geométrica de um grafo $G$ \'{e} parte de um functor adjunto à esquerda $\F _ {\Top _1}: \grph\to \Top $
definido objeto a objeto pela coinserção topol\'{o}gica. Al\'{e}m disso, como uma ferramenta fundamental para o estudo de
estruturas categoriais finas e localmente finas, apresentamos uma constru\c{c}\~{a}o detalhada de um an\'{a}logo de dimens\~{a}o $2$ de $\F _ {\Top _1}: \grph\to \Top $,
denotado por $\F _ {\Top _2}: \cmp\to \Top$.
No caso de grupos, 
$\F _ {\Top _2}$ formaliza e, portanto, generaliza o diagrama de Lyndon-van Kampen.
Finalizamos o capítulo dando uma constru\c{c}\~{a}o de uma vers\~{a}o em dimens\~{a}o $3$, denotada por $\F _ {\Top _3}$,
que associa um CW-complexo de dimens\~{a}o $3$ para cada $3$-computada.

O Capítulo 3 corresponde ao artigo
\textit{Pseudo-Kan Extensions and descent theory}, em revisão para publicação.
Desenvolvemos e aplicamos resultados sobre pseudomónadas idempotentes,
obtendo teoremas no contexto geral da teoria da descida que, em nossa perspectiva,
é o estudo da imagem de pseudofunctores pseudomonádicos. Depois de apresentar uma prova direta
do teorema de Fubini para bilimites ponderados e de construir pseudo-extensões de Kan, aplicamos os
resultados sobre pseudomónadas para provar teoremas sobre comutatividade de
bilimites ponderados. Com o objetivo de usar tais resultados como base para lidar com a teoria da descida clássica
no contexto de \cite{MR1466540}, provamos que a categoria (objeto) de descida de uma categoria (objeto)
pseudocosimplicial é seu bilimite cónico. Usamos, então, esse tratamento formal de comutatividade 
de bilimites para (1) recuperar teoremas clássicos da teoria da descida, (2) provar generalizações
desses teoremas e (3) obter novos resultados de teoria da descida. Nesse sentido, apresentamos provas
formais (de generalizações) de teoremas de transferência, de teoremas de mergulho,
do teorema de Galois e do teorema de B\'{e}nabou-Roubaud. Provamos também um resultado sobre
morfismos de descida efetiva em pseudoprodutos fibrados de categorias e o aplicamos para obter morfismos de descida
efetiva em algumas categorias de categorias enriquecidas.

O Capítulo 4 corresponde
ao artigo \textit{On Biadjoint Triangles}, publicado no 
\textit{Theory and Applications of Categories, Vol 31, N. 9} (2016).
As contribuções principais são os teoremas de triângulos biadjuntos,
os quais possuem muitas aplicações em teoria de categorias de dimensão $2$.
Apresentamos exemplos de aplicações no próprio artigo: provamos
explicitamente o teorema de pseudomonadicidade~\cite{MR1916482},
melhoramos resultados sobre o tratamento $2$-monádico do problema de coerência
de \cite{MR985657, MR1007911, MR1935980}, generalizamos resultados de levantamentos de 
biadjuntos 
e introduzimos o conceito
de \textit{pseudo-extensões de Kan} para, então, construir as pseudo-extensões de Kan
via bilimites ponderados.

O Capítulo 5 corresponde ao
artigo \textit{On lifting of biadjoints and lax algebras}, a ser publicado no 
\textit{Categories and General Algebraic Structures with Applications}. 
O principal tema deste artigo é a demonstração de teoremas de levantamento de biadjuntos,
ao assumir que conseguimos descrever a categoria de morfismos de um domínio em termos 
de (bi)limites ponderados. Esse tratamento, junto com resultados sobre objetos de descida lassos e álgebras
lassas, nos permite obter resultados sobre levantamento de biadjuntos envolvendo sub-$2$-categorias (plenas)
da $2$-categoria de álgebras lassas. Como conseqüencia, concluímos nossos resultados de caracterização
sobre o tratamento $2$-monádico do problema de coerência, via teoremas de triângulos biadjuntos.

\end{abstract}


\tableofcontents

%
%





\mainmatter
\chapter{Introduction}\label{Introducao PhD}
The aim of this chapter is to introduce our main setting, which is $2$-dimensional universal algebra, and to give a glimpse of
 the contributions of this thesis. We start by roughly explaining aspects of the interrelation between pseudomonads and 
descent objects in Section \ref{Overview PhD}. Then, in Section \ref{2dimensionalstructures PhD}, we 
introduce basic notions of $2$-dimensional category theory. We take this opportunity to introduce, 
among other concepts, the notion of colax $\TTTTT$-morphisms of lax $\TTTTT$-algebras, 
which is not introduced elsewhere in this thesis. The notion of colax  $\TTTTT$-morphisms are, then, used in 
Section \ref{Teoria Formal das Monadas e levantamentos PhD}
to relate the formal theory of monads with the problem of lifting of biadjoints studied in Chapter 5~\cite{2016arXiv160703087L}. 
We also use the concept
of colax $\TTTTT$-morphisms to talk about doctrinal adjuntion
in Section \ref{Doctrinal Adjunction PhD}, which is a brief exposition of the main theorem of \cite{MR0360749} focusing on the Beck-Chevalley condition as used in our work on the B\'{e}nabou-Roubaud Theorem in Chapter 3.

Sections \ref{Capitulo2 PhD}, \ref{Capitulo3 PhD} and \ref{Capitulo4 PhD} are dedicated to briefly describe elements of the contributions of the  
Chapters 2, 3, 4 and 5, which are respectively the papers \cite{2017arXiv170404474L}, \cite{2016arXiv160604999L}, \cite{MR3491845} and \cite{2016arXiv160703087L}. Finally, the last section is an application of the biadjoint triangles of Chapter 4  in the 
context of exponentiable objects within bicategory theory: we prove that, under suitable hypothesis, a pseudocoalgebra is (pseudo)exponentiable whenever the underlying object is (pseudo)exponentiable.

\section{Overview: Pseudomonads and the Descent Object}\label{Overview PhD}

In $2$-dimensional category theory, 
by replacing strict conditions
 (involving commutativity of diagrams) with pseudo or lax ones 
 (involving a  $2$-cell plus coherence) we get important notions and problems. 
 We briefly describe two examples of these notions below: namely,
  \textit{descent object} and \textit{pseudomonad}.

Firstly, to give an idea of the role of the descent object in $2$-dimensional universal algebra,
 it is useful to make an analogy with the equalizer: while the equalizer encompasses equality and 
 commutativity of diagrams in $1$-dimensional category theory, the descent object and its variations encompass
$2$-dimensional coherence: structure ($2$-cell) plus coherence. 
 
One obvious example of the importance of the descent object is within \textit{descent theory}, 
as introduced by Grothendieck, initially motivated by the problem of
understanding the image of functors induced by fibrations. This theory features a $2$-dimensional
analogue of the sheaf condition: the (strict) gluing condition, 
given by an equalizer of sets, is replaced by 
the descent condition, 
given by a \textit{descent object} of a diagram of categories. 
 
Secondly, analogously to the case of 
monad theory in $1$-dimensional universal algebra, \textit{pseudomonad theory} 
 encompasses aspects of 
$2$-dimensional universal algebra, being useful to study many 
important aspects of $2$-dimensional category theory. Again, in the definition of 
\textit{pseudomonad}, the commutative diagrams of the definition of monad are 
replaced by invertible $2$-cells plus coherence.
In this theory, then, we have $2$-dimensional versions of the features of monad theory. For instance,
adjunctions are replaced by biadjunctions, and we have an Eilenberg-Moore Factorization provided that
we consider the $2$-category of pseudoalgebras, pseudomorphisms and algebra transformations, as it is shown in Section 5 of Chapter 4~\cite{MR3491845}.

The main topic of this thesis is the study 
of $2$-dimensional categorical structures, mostly related with descent theory 
and pseudomonad theory, with applications to $1$-dimensional category theory. For instance:

\begin{itemize}\renewcommand\labelitemi{--}
\item The contributions of Chapter 2  in this context are within the study of freely generated and finitely presented categorical structures. In particular, 
for instance, we deal with presentations of
domain $2$-categories related to the universal property of descent objects. More precisely, presentations of the inclusion of the $2$-categories such that
Kan extensions along such inclusion gives the (strict) descent object;
\item In Chapter 3, we develop an abstract perspective of descent theory in which the fundamental problem 
is the existence of pseudoalgebra structures over objects: more precisely, the image of pseudomonadic pseudofunctors that induce
idempotent pseudomonads. Having this goal, we develop some aspects of biadjoint triangles and lifting of pseudoalgebra structures
involving pseudofunctors that induce idempotent pseudomonads and apply it to get results on commutativity of bilimits. We finish the article, then, 
 applying our perspective to the classical context of \cite{MR1285884, MR1466540};
\item As a (strict) morphism of algebras is given by a morphism plus the commutativity of a diagram, a pseudomorphism between pseudoalgebras
is given by a morphism, an invertible $2$-cell plus coherence. In Chapters 4 and 5, we show that the coherence aspects of pseudomonad theory are encompassed by descent objects and their variations.
More precisely, we show that the category of (lax-)(pseudo)morphisms between (lax-)(pseudo)algebras is given by (lax-)descent objects.
As it is proven and explained in Chapter 5, with these results on the category of (lax-)(pseudo)morphisms, we can prove biadjoint triangle theorems and results on lifting of biadjoints.
\end{itemize}

\section{2-Dimensional Categorical Structures}\label{2dimensionalstructures PhD}

Two of the most fundamental notions of $2$-dimensional categorical structures are those of
\textit{double category} and 
$2$-\textit{category}, both introduced by Ehresmann~\cite{MR0213410}. 
The former is an example of an internal category (introduced in \cite{MR0213410}), while the 
latter is an example of 
enriched category (as introduced in \cite{MR0225841}).

The study of the dichotomy between the 
theory of enriched categories and internal categories, including
 unification theories, is still of much interest. For instance, within the more general setting of  generalized multicategories, we have the introduction of 
$(T,V)$-categories~\cite{MR1957813} (which generalizes enriched categories), $T$-categories~\cite{MR0308236, MR1758246} (which generalizes internal categories) and
 possible unification theories~\cite{MR2770076, MR3337078}. 

Since there is no definitive general framework and the approaches mentioned above are
focused on the examples related to multicategories, we do not follow any of them. 
Instead, we follow the basic idea that internal categories and
enriched categories can be seen as monads in suitable \textit{bicategories}~\cite{MR0220789, MR707614}.

For simplicity, we use the concept of enriched graphs. This is given in Definition 7.1 of Chapter 2~\cite{2017arXiv170404474L},
but we only need to recall that, given a category $V$,  a \textit{$V$-enriched graph $G$} is a 
collection of objects $G(\mathsf{0}) = G_\mathsf{0} $ 
endowed with one object $G(A,B) $ of $V$ for each ordered pair of objects $(A,B) $  of $G_ \mathsf{0} $.

\begin{defi}[Bicategory~\cite{MR0220789}]
A bicategory is a $\CAT $-enriched graph $\bbb  $ endowed with:
\begin{itemize}\renewcommand\labelitemi{--}
 \item Identities: a functor $\mathtt{I}_A: \mathsf{1}\to \bbb (A,A) $ for each object $A$ of $\bbb $;
 \item Composition: a functor, called composition,  
 $\circ = \circ _ {ABC} : \bbb (B,C)\times \bbb(A,B)\to \bbb (A,C) $
for each ordered triple $(A,B,C) $ of objects of $\bbb $; 
 \item  Associativity: natural isomorphisms 
 $$\mathfrak{a}_{ABCD}:  
 \circ _ {ABD} \left( \circ _ {BCD}\times \Id _{\bbb (A,B) }\right)\Rightarrow   
  \circ _ {ACD} \left( \Id _{\bbb (C,D) }\times \circ _ {ABC} \right);
 $$
for every quadruplet $(A, B, C, D) $ of objects of $\bbb $;
 \item Action of Identity: natural isomorphisms 
 $$\mathfrak{e} _{AB}: \circ _ {ABB}\left(\mathtt{I}_B\times \Id  _ {\bbb (A,B) }\right)\Rightarrow \textrm{pro} ^\mathfrak{e}_ {\bbb (A,B) },  $$
 $$ \mathfrak{d}_{AB}:  \circ _ {AAB}\left( \Id _ {\bbb (A,B) }\times \mathtt{I}_A \right)\Rightarrow \textrm{pro}
^\mathfrak{d} _ {\bbb (A,B) },  $$
in which $\textrm{pro}
^\mathfrak{d} _ {\bbb (A,B) } : \bbb (A,B)\times \mathsf{1}\to \bbb (A,B)$ and 
$\textrm{pro} ^\mathfrak{e}_ {\bbb (A,B) }: 1\times \bbb (A,B)\to \bbb (A,B)$ are the invertible projections, for each 
pair $(A, B ) $ of objects in $\bbb $;
\end{itemize}
such that the diagrams
$$\xymatrix@R=6em{
 \circ _ {ABE} \left( \circ _ {BCE}\times \Id _{\bbb (A,B) }\right)\left( \circ _ {CDE}\times \Id _ {\bbb (A, B, C) }\right)
\ar[r]^-{\mathfrak{a}_{ABCE}\ast \id } 
\ar[d]_-{\id _ {{}_{ \circ _ {ABE}}}\ast \left(\mathfrak{a}_{ABCE}\times \id _ {{}_{\Id _{\bbb (A,B) }}} \right) }
&
 \circ _ {ACE} \left( \Id _{\bbb (C,E) }\times \circ _ {ABC}\right)\left( \circ _ {CDE}\times \Id _ {\bbb (A, B, C) }\right)
\ar[d]^-{\mathfrak{a}_{ACDE}\ast \id _ {{}_{\left( \Id _{\bbb (C,D,E) }\times \circ _ {ABC}\right)}}  }
\\
\circ _ {ABE} \left( \circ _ {BDE}\times \Id _{\bbb (A,B) }\right)\left( \Id\times\circ _ {BCD}\times \Id\right)
\ar[rd]_-{\mathfrak{a}_{ABCE}\ast \id}
&
\circ _ {ADE} \left( \Id _{\bbb (D,E) }\times \circ _ {ACD}\right)\left( \Id _ {\bbb(C,D,E) }\times \circ _ {ABC }\right)
\\
&
\circ _ {ADE} \left( \Id _{\bbb (D,E) }\times \circ _ {ABD}\right)\left( \Id\times\circ _ {BCD}\times \Id\right)
\ar[u]_-{\id _ {{}_{\circ _ {ADE}}}\ast\left( \id _ {{}_{\Id _{\bbb (D,E) }}} \times\mathfrak{a}_{ABCE} \right) }
&
}$$
$$\xymatrix@R=6em{
\circ _ {ABC} \left( \circ _ {BBC}\times \Id _{\bbb (A,B) }\right)\left( \Id\times \mathtt{I}_B\times \Id\right) 
\ar[r]^-{\mathfrak{a}_{ABBC}\ast\id }
\ar[rd]|-{\id_{{}_{\circ _ {ABC}}}\ast \left(\mathfrak{d}_{BC}\times \id  \right) }
&
\circ _ {ABC} \left( \Id _{\bbb (B, C) }\times \circ _ {ABB}\right)\left( \Id\times \mathtt{I}_B\times \Id\right)
\ar[d]^-{\id _ {{}_{\circ _ {ABC}}}\ast \left(\id\times \mathfrak{e} _{AB} \right) } 
\\
&
\circ _{ABC}\,\,\left(\textrm{pro}_{(A,B,C)}\right)
}$$
commute for every quintuple $(A,B,C,D,E)$ of objects in $\bbb $, in which $\Id _ {\bbb (A, B, C) }:=\Id _{\bbb (B, C)\times\bbb (A,B) }$,
$\textrm{pro}_ {(A,B,C)}: \bbb (B, C)\times \mathsf{1}\times \bbb (A,B)\to \bbb (B, C)\times  \bbb (A,B)$ is the invertible projection and
  the omitted subscripts of the identities are the obvious ones.

For simplicity, assuming that the structures are implicit, we denote such a bicategory by 
$(\bbb , \circ , \mathtt{I}, \mathfrak{a}, \mathfrak{e}, \mathfrak{d} ) $
or just by $\bbb $.	
For each pair $(A,B)$ of objects of a bicategory $\bbb $, if $f$ is an object of the 
category $\bbb (A,B) $, $f$ is called an $1$-cell of $\bbb $ and it is denoted
by $f: A\to B $. A morphism $\alpha : f\Rightarrow g $ of $\bbb (A,B) $ 
is called a $2$-cell of $\bbb $. 
\end{defi}

\begin{rem}
In order to take advantage of the context of introducing monads, 
internal categories, double categories and enriched categories,  we define $2$-categories via enriched categories below. However, a brief and obvious definition of $2$-category is that of a \textit{strict bicategory}. More precisely, a \textit{$2$-category} is a bicategory such that its natural isomorphisms are identities. We assume this definition herein.
\end{rem}

\begin{rem}
Since the work of this thesis is mainly within 
the tricategory $2\textrm{-}\CAT $ of $2$-categories, pseudofunctors and pseudonatural 
transformations (as defined in Section 2 of Chapter 4~\cite{MR3491845}),  the results and definitions on 
$2$-dimensional category theory of this thesis are within the general setting of bicategories
up to minor trivial adaptations. Specially in this section, since we should consider the bicategories of Definitions \ref{Definicao Matrizes PhD} and \ref{Definicao Spans PhD}, we freely assume these adaptations.
\end{rem}

\begin{rem}
In this chapter, we do not give any further comment on size issues. In this direction, when necessary, 
we implicitly make similar assumptions to those given in Section 1 of Chapter 2~\cite{2017arXiv170404474L} or Section 1 of Chapter 5.
\end{rem}

Given a bicategory $\bbb $, there are two main duals of $\bbb $ which give rise to four duals, including $\bbb $ itself. The first dual, denoted by $\bbb ^{\op} $,
comes from getting the dual of the underlying category of $\bbb $ (that is to say, the opposite w.r.t. $1$-cells), while the other dual, denoted by $\bbb ^\co $, is obtained from getting the duals of the hom-categories (that is to say, the opposite w.r.t. $2$-cells). Then, we have $\bbb $ itself and $\bbb ^{\coop}:= \left(\bbb ^{\op }\right) ^\co \cong \left(\bbb ^{\co }\right) ^\op $.

\begin{rem}\label{Trifunctores PhD}
We do not define tricategories~\cite{MR1261589}, but we give some independent remarks. For instance, as observed in Section 4.2, $2\textrm{-}\CAT $
is a tricategory and, specially in the present section, we consider the tricategory $\BICAT $ of bicategories, pseudofunctors, pseudonatural transformations and modifications as well. The dualizations mentioned above define invertible trifunctors:
$$ (-)^\op :  \BICAT  \cong  \BICAT ^\co , \quad (-)^\co : \BICAT \cong  \BICAT ^\tco , \quad  (-)^\coop : \BICAT \cong  \BICAT ^{\co\tco}, $$ 
$$ (-)^\op : 2\textrm{-}\CAT  \cong  2\textrm{-}\CAT ^\co , \quad (-)^\co : 2\textrm{-}\CAT \cong  2\textrm{-}\CAT ^\tco , \quad  (-)^\coop : 2\textrm{-}\CAT \cong  2\textrm{-}\CAT ^{\co\tco}, $$
in which $\mathfrak{T} ^\tco $ denotes the dual of the tricategory $\mathfrak{T}$ obtained from reversing the $3$-cells, and $\mathfrak{T} ^{\co\tco} := \left(\mathfrak{T} ^\tco\right) ^\co $. These isomorphisms are $2$-dimensional analogues of the invertible $2$-functor $(-) ^\op : \CAT\cong \CAT ^\co $.
\end{rem}

Given a bicategory $\bbb $, we can consider the bicategory of monads of $\bbb $. This was introduced in 
\cite{MR0347936, MR0299653} taking the point of \cite{MR0220789} that monads in $\bbb $ are given by lax functors between the terminal bicategory
$\mathsf{1} $ and $\bbb$. Herein, we introduce monads via lax algebras of the identity pseudomonad. This perspective takes advantage of the concepts 
introduced in Chapter 5~\cite{2016arXiv160703087L} and it gives a shortcut to understand the role of lifting of biadjoints in the formal theory of monads, which is sketched in Section \ref{Teoria Formal das Monadas e levantamentos PhD}.
It should be observed that our viewpoint also has connections with the approach of \cite{MR1957813}  to introduce $(T,V)$-categories.

We assume the definition of pseudomonads of Section 4 of Chapter 5 (or Definition 5.1 of Chapter 4~\cite{MR3491845} of pseudocomonads) and Definition 4.1 of Chapter 5~\cite{2016arXiv160703087L} of the bicategory of lax algebras and lax morphisms.  Within this context, it is easy to verify that:

\begin{lem}
The identity pseudofunctor $\Id _ {{}_{\bbb }}: \bbb\to\bbb $, with identities $2$-natural transformations and modifications, gives
a $2$-monad and a $2$-comonad on $\bbb $. 
\end{lem} 

\begin{defi}[Bicategory of Monads]
The bicategory of monads of a bicategory $\bbb$, denoted by $\mathsf{Mnd}(\bbb )$, is the bicategory of lax $\Id _ {{}_{\bbb }}$-algebras 
$\mathsf{Lax}\textrm{-}\Id _ {{}_{\bbb }}\textrm{-}\Alg _\ell  $.  
\end{defi}

Following the notation of lax algebras of Definition 4.1 of Chapter 5~\cite{2016arXiv160703087L}, we have that a monad in a $2$-category $\bbb$ is defined by a quadruplet $\mathtt{z}= (Z, \alg _ {{}_{\mathtt{z}}}, \overline{{\underline{\mathtt{z}}}}, \overline{{\underline{\mathtt{z}}}}_0 ) $ satisfying the condition of lax $\Id _ {{}_{\bbb }}$-algebra,
in which $Z$ is an object of $\bbb $,  $\alg _ {{}_{\mathtt{z}}}: Z\to Z $ is an endomorphism
of $\bbb $,  $\overline{{\underline{\mathtt{z}}}}: \alg _ {{}_{\mathtt{z}}}\alg _ {{}_{\mathtt{z}}}\Rightarrow \alg _ {{}_{\mathtt{z}}} $ 
is a $2$-cell of $\bbb$, called the multiplication of $\mathtt{z} $, 
and $\overline{{\underline{\mathtt{z}}}}_0 : \id _ {{}_{Z}}\Rightarrow  \alg _ {{}_{\mathtt{z}}} $ is a $2$-cell of $\bbb $, called the unit of the monad 
$\mathtt{z}$. In this case, we say that $\mathtt{z}$ is a monad on $Z$.

In order to introduce enriched categories and internal categories as monads, we define bicategories constructed from a suitable 
categorical structure $V$. These are the bicategory of matrices and the bicategory of spans, denoted respectively by  $V$-$\mathrm{Mat}$  and $\textit{Span}(V)$.

\begin{rem}\label{Suspension Categorias Monoidais PhD Capitulo 1}
The bicategory of matrices $V$-$\mathrm{Mat}$ is constructed from a monoidal category $V$. 
A \textit{monoidal category} is a bicategory $(\bbb , \circ , \mathtt{I}, \mathfrak{a}, \mathfrak{e}, \mathfrak{d} ) $ which has only one object $\triangle $. The composition is called, in this case, the \textit{monoidal product/tensor}. 
The \textit{underlying category} of a monoidal category is the hom-category $ \bbb (\triangle , \triangle ) $. The objects and morphisms of a monoidal
category are the objects and morphisms of its underlying category. 

Since we are talking about a bicategory with only one object $\triangle $, we actually have
only one natural isomorphism of associativity, one identity and two natural isomorphisms of action of identity. That is to say, given a monoidal
category $(\bbb , \circ , \mathtt{I}, \mathfrak{a}, \mathfrak{e}, \mathfrak{d} ) $, we denote 
 $\mathfrak{a}:=\mathfrak{a} _ {\triangle \triangle \triangle } $, $\mathfrak{d}:=\mathfrak{d} _ {\triangle  \triangle }$, $\mathfrak{e}:=\mathfrak{e} _ {\triangle  \triangle }$
and $I:= \mathtt{I} _ \triangle $. Therefore such a monoidal category  is given by a sextuple $(V, \otimes , I, \mathfrak{a}, \mathfrak{e}, \mathfrak{d} ) $ in which
$ V:= \bbb (\triangle , \triangle ) $ is the underlying category, $\otimes $ is the monoidal product, $I$ is the identity and $$\mathfrak{a}: (-\otimes -)\otimes -\longrightarrow -\otimes (- \otimes - ),\, \mathfrak{e}: (I\otimes -)\longrightarrow \id _{{}_{V}} ,\, \mathfrak{d}: (-\otimes I)\longrightarrow \id _ {{}_{V}} $$ are the respective natural isomorphisms satisfying the axioms of a bicategory with the only object $\triangle $.

For simplicity, letting the natural isomorphisms implicit, a monoidal category is usually denoted by 
$(V , \otimes , I) $, while $(\otimes , I ) $ together with the natural isomorphisms $\mathfrak{a}$, $\mathfrak{e}$, $\mathfrak{d}$  is called
a monoidal structure for $V$. When the monoidal structure is implicit in the context, we denote the monoidal category $(V , \otimes , I) $ by $V$.

A \textit{symmetric monoidal category} is a monoidal category $(V, \otimes , I ) $ endowed with a natural isomorphism, called braiding, $\mathfrak{b}  : -\otimes - \to -\otimes ^\op - $ in which $\otimes ^\op $ is the composition of the opposite of the bicategory that corresponds to $(V, \otimes , I ) $,
such that
$$
  \xymatrix@=5.5em{
   (A \otimes B) \otimes C 
	\ar[d]|-{\mathfrak{b} _ {(A,B)} \otimes \id _ {{}_{C}} }
   \ar[r]|-{\mathfrak{a}_{(A,B,C)}}&
   A \otimes (B \otimes C)
   \ar[r]|-{\mathfrak{b} _ {(A,B \otimes C)}  }  &
   (B \otimes C) \otimes A
	\ar[d]|-{\mathfrak{a}_{(B,C,A)}}
   \\
   (B \otimes A ) \otimes C
   \ar[r]|-{\mathfrak{a}_{(B,A,C)}}
	&
   B \otimes (A \otimes C)\ar[r]|-{\id _ {{}_{B}} \otimes \mathfrak{b}_{A,C}}
	&
   B \otimes (C \otimes A)
  }
$$
$$\xymatrix@=4.5em{
(A\otimes B)
\ar[r]|-{\mathfrak{b} _ {(A, B)} }
&
( B\otimes A )
\ar[ld]|-{\mathfrak{b} _ {( B, A)} }
\\
A\otimes B\ar@{=}[u]
&
}$$
commute for every triple $(A,B,C) $ of objects of $V$. A \textit{symmetric monoidal closed category} $(V, \otimes , I ) $ is a symmetric monoidal category such that every object of $V$ is exponentiable w.r.t. the monoidal product $\otimes $.
In other words, this means that the representable functor $A\otimes - : V\to V $ has a right adjoint for every $A$ of $V$.

If a category $V$ has finite products, it has a natural symmetric monoidal structure called \textit{cartesian monoidal structure}. That is to say $\otimes := \times $
and the unit is given by the terminal object $1 $ of $V$. The natural isomorphisms for associativity, actions of the identity and braiding are given by the
universal property of the product. A category with finite products endowed with this monoidal structure is called a \textit{monoidal cartesian category} or just
a \textit{cartesian category}, and it is denoted by $(V, \times , 1 ) $.

It would be appropriate
to define monoidal categories via $2$-dimensional monad theory, as, for instance, it is shown in Remark 5.4.3~\cite{2016arXiv160703087L}.
But our interest herein is mostly restricted to the case of cartesian closed categories.  Even our result on effective descent morphisms of enriched categories, which is 
Theorem 9.11 of Chapter 3~\cite{2016arXiv160604999L}, is given within the context of cartesian closed categories.  For this reason, we avoid giving further comments on general aspects
of monoidal categories. We refer to \cite{MR1712872, MR3307165, MR2094071} for the basics on such general aspects.
\end{rem}

\begin{defi}[Bicategory of Matrices]\label{Definicao Matrizes PhD}
Let $(V, \otimes , I)$ be a symmetric monoidal closed category with finite  coproducts. 
We define  $V\textrm{-}\mathrm{Mat}$ as follows:
\begin{itemize} \renewcommand\labelitemi{--}
\item The objects are the sets;
\item A morphism $M: A\to B$ in $V\textrm{-}\mathrm{Mat}$ is a matrix of objects in $V$, that is to say, 
a functor 
$A\times B  \to V$,
considering $A, B$  as discrete categories;
\item The $2$-cells are natural transformations. In other words, the category of morphisms for a ordered pair $(A,B) $ of sets
is the category of functors and natural transformations $\CAT\left[ A\times B, \Set \right] $;
\item The composition is given by the usual formula of product of matrices. More precisely, given matrices
$ M: A\times B\to V$ and $ N: B\times C\to V$,
the composition is defined by
\begin{eqnarray*}
N\circ M : & A\times C & \to V\\
           & (i,j) & \mapsto  \displaystyle\sum _{k\in B} M(i,k)\otimes N(k, j)
\end{eqnarray*}
in which $\sum $ denotes the coproduct;
\item For each set $A$, the identity on $A$ is the matrix
\begin{eqnarray*}
\id _{{}_{A}} : & A\times A & \to V\\
           & (i,j) & \mapsto  \begin{cases}
    I, & \text{if } i=j\\
    0,              & \text{otherwise}
\end{cases}
\end{eqnarray*}
in which $0$ is the initial object;
\item The natural isomorphisms for associativity and actions of identities are given by the universal property of coproducts, the isomorphisms of the preservation of the coproducts by $\otimes $ and the isomorphisms of the monoidal
structure.
\end{itemize}
\end{defi}

In order to define the bicategory of spans $\textit{Span}(V)$ of a category with pullbacks $V$, we denote by
$\textit{span} $ the category with $3$ objects ($\mathrm{0}$, $\mathrm{1}$ and $\mathrm{2}$) whose nontrivial morphisms are given by
$$\xymatrix{
\mathrm{0}
&
\mathrm{2}\ar[l]_-{d^1}\ar[r]^-{d^0}
&
\mathrm{1}.
} $$

\begin{defi}[Bicategory of Spans]\label{Definicao Spans PhD}
Let $V$ be a category with pullbacks. The bicategory $\textit{Span}(V)$ is defined by
\begin{itemize} \renewcommand\labelitemi{--}
\item The objects are the objects of $V$;
\item A morphism $M: A\to B$ in $\textit{Span}(V)$ is a span in $V$ between $A$ and $B$, that is to say, 
a functor 
$M: \textit{span}  \to V$,
such that $M(\mathrm{0} ) = A $ and $M(\mathrm{1} ) = B $;
\item A $2$-cell $f $ between two morphisms $M, K : A\to B $ is a natural transformation $f : M\longrightarrow K$
such that $f_{{}_{\mathrm{0}}} = \id _ {{}_{A}} $ and $f_{{}_{\mathrm{1}}} = \id _ {{}_{B}} $. That is to say,
it is a morphism $f: M(\mathrm{2} )\to K(\mathrm{2} ) $ such that
$$\xymatrix@=3.5em{
&
M(\mathrm{2} )
\ar[ld]|-{M(d^1)}\ar[rd]|-{M(d^0)}
\ar[dd]|-{f}
&
\\
A
&
&
B
\\
&
K(\mathrm{2} )
\ar[lu]|-{K(d^1)}\ar[ru]|-{K(d^0)}
&
}$$
commutes in $V$.
\item The composition is given by the pullback. More precisely, given a span
$ M: \textit{span} \to V$ between $A$ and $B$ and a span $ N: \textit{span}\to V$ between $B$ and $C$,
the composition is defined by the span
$$\xymatrix@=3em{
&
&
\text{P}(M, N )\ar[ld]\ar[rd]
&
&
\\
&
M(\mathrm{2})\ar[ld]|-{M(d^1)}\ar[rd]|-{M(d^0)}
&
&
N(\mathrm{2})\ar[ld]|-{N(d^1)}\ar[rd]|-{N(d^0)}
&
\\
A
&
&
B
&
&
C
}
$$
between the objects $A$ and $C$, in which $\text{P}(M, N )$ is the pullback $M(\mathrm{2} )\times _ {(M(d^0 ),\,  N(d ^1 )) } N(\mathrm{2} )$  of $M(d^0 )$ along $N(d ^1 )$ and the unlabeled arrows are the morphisms induced by the pullback;
\item For each set $A$, the identity on $A$ is the span $\textit{span} \to V$ constantly equal to $A$, taking the morphisms of the domain to the identity on $A$;
\item The natural isomorphisms for associativity and actions of identities are given by the universal property of pullbacks.
\end{itemize}

\end{defi}

Assuming that $V$ is a symmetric monoidal closed category with coproducts, a (small) category enriched in $V$ is a monad of the bicategory $V\textrm{-}\mathrm{Mat}$, while, if $V'$ has pullbacks, an internal category of $V' $ is a monad of $\textit{Span}(V')$. But $1$-cells and $2$-cells of  $\mathsf{Mnd}(V\textrm{-}\mathrm{Mat} )$
and $\mathsf{Mnd}(\textit{Span}(V'))$ do not respectively coincide with what should be $1$-cells and $2$-cells of the
$2$-category of enriched categories and the $2$-category of internal categories. In order to get the appropriate notion of 
$1$-cells, firstly we need to consider 
co-morphisms of monads and, secondly, we need to consider proarrow equipments. 

We introduce co-morphisms between monads via colax morphisms between lax algebras. For short, we use the
notation introduced in Definition 4.1 of Chapter 5~\cite{2016arXiv160703087L} and we define the $2$-category of lax algebras and colax morphisms w.r.t. a pseudomonad on a $2$-category, assuming, then, the minor adaptations to $\BICAT$  when necessary.

\begin{defi}[Colax morphisms of lax algebras]\label{LAXALGEBRASCOLAX}
Let $\TTTTT = (\TTTTT,  m  , \eta , \mu,   \iota, \tau )$ be a pseudomonad on a $2$-category $\bbb $. We define the $2$-category
of lax $\TTTTT$-algebras and colax $\TTTTT$-morphisms, denoted by $\mathsf{Lax}\textrm{-}\TTTTT\textrm{-}\Alg _{\mathsf{c}\ell } $,
as follows:
\begin{enumerate}
\item Objects: \textit{lax $\TTTTT$-algebras} as in Definition .5.4.1~\cite{2016arXiv160703087L}.

\item Morphisms: \textit{colax $\TTTTT$-morphisms} $\mathtt{f}:\mathtt{y}\to\mathtt{z} $ between lax $\TTTTT$-algebras $$\mathtt{y}=(Y, \alg _ {{}_{\mathtt{y}}}, \overline{{\underline{\mathtt{y}}}}, \overline{{\underline{\mathtt{y}}}}_0 ), \mathtt{z}= (Z, \alg _ {{}_{\mathtt{z}}}, \overline{{\underline{\mathtt{z}}}}, \overline{{\underline{\mathtt{z}}}}_0 )$$ are pairs $\mathtt{f} = (f, \left\langle \overline{\mathsf{f}}\right\rangle ) $ in which 
$f: Y\to Z  $
is a morphism in $\bbb $ and
$$\left\langle   \overline{\mathsf{f}}\right\rangle :  f\alg _ {{}_{\mathtt{y}}}\Rightarrow\alg _ {{}_{\mathtt{z}}}\TTTTT (f)    $$  
is a $2$-cell of $\bbb $ such that, defining $\widehat{\TTTTT (\left\langle   \overline{\mathsf{f}}\right\rangle )} : = \tttt ^{-1} _ {{}_{(\alg _ {{}_{\mathtt{z}}})(\TTTTT(f))  }}  \TTTTT (\left\langle   \overline{\mathsf{f}}\right\rangle ) \tttt _ {{}_{(f)(\alg _ {{}_{\mathtt{y}}}) }}
$, the equations
$$\xymatrix@R=3.5em{
&
\TTTTT ^2 Y\ar@{}[d]|-{\xRightarrow{\widehat{\TTTTT (\left\langle   \overline{\mathsf{f}}\right\rangle)}             } }
\ar[ld]_ -{\TTTTT (\alg _ {{}_{\mathtt{y}}}) }\ar[r]^-{\TTTTT ^2(f) }
&
\TTTTT ^2 Z\ar[rd]^-{m _ {{}_{Z}} }\ar[ld]|-{\TTTTT (\alg _ {{}_{\mathtt{z}}}) }\ar@{}[dd]|-{\xRightarrow{\overline{{\underline{\mathtt{z}}} }} }
&
\ar@{}[rdd]|-{=}
&
&
\TTTTT Z\ar[r]^{\alg _ {{}_{\mathtt{z}}}}\ar@{}[dd]|-{\xLeftarrow{ m_{{}_{f}}  }}
&
Z\ar@{}[d]|-{\xLeftarrow{\left\langle\overline{\mathsf{f}}\right\rangle } }
&
\\
\TTTTT Y\ar[r]_{\TTTTT (f) }\ar[rd]_{\alg _ {{}_{\mathtt{y}}}}
&
\TTTTT Z\ar[rd]|-{\alg _ {{}_{\mathtt{z}}} }\ar@{}[d]|-{\xRightarrow{\left\langle   \overline{\mathsf{f}}\right\rangle} }
&
&
\TTTTT Z\ar[ld]^{\alg _ {{}_{\mathtt{z}}} }
&
\TTTTT ^2 Z\ar[ru]^-{m _ {{}_{Z}} }
&
&
\TTTTT Y\ar[lu]|-{\TTTTT (f) }\ar[r]^{\alg _ {{}_{\mathtt{y}}} }\ar@{}[d]|-{\xLeftarrow{\overline{{\underline{\mathtt{y}}} } } }
&
Y\ar[lu]_ {f}
\\
&
Y\ar[r]_{f}
&
Z
&
&
&
\TTTTT ^2Y\ar[ru]|-{m _ {{}_{Y}} }\ar[r]_{\TTTTT (\alg _ {{}_{\mathtt{y}}}) }\ar[lu]^{\TTTTT ^2(f)}
&
\TTTTT Y\ar[ru]_-{\alg _ {{}_{\mathtt{y}}} }
&
}$$
$$\xymatrix@R=3.2em{ 
Z\ar[d]_{\eta _ {{}_{Z}} }\ar@{}[rd]|-{\xLeftarrow{\eta _ {{}_{f}}} }
&
&
Y\ar[ld]|-{\eta _ {{}_{Y}} }\ar@{=}[dd]\ar@{}[rdd]|-{=}\ar[ll]_-{f}
&
&
Z\ar[ld]_{\eta _ {{}_{Z}} }\ar@{=}[dd]
&
\\
\TTTTT Z\ar[d]_-{\alg _ {{}_{\mathtt{z}}} }\ar@{}[rd]|-{\xLeftarrow{\left\langle   \overline{\mathsf{f}}\right\rangle}}
&
\TTTTT Y\ar@{}[r]|-{\xLeftarrow{\overline{{\underline{\mathtt{y}}}}_0} }\ar[rd]|-{\alg _ {{}_{\mathtt{y}}} }\ar[l]^-{\TTTTT (f)}
&
&
\TTTTT Z\ar@{}[r]|-{\xLeftarrow{\overline{{\underline{\mathtt{z}}}}_0 } }\ar[rd]_-{\alg _ {{}_{\mathtt{z}}}}
&
&
\\
Z
&
&
Y\ar[ll]^-{f}
&
&
Z
&
Y\ar[l]_{f}
}$$
hold. The composition of colax $\TTTTT $-morphisms $\mathtt{f}:\mathtt{y}\to\mathtt{z} $ and $\mathtt{g}:\mathtt{x}\to\mathtt{y} $ of lax $\TTTTT $-algebras is denoted by
$\mathtt{f}\mathtt{g} $ and it is defined by the pair $(fg,  \left\langle \overline{\mathsf{f}\mathsf{g}}\right\rangle ) $ in which $\left\langle \overline{\mathsf{f}\mathsf{g}}\right\rangle $ is the $2$-cell defined by
$$\xymatrix@=3.5em{ 
&X\ar[d]_-{g}\ar@{}[rd]|-{\xRightarrow{\left\langle \overline{\mathsf{g}}\right\rangle} }
&
\TTTTT X\ar[l]_-{\alg _ {{}_{\mathtt{x}}}}\ar[d]|-{\TTTTT (g) }\ar@/^4pc/[dd]|-{\TTTTT (fg)}\ar@{{}{ }{}}@/^2pc/[dd]|-{\xRightarrow{\tttt  _ {{}_{(f)(g)  }}} }
\\
\left\langle \overline{\mathsf{f}\mathsf{g}}\right\rangle\ar@{}[r]|-{:=}
&
Y\ar[d]_-{f}\ar@{}[rd]|-{\xRightarrow{\left\langle \overline{\mathsf{f}}\right\rangle} }
&
\TTTTT Y\ar[l]|-{\alg _ {{}_{\mathtt{y}}} }\ar[d]|-{\TTTTT (f) }
\\
&Z
&
\TTTTT Z \ar[l]^-{\alg _ {{}_{\mathtt{z}}}}
} $$
\item $2$-cells: a \textit{$\TTTTT $-transformation}  $\mathfrak{m} : \mathtt{f}\Rightarrow\mathtt{h} $ between lax $\TTTTT $-morphisms $\mathtt{f} = (f,\left\langle   \overline{\mathsf{f}}\right\rangle )$, $\mathtt{h} = (h, \left\langle\overline{\mathsf{h}}\right\rangle )$ is a $2$-cell $\mathfrak{m} : f\Rightarrow h $ in $\bbb $ such that  the equation below holds.
$$\xymatrix@=2.5em{  \TTTTT Y\ar@/_4ex/[dd]_{\TTTTT (h) }
                    \ar@{}[dd]|{\xLeftarrow{\TTTTT(\mathfrak{m}) } }
                    \ar@/^4ex/[dd]^{\TTTTT (f) }
										\ar[rr]^{\alg _ {{}_{\mathtt{y}}} } && 
										 Y\ar[dd]^{f }    
&&
\TTTTT Y \ar[rr]^{ \alg _ {{}_{\mathtt{y}}} }\ar[dd]_-{\TTTTT (h) }  &&
Y\ar@/_5ex/[dd]_{h}
                    \ar@{}[dd]|{\xLeftarrow{\mathfrak{m}  } }
                    \ar@/^5ex/[dd]^{f }
\\
&\ar@{}[r]|{\xLeftarrow{\hskip .2em \left\langle\overline{\mathsf{f}}\right\rangle  \hskip .2em } }   &
 &=& 
&\ar@{}[l]|{\xLeftarrow{\hskip .2em \left\langle \overline{\mathsf{h}}\right\rangle  \hskip .2em } }  & 
\\
 \TTTTT Z\ar[rr]_ {\alg _ {{}_{\mathtt{z}}} } &&  Z
 &&
\TTTTT Z\ar[rr]_ {\alg _ {{}_{\mathtt{z}}} } && Z	 }$$ 
The compositions of $\TTTTT $-transformations are defined in the obvious way and these definitions make $\mathsf{Lax}\textrm{-}\TTTTT\textrm{-}\Alg _ {\mathsf{c}\ell } $ a $2$-category.

\end{enumerate}

\end{defi}

\begin{rem}[Identity on a lax $\TTTTT$-algebra]\label{Identidade de Lax Algebra PhD}
The identities on a lax $\TTTTT $-algebra $\mathtt{y}= (Y, \alg _ {{}_{\mathtt{y}}}, \overline{{\underline{\mathtt{y}}}}, \overline{{\underline{\mathtt{y}}}}_0 )$ 
in $\mathsf{Lax}\textrm{-}\TTTTT\textrm{-}\Alg _{\mathsf{c}\ell } $ and in $\mathsf{Lax}\textrm{-}\TTTTT\textrm{-}\Alg _{\ell } $ are respectively given by
\begin{equation*}
\left(\id _ {{}_{Y}},\, \vcenter{
\xymatrix@=6em{
\TTTTT Y
\ar@/_1.5pc/[d]|-{\TTTTT (\id _ {{}_{Y}}) }
\ar@/^1.5pc/[d]|-{\id _ {{}_{\TTTTT Y }} }
\ar@{}[d]|-{\xLeftarrow{\mathfrak{t}_{{}_{Y}}  }  }
\ar@{{}{ }{}}[dr]|-{=}
\ar[r]|-{  \alg _ {{}_{\mathtt{y}}}  }
&
Y
\ar[d]|-{\id _ {{}_{Y}} }
\\
\TTTTT Y
\ar[r]|-{  \alg _ {{}_{\mathtt{y}}}   }
&
Y
}}\right)
\quad\mbox{and}\quad
\left(\id _ {{}_{Y}},\, \vcenter{
\xymatrix@=6em{
\TTTTT Y
\ar@/_1.5pc/[d]|-{\TTTTT (\id _ {{}_{Y}}) }
\ar@/^1.5pc/[d]|-{\id _ {{}_{\TTTTT Y }} }
\ar@{}[d]|-{\xRightarrow{\mathfrak{t}_{{}_{Y}}^{-1}  }  }
\ar@{{}{ }{}}[dr]|-{=}
\ar[r]|-{  \alg _ {{}_{\mathtt{y}}}  }
&
Y
\ar[d]|-{\id _ {{}_{Y}} }
\\
\TTTTT Y
\ar[r]|-{  \alg _ {{}_{\mathtt{y}}}   }
&
Y
}
}\right).
\end{equation*}
\end{rem}

\begin{rem}[Colax algebras and coalgebras]
By dualizing, we get below the notions of colax algebras, lax coalgebras and colax coalgebras. The dualization $(-) ^\co $ preserves pseudomonads and pseudocomonads, while $(-)^\op $ takes pseudomonads to pseudocomonads and pseudocomonads to pseudomonads. In particular, given a pseudomonad $\TTTTT $ on a bicategory $\bbb $ and a  pseudocomonad $\SSSSS $ on a bicategory $\ccc $, we have that $\TTTTT ^\co $ is a pseudomonad on $\bbb ^\co $ and $\SSSSS ^\op , \SSSSS ^\coop $ are pseudomonads respectively on $\ccc ^\op $ and $\ccc ^\coop $. Therefore we can define: 
\begin{itemize} \renewcommand\labelitemi{--}
\item The bicategories of colax $\TTTTT $-algebras of the pseudomonad $\TTTTT $ (respectively with colax $\TTTTT $-morphisms and lax $\TTTTT $-morphisms):   
$$ \mathsf{Colax}\textrm{-}\TTTTT\textrm{-}\Alg _{\mathsf{c}\ell } := \left(\mathsf{Lax}\textrm{-}\TTTTT ^\co \textrm{-}\Alg _{\ell }\right) ^\co \quad
\mbox{and}\quad 
\mathsf{Colax}\textrm{-}\TTTTT\textrm{-}\Alg _{\ell } := \left(\mathsf{Lax}\textrm{-}\TTTTT ^\co \textrm{-}\Alg _{\mathsf{c}\ell }\right) ^\co; $$
\item The bicategories of lax $\SSSSS $-coalgebras:
$$ \mathsf{Lax}\textrm{-}\SSSSS \textrm{-}\CoAlg _{\ell } := \left(\mathsf{Lax}\textrm{-}\SSSSS ^\op \textrm{-}\Alg _{\ell }\right) ^\op \quad
\mbox{and}\quad  \mathsf{Lax}\textrm{-}\SSSSS \textrm{-}\CoAlg _{\mathsf{c}\ell } := \left(\mathsf{Lax}\textrm{-}\SSSSS ^\op \textrm{-}\Alg _{\mathsf{c}\ell }\right) ^\op;$$
\item The bicategory of colax $\SSSSS $-coalgebras:
$$ \mathsf{Colax}\textrm{-}\SSSSS\textrm{-}\CoAlg _{\mathsf{c}\ell } := \left(\mathsf{Lax}\textrm{-}\SSSSS ^\coop \textrm{-}\Alg _{\ell }\right) ^\coop
\quad
\mbox{and}\quad
 \mathsf{Colax}\textrm{-}\SSSSS\textrm{-}\CoAlg _{\ell } := \left(\mathsf{Lax}\textrm{-}\SSSSS ^\coop \textrm{-}\Alg _{\mathsf{c}\ell }\right) ^\coop .$$
\end{itemize}
\end{rem} 

\begin{rem}\label{Importante Mas Negligenciado PhD}
A colax $\TTTTT $-morphism $\mathtt{f} = (f,  \left\langle \overline{\mathsf{f}}\right\rangle ) $ is a $\TTTTT $-pseudomorphism if $\left\langle   \overline{\mathsf{f}}\right\rangle$ is an invertible $2$-cell. In particular, we have an inclusion $2$-functor  
$\mathsf{Lax}\textrm{-}\TTTTT \textrm{-}\Alg \to \mathsf{Lax}\textrm{-}\TTTTT  \textrm{-}\Alg _{\mathsf{c}\ell }$, in which
$\mathsf{Lax}\textrm{-}\TTTTT\textrm{-}\Alg$ denotes the $2$-category of lax $\TTTTT $-algebras and $\TTTTT $-pseudomorphisms as introduced in Section 4 of Chapter 5~\cite{2016arXiv160703087L}.
\end{rem}

\begin{defi}[Co-bicategory of monads]
The bicategory of monads and co-morphisms of a bicategory $\bbb $, called herein the co-bicategory of monads and denoted by $\mathsf{Mnd}_\co (\bbb )$, 
is the bicategory of lax $\Id _ {{}_{\bbb }}$-algebras and colax $\Id _ {{}_{\bbb }}$-morphisms, that is to say:
$\mathsf{Mnd}_\co (\bbb ):=\mathsf{Lax}\textrm{-}\Id _ {{}_{\bbb }}\textrm{-}\Alg _{\mathsf{c}\ell } $.
\end{defi}

The duals of the bicategories of monads are the bicategories of comonads. More precisely, the bicategories are defined by $\mathsf{CoMnd}(\bbb ):=
\left(\mathsf{Mnd} (\bbb ^\co)\right)^\co$ and $\mathsf{CoMnd}_\co (\bbb ):=
\left(\mathsf{Mnd}_\co (\bbb ^\co)\right)^\co$. 

\begin{lem}
Given a bicategory $\bbb $, 
$$\mathsf{Mnd}_\co (\bbb )\cong \mathsf{Lax}\textrm{-}\Id _ {{}_{\bbb }}\textrm{-}\CoAlg _\ell\cong \left(\mathsf{Mnd} (\bbb ^\op)\right)^\op
\quad\mbox{and}\quad \mathsf{Mnd} (\bbb )\cong \mathsf{Lax}\textrm{-}\Id _ {{}_{\bbb }}\textrm{-}\CoAlg _{\mathsf{c}\ell}. $$
\end{lem}

Herein, we actually do not need to give the full definition of proarrow equipment introduced in \cite{MR675339, MR794752}. Instead, we can give
a much less structured version:

\begin{defi}[Proarrow equipment]
A proarrow equipment on a $2$-category $\bbb _ 0 $ is a pseudofunctor $P: \bbb _ 0\to\bbb _1 $ which is the identity on objects and locally fully faithful. 
\end{defi}

Clearly, the category of proarrow equipments is a subcategory of the category of morphisms of the category of bicategories and pseudofunctors.
Similarly, in the context of Remark \ref{Trifunctores PhD},
 there is a tricategory of proarrow equipments which is a sub-tricategory of the tricategory of morphisms of $\BICAT $. Thereby, 
it is natural to consider pseudomonads on pseudofunctors and on proarrow equipments.

\begin{defi}[Pseudomonad on proarrow equipments]
A pseudomonad $\TTTTT$ on a pseudofunctor $P: \bbb _0\to\bbb _1  $ is a pair $(\TTTTT _ 0, \TTTTT _1 )$ in which 
$\TTTTT _ 0 = (\TTTTT ^0 ,  m ^0  , \eta ^0 , \mu ^0,   \iota ^0, \tau ^0 )$ is a pseudomonad on $\bbb _0$ and 
$\TTTTT _ 1 = (\TTTTT ^1 ,  m ^1  , \eta ^1  , \mu ^1,   \iota ^1, \tau ^1 )$ is a pseudomonad on $\bbb _1 $ such that this pair of pseudomonads agrees with $P$, which means that: 
$$\TTTTT ^1 P = P\TTTTT ^0, \quad m^1  P = P m^0, \quad \eta ^1  P = P \eta ^0 , \quad \mu ^1  P = P  \mu ^0, \quad \iota ^1  P = P \iota ^0,\quad\tau ^1  P = P \tau ^0. $$

\end{defi}

For our purposes, we could define a simpler version of pseudomonads on proarrow equipments on $2$-categories. 
That is to say, we could say that a pseudomonad on 
$P: \bbb _0 \to \bbb _1 $ is just a pseudomonad on $\bbb _1 $.

\begin{defi}\label{Definicao com Pullback PhD}
Given a pseudomonad $\TTTTT = (\TTTTT _ 0, \TTTTT _1) $ on a proarrow equipment $P : \bbb _ 0\to \bbb _1 $ on a $2$-category $\bbb _0 $, the bicategory of lax $(\TTTTT , P)$-algebras, denoted herein
by $\mathsf{Lax}\textrm{-}(\TTTTT , P) \textrm{-}\CoAlg _{\mathsf{c}\ell}$,
is the pullback of $P$ along the forgetful pseudofunctor $\mathsf{Lax}\textrm{-}\TTTTT _1 \textrm{-}\Alg _{\mathsf{c}\ell}\to \bbb _ 1 $ in the category of bicategories and pseudofunctors. 
\end{defi}

The category of bicategories and pseudofunctors does not have all pullbacks (or equalizers). However, in the context of Definition \ref{Definicao com Pullback PhD},
the pullbacks always exist. Moreover, $\mathsf{Lax}\textrm{-}(\TTTTT , P) \textrm{-}\CoAlg _{\mathsf{c}\ell}$ is always a $2$-category. Explicitly, the objects of 
$\mathsf{Lax}\textrm{-}(\TTTTT , P) \textrm{-}\CoAlg _{\mathsf{c}\ell}$ are lax $\TTTTT $-algebras $\mathtt{z}=  (Z, \alg _ {{}_{\mathtt{z}}}, \overline{{\underline{\mathtt{z}}}}, \overline{{\underline{\mathtt{z}}}}_0 )$ and the morphisms between two objects in $\mathsf{Lax}\textrm{-}(\TTTTT , P) \textrm{-}\CoAlg _{\mathsf{c}\ell}$ are colax $\TTTTT $-morphisms $\mathtt{g} = (g, \left\langle \overline{\mathsf{g}}\right\rangle ) $ between lax $\TTTTT$-algebras such that $P(\dot{g}) = g $ for some morphism $\dot{g}$ of $\bbb _ 0 $. The composition of morphisms $\mathtt{f} = (P(\dot{f}), \left\langle \overline{\mathsf{f}}\right\rangle ), \mathtt{g} = (P(\dot{g}), \left\langle \overline{\mathsf{g}}\right\rangle )  $ is defined by 
$\mathtt{f}\cdot \mathtt{g} := (P(\dot{f}\dot{g} ), \left\langle \underline{\overline{\mathsf{f}\mathsf{g}}}\right\rangle ) $ such that $\left\langle \underline{\overline{\mathsf{f}\mathsf{g}}}\right\rangle $ is defined by
$$\xymatrix@R=4em{ 
X\ar[d]_-{g = P(\dot{g}) }\ar@/_6pc/[dd]|-{ P (\dot{f}\dot{g})}\ar@{{}{ }{}}@/_3pc/[dd]|-{\xRightarrow{ \pppp ^{-1} _ {{}_{(\dot{f} )(\dot{g} )  }} } }
&
\TTTTT X\ar[dd]|-{\TTTTT (fg) }\ar[l]_-{\alg _ {{}_{\mathtt{x}}}}
\ar@/^6pc/[dd]|-{\TTTTT P (\dot{f}\dot{g})}\ar@{{}{ }{}}@/^3pc/[dd]|-{\xRightarrow{\TTTTT\left( \pppp  _ {{}_{(\dot{f} )(\dot{g} )  }} \right) } }
\\
Y\ar[d]_-{f = P(\dot{f})}
&
\ar@{}[l]|-{\xRightarrow{\left\langle \overline{\mathsf{f}\mathsf{g}}\right\rangle } }   
\\
Z
&
\TTTTT Z \ar[l]^-{\alg _ {{}_{\mathtt{z}}}}
} $$
in which
$\left\langle \overline{\mathsf{f}\mathsf{g}}\right\rangle $ denotes the $2$-cell component of the usual composition of the colax $\TTTTT $-morphisms of
$\mathtt{f} $ and $\mathtt{g}$.

The \textit{$2$-category of monads and co-morphisms in a proarrow equipment} 
$P : \bbb _ 0\to \bbb _1$ is defined, then, by 
$\mathsf{Mnd}  _ \co (P) := \mathsf{Lax}\textrm{-}(\Id _P , P) \textrm{-}\Alg _{\mathsf{c}\ell}$. 

\begin{defi}[Proarrows of Matrices]
Given a symmetric closed  monoidal category $(V, \otimes , I)$ with coproducts, the bicategory $V\textrm{-}\mathrm{Mat}$ gives a natural proarrow equipment 
$\Set\to V\textrm{-}\mathrm{Mat}$ on the locally discrete bicategory $\Set $ 
in which a function $f:A\to B $ is taken to $\breve{f}: A\times B \to V $ defined by
$$
\breve{f} (i,j) =
\begin{cases}
    I, & \text{if } f(i) = j\\
    0,              & \text{otherwise}
\end{cases}
$$
in which $0$ is the initial object. 
\end{defi}

\begin{defi}[Proarrows of Spans]
Given a category $V$ with pullbacks, the bicategory $\textit{Span}(V)$ gives a natural proarrow equipment 
$V\to \textit{Span}(V)$ on the locally discrete bicategory $V $ 
in which a morphism $f:A\to B $ is taken to the span
$$
\xymatrix{
A
&
A\ar[l]_-{\id _ {{}_{A}} }
\ar[r]^- {f}
&
B.
}
$$
\end{defi}

\begin{defi}[Enriched Categories and Internal Categories]
Given a symmetric closed  monoidal category $(V, \otimes , I)$ with coproducts, the category of small $V$-enriched categories $V$-$\cat $ is the underlying category of  $\mathsf{Mnd}  _ \co (\Set\to V\textrm{-}\mathrm{Mat}) $.

Given a category with pullbacks $V'$, the category of $V'$-internal categories $\Cat (V') $  is the underlying category of  $\mathsf{Mnd}  _ \co (V' \to\textit{Span}(V')) $.
\end{defi}

\begin{rem}
The introduction of enriched categories via monads of $V\textrm{-}\mathrm{Mat}$ does not work well for large enriched categories, unless we make 
tiresome considerations about enlargements of universes/completions. 
For our setting, however, it is enough to observe that, if $V$ is a large symmetric monoidal closed category that has large coproducts (indexed by discrete categories), one can consider the bicategory as in Definition 
\ref{Definicao Matrizes PhD} but with discrete categories (objects of $\SET $) instead of sets, and matrices indexed by them. By abuse of language, denoting this 
bicategory as $V\textrm{-}\mathrm{Mat}$, it is clear that we can consider the category of large $V$-enriched categories $V\textrm{-}\Cat $ defined by the
underlying (large) category $\mathsf{Mnd}  _ \co (\SET\to V\textrm{-}\mathrm{Mat}) $.
\end{rem}

Explicitly, an internal category of a category with pullbacks $V'$ is a span $A_1\xleftarrow{d ^1 } A _2\xrightarrow{d ^0 } A_1 $, which we denote by $\mathtt{a} $,
together
with the multiplication and identity, $2$-cells $\mathtt{a}\circ\mathtt{a}\Rightarrow \mathtt{a} $ and $\id _ A \Rightarrow \mathtt{a} $ 
of $\textit{Span}(V')$, satisfying the conditions of monad/lax $\Id _ {{}_\Cat} $-algebra of associativity and action of identity described in Definition 4.1 of Chapter 5~\cite{2016arXiv160703087L}. Recall that, by definition, the $2$-cells $\mathtt{a}\circ\mathtt{a}\Rightarrow \mathtt{a} $ and $\id _ A \Rightarrow \mathtt{a} $ are just morphisms $\overline{\underline{\mathtt{a}}} : A_2\times _{(d^0, d^1)} A_2\to A_1$ and $\overline{\underline{\mathtt{a}}}_0 : A_1\to A_2 $ of $V' $ such that
$$\xymatrix@C=1em@R=1.2em{
&
A_1
\ar@{=}[rdd]
&
&
&
A_1
&
&
\\
&
&
&
&
&
A_2
\ar[lu]|-{d^0}
&
\\
A_2
\ar[rdd]|-{d^1}
\ar[ruu]|-{d^0}
&
&
A_1
\ar[ll]|-{\overline{\underline{\mathtt{a}}} _0}
&
A_2
\ar[ruu]|-{d^0}
\ar[rdd]|-{d^1}
&
&
&
A_2\times _{(d^0, d^1 )} A_2
\ar@/_1pc/[lluu]|-{D^0}
\ar@/^1pc/[lldd]|-{D^1}
\ar[lll]|-{\overline{\underline{\mathtt{a}}} }
\ar[lu]
\ar[ld]
\\
&
&
&
&
&
A_2
\ar[ld]|-{d ^1}
&
\\
&
A_1
\ar@{=}[ruu]
&
&
&
A_1
&
&
}$$
commute. In this case, the object $A_1 $ is called the \textit{object of objects}, the object $A_2$ is 
called the \textit{object of morphisms}, $d^1 $ is the \textit{domain morphism}, $d^0 $ is the \text{codomain morphism}, the morphism $\overline{\underline{\mathtt{a}}}$ is the \textit{composition} and $\overline{\underline{\mathtt{a}}}_0 : A_1\to A_2 $ is the \textit{identity assigning}.

\begin{defi}[Double Category]
The category of double categories is the category of internal categories of $\Cat $, that is to say $\Cat (\Cat ) $.

A double category $\mathfrak{X} $ is, then, a span $\mathfrak{X}_1\xleftarrow{d ^1 } \mathfrak{X} _2\xrightarrow{d ^0 } \mathfrak{X}_1 $ of $\Cat $
with the composition and identity satisfying the usual conditions. Given such a double category, 
the objects of $\mathfrak{X}_1 $ are called the objects of the double category, while the morphisms of $\mathfrak{X}_1 $ are called vertical arrows.
The objects $f$ of $\mathfrak{X} _2$ are called horizontal arrows (or morphisms) of the double category $\mathfrak{X}$ and we denote it by
$f: d^1(f)\to d^0 (f) $ in which $d^1(f) $ is called the domain of $f$ and $d^0(f) $ is the codomain of $f$. Finally, if $\alpha : f\to g $ is a morphism
of $\mathfrak{X}_2 $, we denote it by
$$\xymatrix@=4em{
d^1 (f)
\ar@{}[rd]|-{\alpha}
\ar[r]|-{f}
\ar[d]|-{d^1(\alpha )}
&
d^0 (f)
\ar[d]|-{d^0 (\alpha )}
\\
d^1 (g)
\ar[r]|-{g }
&
d^0(g)
}
$$
and we say that $\alpha $ is a $2$-cell (or a \textit{square}) of $\mathfrak{X} $.
\end{defi}

A \textit{$2$-category} is just a $\Cat $-enriched category. Clearly, if the 
category of objects of a double category $\mathfrak{D}$ is discrete, then it is a $2$-category. More generally, the full subcategory of $\Cat (\Cat ) $ consisting of the double categories without nontrivial vertical arrows is isomorphic to the category of $2$-categories and $2$-functors $\Cat\textrm{-}\Cat $. 
For suitable categories $V$ instead of $\Cat $, a similar property holds. This generalization is given by Lemma 9.10 of Chapter 3~\cite{2016arXiv160604999L}.

\section{Freely generated n-categories, coinserters and presentations of low dimensional categories}\label{Capitulo2 PhD}

Chapter 2  corresponds to the paper 
\textit{Freely generated n-categories, coinserters 
and presentations of low dimensional categories}~\cite{2017arXiv170404474L},
\textit{DMUC 17-20} or \textit{arXiv:1704.04474}. As the title suggests, 
the main subjects of this paper
are related to development of a theory towards the study of presentations of low dimensional categories and freely generated categorical structures.
Although it was the last paper to be written, it introduces some basic aspects of $2$-category theory.

The chapter starts by giving basic aspects of $2$-dimensional (weighted) colimits, 
focusing on coinserters, coequifiers and coinverters. These are the $2$-dimensional
colimits that have direct applications respectively in the study of adding free morphisms to a category, forcing relations between morphisms and categories of fractions~\cite{MR1245795, MR998024}. We are, however, interested in particular cases: freely generated categories, the left adjoint of the inclusion of thin categories and the left
adjoint of the
inclusion of groupoids.

The category of thin categories $\Prd $ and the category of groupoids $\Gr $, as defined in 1.1 of Chapter 2, are
replete reflective subcategories of $\Cat$, with inclusions $M_1 : \Prd\to\Cat $ and $\U _1 : \Gr\to \Cat $. Hence there is an easy way of characterizing the
images of $M_1 $ and $\U _1 $ via universal properties. More precisely, if we denote the reflectors by $\M _1 : \Cat\to\Prd $ and $\L _1 : \Cat\to\Gr $, we have that, given a category $X$ of $\Cat $, there is $Y$ of $\Prd $ ($Y$ of $\Gr $) such that $M_1(Y)\cong X $ ($\U _1 (Y)\cong X$ ) if and only if  $M_1\M_1 (X)\cong X $
($\U_1\L_1 (X)\cong X $). Since $\U_1\L_1 (X)$ and $M_1\M_1 (X) $ are given, respectively, by a coinverter and a coequifier, we get Proposition 2.1.4 and Theorem 2.1.7.
Actually, within the context of Chapter 2, the most important fact in this direction is that we can get the category  freely generated by a graph $G: \G ^\op\to\Set $, denoted by $\F _1 (G) $, 
via the coinserter of $G$ composed with the inclusion $\Set\to\Cat $. This is Lemma  2.2.3. It motivates one of the main points of the paper: to give freely generated categorical structures via coinserters.

After showing these facts, we give further background in Section 2.2. We study basic aspects of freely generated categories. We introduce basic notions of 
graphs (trees, forests and connectedness) relating with the groupoids and categories freely generated by them. We study reflexive graphs as well. The main
importance of reflexive graphs within our context is the fact that the terminal category is freely generated by 
the terminal reflexive graph, 
while $\F _1 $ does not preserve the terminal object. Finally, we also characterize the totally ordered sets that are 
free categories and show that freeness of groupoids is a property preserved by equivalences.

Then, we give the basic notions of presentations of this paper. On one hand, we show how every monad $\TTTTT $ induces a natural notion of presentation of $\TTTTT$-algebras
in Section 3 of Chapter 2. In particular, we have that the free category monad induces a notion of presentation of categories. On the other hand, in Section 4 of Chapter 2, we define computads and show how it induces a notion of presentation of categories (and groupoids) with equations between morphisms. We compare both notions of presentations in Theorems 2.4.5 and 2.4.6.

Since we can see computads as free categories together with relations between morphisms, we introduce the suitable variation of the concept of computad
 to deal with presentation of groupoids:
groupoidal computads. This has particular interest in Section 5 of Chapter 2   which deals with the relation between topology and computads. Moreover, the notion
of presentations of groups via groupoidal computads coincides with the usual notion, as explained in Remark 4.19 of Chapter 2.

Section 5 of Chapter 2 (\cite{2017arXiv170404474L}) establishes fundamental connections between 
topology and computads. We start by showing that the usual association of a 
\textit{topological graph} to each graph, usually called \textit{geometric realization}, 
consists of a left adjoint functor $\F _ {\Top _1}: \grph\to \Top $ given objectwise by the 
\textit{topological coinserter} introduced therein. In this context, 
we show that there is a distributive law 
between the monads induced by $\F _ 1 $ and $\F _ {\Top _1} $ which is constructed from the usual notion of concatenation of continuous paths.
As a fundamental tool to study presentation of
thin and locally thin
categorical structures, using the distributive law mentioned above, 
we give a detailed construction of a $2$-dimensional analogue of   $\F _ {\Top _1} $, 
denoted by $\F _ {\Top _2}: \cmp\to \Top$, which associates each computad to a topological space. 

We introduce the fundamental groupoid functor $\Pi : \Top\to\gr $ using the concept of presentation via computads. 
More precisely, we firstly associate a computad to
each topological space and, then, $\Pi $ is given by the groupoid presented by the associated computad. We prove 
theorems that relate the fundamental groupoid 
and freely generated groupoids. In particular, the last results of Section 5 of Chapter 2 state that the fundamental groupoid of $\F _ {\Top _2}(\mathfrak{g}) $
is equivalent to the groupoid presented by $\mathfrak{g} $. 
These results show that $\F _ {\Top _2}$ formalizes the usual construction of a 
CW-complex from each
presentation of groups, the Lyndon–van Kampen diagrams~\cite{MR1506963}.

In Section 6 of Chapter 2~\cite{2017arXiv170404474L}, we introduce our main notions of deficiency. More precisely, we introduce the notion of deficiency of a groupoid w.r.t.
presentation of groupoids and the notion deficiency of a presentation of a $\TTTTT$-algebra w.r.t.
a finite measure. Under suitable hypotheses, we find both notions to coincide in the particular case of groupoids. They also coincide with the usual notions
of deficiency. In this section, mostly using the results of Section 5 of Chapter 2~\cite{2017arXiv170404474L}, we also develop a theory towards the study of thin categories and thin groupoids. For instance, we prove that, whenever $\mathfrak{g}$ is a computad such that $\F _ {\Top _2}(\mathfrak{g}) $ has Euler characteristic smaller than $1$, then  the groupoid presented by $\mathfrak{g}$ is not thin. From this fact, we can prove that deficiency of a thin groupoid is $0$, recasting and generalizing the result that says that trivial groups have deficiency $0$.

Although our definition of computads is equivalent to the original one of \cite{MR0401868}, 
we introduce it via a graph satisfying a coincidence property, as it is shown in Remark 8.12 of Chapter 2. The main point of our perspective, besides giving a concise recursive definition, is that it  
 allows us to prove that the $2$-category freely generated by a computad $\mathfrak{g} $ is the coinserter of $\mathfrak{g}$, when we consider $\mathfrak{g}$ as
a graph internal to an appropriate $2$-category of
$2$-categories: the $2$-category of $2$-categories, $2$-functors and icons~\cite{MR2640216}. In order to get freely generated $n$-categories via coinserters, we  introduce higher dimensional analogues of icons. These concepts and results, including the general result that states that the $n$-category freely
generated by an $n$-computad is its coinserter, are given in Sections 7 and 8 of Chapter 2~\cite{2017arXiv170404474L}.

We finish the paper studying aspects of presentations of $2$-categories. We show simple examples of 
locally thin $2$-categories that are not free and develop a theory to study presentations of locally thin and groupoidal 
$2$-categories. We  give efficient presentations of $2$-categories related to the \textit{strict descent object}.
In the end of Chapter 2~\cite{2017arXiv170404474L}, we sketch a construction of the $3$-dimensional analogue of $\F _ {\Top _2}$, that is to say $\F _ {\Top _3}$, which
associates a $3$-dimensional CW-complex to each $3$-computad.

\section{Beck-Chevalley}\label{Doctrinal Adjunction PhD}

The \textit{mate correspondence}~\cite{MR0357542, MR2781909} is a fundamental framework in $2$-dimensional category theory. For instance, this correspondence is 
in the core of the techniques of Chapter 2 to construct $\F _ {\Top _2}$ and the distributive law   between the
 monads induced by $\F _ 1 $ and $\F _ {\Top _1} $. Another important example is in Chapter 3: the Beck-Chevalley condition, written
in terms of a simple mate correspondence, 
plays an important role in the proof of the B\'{e}nabou-Robaud Theorem. 

The main aim of this section is to present the Beck-Chevalley condition within the context of $2$-dimensional monad theory. In order to do so, 
we present 
the most elementary version of \textit{mate correspondence} in Theorem \ref{Mate Correspondence Theorem PhD}. We start by defining and giving elementary results on adjunctions in a
$2$-category. 

\begin{defi}[Adjunction]\label{Adjuncao Capitulo 1 PhD}
An adjunction in a $2$-category $\aaa $ is a quadruplet $$(f: Y\to Z ,g: Z\to Y  , \epsilon : fg\Rightarrow \id _ {{}_{Z}}, \eta : \id _{{}_{Y}}\Rightarrow gf ) ,$$ in which 
$f, g $ are $1$-cells and $\epsilon , \eta $ are $2$-cells of $\aaa $, such that 
$$\xymatrix@=3em{
Y
\ar@{}[rd]|-{\xRightarrow{\enspace\eta\quad } }
\ar[rr]^-{f}
\ar@{=}[dd]
&
&
Z
\ar[lldd]|-{g}
\ar@{=}[dd]
&
Y
\ar@{=}[dd]
\ar@{}[rd]|-{\xRightarrow{\enspace\eta\quad } }
&
&
Z
\ar[ll]_-{g}
\ar@{=}[dd]
\\
&&&&&
\\
Y
\ar[rr]_-{f}
&
&
Z
\ar@{}[lu]|-{\xRightarrow{\enspace\epsilon\quad } }
&
Y
\ar[rruu]|-{f}
&
&
Z
\ar[ll]^-{g}
\ar@{}[lu]|-{\xRightarrow{\enspace\epsilon\quad } }
}$$
are respectively the identity $2$-cells $f\Rightarrow f $ and $g\Rightarrow g $. 
In this case, we denote the adjunction by 
$(f\dashv g , \epsilon , \eta ): Y\to Z  $.
For short, we also denote such an adjunction by just $f\dashv g $ when the counit and unit are implicit.
\end{defi} 

\begin{rem}
If $(f\dashv g , \epsilon , \eta ): Y\to Z  $ is an adjunction, $f$ is called left adjoint, $g $ is called right adjoint,
$\epsilon $ is called the counit and $\eta $ 
is called the unit of the adjunction. Moreover, the equations of Definition  \ref{Adjuncao Capitulo 1 PhD} are called triangle identities.
 
\end{rem}

\begin{rem}\label{Adjuntos unicos capitulo 1 PhD}
It is clear that adjoints are unique up to isomorphism. More precisely, if 
$(\widetilde{f}\dashv g , \mu  , \rho )$ and $(f\dashv g , \epsilon , \eta ): Y\to Z $ then
$$\xymatrix@=2.5em{
Y
\ar@{}[rd]|-{\xRightarrow{\enspace\eta\quad } }
\ar[rr]^-{f}
\ar@{=}[dd]
&
&
Z
\ar[lldd]|-{g}
\ar@{=}[dd]
&
&
Y
\ar[rr]^-{\widetilde{f}}
\ar@{=}[dd]
\ar@{}[rd]|-{\xRightarrow{\enspace\rho\quad } }
&
&
Z
\ar[lldd]|-{g }
\ar@{=}[dd]
\\
&&&\mbox{is the inverse of}&&&
\\
Y
\ar[rr]_-{\widetilde{f}}
&
&
Z
\ar@{}[lu]|-{\xRightarrow{\enspace\mu\quad } }
&
&
Y
\ar[rr]_-{f}
&
&
Z
\ar@{}[lu]|-{\xRightarrow{\enspace\epsilon\quad } }
}$$
by the triangle identities. In particular, $f\cong \widetilde{f} $.
\end{rem}

Let $\aaa $ be a $2$-category. We can construct a category of adjunctions $\aaa ^\mathsf{adj} $ of $\aaa $. The objects of $\aaa ^\mathsf{adj} $ are the objects of $\aaa $, but the morphisms are adjunctions $(f\dashv g , \epsilon , \eta ): Y\to Z  $. The identities are the adjunctions between identities with identities counit and unit. Given adjunctions $(f_2\dashv g_2 , \epsilon _ 2 , \eta _2 ): Y\to Z  $
and $(f_1\dashv g_1 , \epsilon _1 , \eta _1 ): X\to Y  $, the composition is given by $(f_2f_1\dashv g_1 g_2 , \epsilon _3 , \eta _3 ): X\to Z  $
in which $\epsilon _3$ and $ \eta _3 $ are defined below. 
$$\xymatrix@=0.9em{
&
X
\ar@{=}[dddd]
\ar[rrd]^-{f_1}
&&
&&
&
&&
Z
\ar@{=}[dddd]
&&
&&
\\
&
&&
Y
\ar[rrd]^-{f_2}
\ar[dd]|-{\id _ {{}_{Y}} }
&&
&
&&
&&
Y
\ar[llu]_-{f_2}
&&
\\
\eta _3 \ar@{}[r]|-{:=}
&
\ar@{}[rr]|-{\xRightarrow{\enspace\eta _1\enspace } }
&&
&&
Z
\ar@{}[ll]|-{\xRightarrow{\enspace\eta _2\enspace } }
\ar[lld]^-{g_2}
&&
\epsilon _3 \ar@{}[r]|-{:=}
&
\ar@{}[rr]|-{\xLeftarrow{\enspace\epsilon _2\enspace } }
&&
&&
X
\ar[llu]_-{f_1}
\ar@{}[ll]|-{\xLeftarrow{\enspace\epsilon _1 \enspace} }
\\
&
&&
Y
\ar[lld]^-{g_1}
&&
&
&&
&&
Y
\ar[uu]|-{\id _ {{}_{Y}} }
\ar[rru]_-{g_1}
&&
\\
&
X
&&
&&
&&
&
Z
\ar[rru]_-{g_2}
&&
&&
}
$$

Of course, $2$-functors take adjunctions to adjunctions. However, pseudofunctors do not.
Instead, in this case, we can say that left (or right) adjoints are taken to left (or right) adjoints.
More precisely, we have Lemma \ref{induced adjunction PhD}.  

In order to prove such result, we use Lemma \ref{Fundamental lema PhD}, a basic result on the image of pasting of $2$-cells. 
Following the notation established in Definition 4.2.1~\cite{MR3491845}, we have:

\begin{lem}\label{Fundamental lema PhD}
If $L:\aaa\to\bbb $ is a pseudofunctor, then:
\begin{equation*}\label{Eq Fundamental De Pasting PhD}
   L \left(
    \vcenter{
    \xymatrix{
W
\ar@{}[rd]|-{\xRightarrow{\enspace\alpha\quad } }
\ar[rr]|-{g}
\ar[dd]|-{m}
&
&
X
\ar[lldd]|-{h}
\ar[dd]|-{n}
\\
&&
\\
Z
\ar[rr]|-{f}
&
&
Y
\ar@{}[lu]|-{\xRightarrow{\enspace\beta\quad } }
}}\right)    
        =         
            \mathfrak{l}_{{}_{ng} }\cdot \left(\vcenter{
    \xymatrix{
L(W)
\ar@{}[rd]|-{\xRightarrow{ \mathfrak{l}_{{}_{hg}}^ {-1}\cdot L(\alpha )  } }
\ar[rr]|-{L(g)}
\ar[dd]|-{L(m)}
&
&
L(X)
\ar[lldd]|-{L(h)}
\ar[dd]|-{L(n)}
\\
&&
\\
L(Z)
\ar[rr]|-{L(f)}
&
&
L(Y)
\ar@{}[lu]|-{\xRightarrow{L(\beta )\cdot \mathfrak{l}_{{}_{fh}} } }
}}
\right)\cdot \mathfrak{l}_{{}_{fm} }^{-1},
\end{equation*}

\begin{equation*}\label{Eq Fundamental De Pasting PhD 2}
   L \left(
    \vcenter{
    \xymatrix{
Z
\ar@{}[rd]|-{\xRightarrow{\enspace\alpha\quad } }
&
&
X
\ar[ll]|-{h}
\\
&&
\\
W\ar[rruu]|-{g}\ar[uu]|-{m}
&
&
Y\ar[uu]|-{t}
\ar[ll]|-{u}
\ar@{}[lu]|-{\xRightarrow{\enspace\gamma\quad } }
}}\right)    
        =         
            \mathfrak{l}_{{}_{ht} }\cdot \left(\vcenter{
    \xymatrix{
L(Z)
\ar@{}[rd]|-{\xRightarrow{   \mathfrak{l}_{{}_{hg}}^{-1} \cdot L(\alpha )  } }
&
&
L(X)
\ar[ll]|-{L(h)}
\\
&&
\\
L(W)\ar[uu]|-{L(m)}
\ar[rruu]|-{L(g)}
&
&
L(Y)\ar[uu]|-{L(t)}
\ar[ll]|-{L(u)}
\ar@{}[lu]|-{\xRightarrow{L(\gamma )\cdot \mathfrak{l}_{{}_{gu}} } }
}}
\right)\cdot \mathfrak{l}_{{}_{mu} }^{-1}.
\end{equation*}

\end{lem} 
\begin{proof}

Firstly, by the interchange law, it is clear that the right side of the first equation
above is equal to 
$$\mathfrak{l}_{{}_{ng}}\cdot \left( L(\beta )\ast L(\id _ {{}_{g}})\right)\cdot 
\left( \mathfrak{l}_{{}_{fh}}\ast\id _ {{}_{L(g)}}\right)\cdot
\left( \id _ {{}_{L(f)}}\ast \mathfrak{l}_{{}_{hg}}^{-1}\right)
\cdot \left( L(\id _ {{}_{f}})\ast L(\alpha ) \right) \cdot \mathfrak{l} _ {{}_{fm}}^{-1} .$$ 
Then, by the naturality of Definition 4.2.1~\cite{MR3491845}, 
this is equal to	
$$ L(\beta \ast \id _ {{}_{g}})\cdot \mathfrak{l}_{{}_{(fh)g}}\cdot 
\left( \mathfrak{l}_{{}_{fh}}\ast\id _ {{}_{L(g)}}\right)\cdot
\left( \id _ {{}_{L(f)}}\ast \mathfrak{l}_{{}_{hg}}^{-1}\right)
\cdot \mathfrak{l} _ {{}_{f(hg)}}^{-1} \cdot L(\id _ {{}_{f}}\ast\alpha ),$$
which is indeed equal to the left side of the equation above, since, by 
the associativity of Definition 4.2.1~\cite{MR3491845},
\begin{eqnarray*}
\mathfrak{l}_{{}_{(fh)g}}\cdot 
\left( \mathfrak{l}_{{}_{fh}}\ast\id _ {{}_{L(g)}}\right)\cdot
\left( \id _ {{}_{L(f)}}\ast \mathfrak{l}_{{}_{hg}}^{-1}\right)
\cdot \mathfrak{l} _ {{}_{f(hg)}}^{-1} &=&
\mathfrak{l}_{{}_{f(hg)}}\cdot 
\left( \id _ {{}_{L(f)}}\ast \mathfrak{l}_{{}_{hg}}\right)\cdot
\left( \id _ {{}_{L(f)}}\ast \mathfrak{l}_{{}_{hg}}\right)^{-1}
\cdot \mathfrak{l} _ {{}_{f(hg)}}^{-1}
\\
&=&
\id _ {{}_{L(fhg)}}
\end{eqnarray*}
\normalsize
and $L (\beta \ast \id _ {{}_{g}})\cdot L(\id _ {{}_{f}}\ast\alpha )= L\left((\beta \ast \id _ {{}_{g}})\cdot(\id _ {{}_{f}}\ast\alpha )\right) $. 
This proves the first equation. The proof
of the second one is analogous.
\end{proof}

\begin{lem}\label{induced adjunction PhD}
If $L:\aaa\to\bbb $ is a pseudofunctor and  $(f\dashv g , \epsilon , \eta ): Y\to Z  $ is
an adjunction in $\aaa $, then $$(L(f)\dashv L(g), 
\mathfrak{l}_{{}_{Z}}^{-1}\cdot L(\epsilon )\cdot \mathfrak{l}_{{}_{fg}} , 
\mathfrak{l}_{{}_{gf}}^{-1}\cdot L(\eta )\cdot \mathfrak{l}_{{}_{Y}})$$  
is an adjunction in $\bbb $. Whenever $(f\dashv g , \epsilon , \eta )$ is implicit, we usually denote the induced adjunction above by $L(f)\dashv L(g) $.
\end{lem}

\begin{proof}
By Lemma \ref{Fundamental lema PhD}, we get that:
\begin{equation*}
   L \left(
    \vcenter{
\xymatrix{
Y
\ar@{=}[dd]
\ar@{}[rd]|-{\xRightarrow{\enspace\eta\quad } }
&
&
Z
\ar[ll]_-{g}
\ar@{=}[dd]
\\
&&
\\
Y
\ar[rruu]|-{f}
&
&
Z
\ar[ll]^-{g}
\ar@{}[lu]|-{\xRightarrow{\enspace\epsilon\quad } }
}}\right)    
        =         
            \mathfrak{l}_{{}_{g\,\id _{{}_{Z}} } }\cdot \left(\vcenter{
  \xymatrix{
L(Y)
\ar@{=}[dd]
\ar@{}[rd]|-{\xRightarrow{\mathfrak{l}_{{}_{gf}}^{-1}\cdot L(\eta )  } }
&
&
L(Z)
\ar[ll]_-{L(g)}
\ar@{=}[dd]
\\
&&
\\
L(Y)
\ar[rruu]|-{L(f)}
&
&
L(Z)
\ar[ll]^-{L(g)}
\ar@{}[lu]|-{\xRightarrow{L(\epsilon )\cdot \mathfrak{l}_{{}_{fg}}} }
}}
\right)\cdot \mathfrak{l}_{{}_{\id_{{}_{Y}} g } }^{-1},
\end{equation*}

\begin{equation*}
   L \left(
    \vcenter{
\xymatrix{
Y
\ar@{}[rd]|-{\xRightarrow{\enspace\eta\quad } }
\ar[rr]^-{f}
\ar@{=}[dd]
&
&
Z
\ar[lldd]|-{g}
\ar@{=}[dd]
\\
&&
\\
Y
\ar[rr]_-{f}
&
&
Z
\ar@{}[lu]|-{\xRightarrow{\enspace\epsilon\quad } }
}
}\right)    
        =         
              \mathfrak{l}_{{}_{ \id _{{}_{Z}} f } }\cdot \left(\vcenter{
  \xymatrix{
L(Y)
\ar@{}[rd]|-{\xRightarrow{\mathfrak{l}_{{}_{gf}}^{-1}\cdot L(\eta ) } }
\ar[rr]^-{L(f)}
\ar@{=}[dd]
&
&
L(Z)
\ar[lldd]|-{L(g)}
\ar@{=}[dd]
\\
&&
\\
L(Y)
\ar[rr]_-{L(f)}
&
&
L(Z)
\ar@{}[lu]|-{\xRightarrow{L(\epsilon )\cdot \mathfrak{l}_{{}_{fg}} } }
}}
\right)\cdot \mathfrak{l}_{{}_{f\,\id_{{}_{Y}} } }^{-1}.
\end{equation*} 
\normalsize
Since, by Equation 2 of Definition 4.2.1~\cite{MR3491845}, 
\begin{equation*}
\mathfrak{l}_{{}_{f\,\id_{{}_{Y}} } }^{-1} = \id _{{}_{L(f) }}\ast \mathfrak{l}_{{}_{Y}} ,\quad 
\mathfrak{l}_{{}_{\id_{{}_{Y}} g } }^{-1}= \mathfrak{l}_{{}_{Y}}\ast \id _{{}_{L(g) }}, \quad
 \mathfrak{l}_{{}_{ \id _{{}_{Z}} f } } = \mathfrak{l}_{{}_{Z}}^{-1}\ast \id _{{}_{L(f) }}, \quad
 \mathfrak{l}_{{}_{g\,\id _{{}_{Z}} } } = \id _{{}_{L(g) }}\ast \mathfrak{l}_{{}_{Z}} ^{-1},
\end{equation*}
 the result follows.
\end{proof}

\begin{theo}[Mate Correspondence]\label{Mate Correspondence Theorem PhD}
Let $(f\dashv g) := (f\dashv g , \epsilon  , \eta  ): Z\to Y   $
and $(l\dashv u) := (l\dashv u , \mu , \rho ): W\to X   $ be adjunctions in a $2$-category $\aaa $. 
Given $1$-cells $m: X\to Y $ and $n: W\to Z $ of $\aaa $,
there is a bijection $\aaa (X,Z)(nu, gm)\cong \aaa (W,Y)(fn, ml) $, given by $\alpha \mapsto \overline{\alpha } ^{f\dashv g}_{l\dashv u} $ in which $\overline{\alpha } ^{f\dashv g}_{l\dashv u}$ is defined by:
$$
\xymatrix@=3.5em{
&
W
\ar[r]^-{l}
\ar@{=}[d]
&
X
\ar[ld]|-{u}
\ar[d]^-{m}
\\
\overline{\alpha } ^{(f\dashv g)}_{(l\dashv u)}
\ar@{}[r]|-{:=}
&
W
\ar@{{}{ }{}}@/^0.8pc/[ru]|-{\xRightarrow{\rho } }
\ar[d]_-{n}
\ar@{}[r]|-{\xRightarrow{\alpha} }
&
Y
\ar[ld]|-{g}
\ar@{=}[d]
\\
&
Z
\ar[r]_-{f}
\ar@{{}{ }{}}@/_0.7pc/[ru]|-{\xRightarrow{\epsilon } }
&
Y
}
$$
We call $\overline{\alpha } ^{f\dashv g}_{l\dashv u}$ the mate of $\alpha $ under the adjunction $(l\dashv u , \mu , \rho )$ and the adjunction $(f\dashv g , \epsilon  , \eta  )$. 
\end{theo}

\begin{proof}
The map $\beta\mapsto \underline{\beta } ^{f\dashv g}_{l\dashv u} $ defined by
$$
\xymatrix@=4.2em{
&
X
\ar[r]^-{u}
\ar@{=}[d]
&
W
\ar[ld]|-{l}
\ar[d]^-{n}
\\
\underline{\beta } ^{(f\dashv g)}_{(l\dashv u)}
\ar@{}[r]|-{:=}
&
X
\ar@{{}{ }{}}@/^0.8pc/[ru]|-{\xLeftarrow{\mu } }
\ar[d]_-{m}
\ar@{}[r]|-{\xLeftarrow{\beta} }
&
Z
\ar[ld]|-{f}
\ar@{=}[d]
\\
&
Y
\ar[r]_-{g}
\ar@{{}{ }{}}@/_0.7pc/[ru]|-{\xLeftarrow{\eta } }
&
Z
}
$$
is clearly the inverse of $\alpha \mapsto \overline{\alpha } ^{(f\dashv g)}_{(l\dashv u)} $.
\end{proof}

Actually, we can say much more about this correspondence. For instance, this is  part of an isomorphism of double categories. More precisely, given 
a $2$-category $\aaa $, we define two double categories $\mathsf{RAdj}(\aaa ) $ and $\mathsf{LAdj} (\aaa )$. The objects and the horizontal arrows of both 
double categories 
are the objects and $1$-cells of $\aaa $, while the vertical arrows are adjunctions $(f\dashv g , \epsilon  , \eta  )   $ of $\aaa $. Given 
vertical arrows $(f\dashv g , \epsilon  , \eta  ): Z\to Y $, $(l\dashv u , \mu , \rho ): W\to X   $   and horizontal arrows $m: X\to Y $, $n: W\to Z $,
the squares of $\mathsf{RAdj}(\aaa ) $ are $2$-cells $\alpha : nu \Rightarrow gm $ of $\aaa $, while the squares of $\mathsf{LAdj} (\aaa )$ are $2$-cells
$\beta : fn\Rightarrow ml $. The composition of squares are given by pasting of $2$-cells, the composition of horizontal arrows is the composition 
of $1$-cells and the composition of vertical arrows is the composition of adjunctions as in $\aaa ^{\mathsf{adj}} $. It is clear, then, that the mate correspondence induces an isomorphism between $\mathsf{RAdj}(\aaa ) $ and $\mathsf{LAdj} (\aaa )$. In particular, the mate correspondence respects vertical and horizontal compositions. 

As a first application of these observations on the mate correspondence, we give another proof of the statement of Remark \ref{Adjuntos unicos capitulo 1 PhD}.
Indeed, in the context of Remark \ref{Adjuntos unicos capitulo 1 PhD}, we take the $2$-cells that are mates of the identity 
$$\xymatrix{ 
Z
\ar[d]_-{g}
\ar@{=}[r]
\ar@{}[rd]|-{=}
&
Z
\ar[d]^-{g}
\\
Y
\ar@{=}[r]
&
Y
}
$$
under the adjunctions $(\widetilde{f}\dashv g , \mu  , \rho )$ and $(f\dashv g , \epsilon , \eta ) $,
and under the adjunctions $(f\dashv g , \epsilon , \eta )$ and $(\widetilde{f}\dashv g , \mu  , \rho )$. They are respectively denoted by
$\psi : f\id _ {{}_{Y}}\Rightarrow \id _ {{}_{Z}} \widetilde{f} $ and $\psi ' : \widetilde{f}\id _ {{}_{Y}}\Rightarrow \id _ {{}_{Z}} f $
 (which actually are the $2$-cells of Remark \ref{Adjuntos unicos capitulo 1 PhD}). Since the mate correspondence preserves horizontal composition, 
the compositions  $\psi ' \psi $ and $\psi\psi ' $ 
 are respectively the mates of the $2$-cell $\id _ {{}_{Y}} g =  g\id _ {{}_{Z}} $ above  under $f\dashv g$ and itself, and under $\widetilde{f}\dashv \widetilde{g}$ and itself; that is to say, the identity on $f$ and the identity on $\widetilde{f}$. In particular, $\psi : f\Rightarrow   \widetilde{f} $ is an isomophism.

\begin{rem}\label{Adjuncoes em CAT Capitulo 1 PhD}
If $(f\dashv g , \epsilon , \eta ): Y\to Z $ is an adjunction in the $2$-category $\Cat $, we know that $Z(f-,-)\cong Y(-, g-) $. The mate correspondence 
generalizes this fact since, assuming now that $(l\dashv u , \mu , \rho ): Y'\to Z' $ is an adjunction in a $2$-category $\aaa $,  as a particular case of Theorem \ref{Mate Correspondence Theorem PhD}, we conclude that
$$\aaa (X , Y') (-, u-)\cong \aaa (X , Z')(l - , -) $$ 
for any object $X$ of $\aaa $. Still, up to size considerations, the Yoneda structure~\cite{MR0463261} of $\CAT $ implies that: given a functor $\widetilde{f}: Y\to Z $, $\widetilde{f}\dashv \widetilde{g} $ in $\Cat $ if and only if there is a natural isomorphism 
\begin{equation}\tag{$\varphi $}\label{phi PhD}
\xymatrix@=3em{
\CAT \left[Y ^\op, \SET \right]
\ar[rr]^-{ \CAT \left[ \widetilde{f} ^\op, \SET \right] }
&\ar@{}[d]|-{\cong }&
 \CAT \left[ Z ^\op, \SET \right]
\\
Y
\ar[rr]_-{\widetilde{g} }
\ar[u]^-{\YYYY _{Y } }
&&
Z
\ar[u]_-{\YYYY _{Z } }
}
\end{equation}
in which $\YYYY _ {Y }$ , $\YYYY _ {Z } $ denote the Yoneda embeddings. 
Moreover, it should be noted that, if $\widetilde{f}\dashv  \widetilde{g} $, then
$$\CAT \left[ \widetilde{f} ^\op, \SET \right] \dashv \CAT \left[ \widetilde{g} ^\op, \SET \right] $$
 by the $2$-functoriality of $\CAT \left[(-) ^\op, \SET \right]: \CAT ^{\co\op}\to\CAT $.
\end{rem}

It is clear that the images of the mates by $2$-functors are the mates of the images.
In the case of pseudofunctors, it follows from Lemma \ref{Fundamental lema PhD} that:

\begin{lem}\label{IMAGE OF MATE PSEUDOFUNCTOR}
Let $E: \aaa\to\bbb $ be a pseudofunctor and $(f\dashv g , \epsilon , \eta )$, $(l\dashv u , \mu , \rho ) $ 
adjunctions in $\aaa $. Given a $2$-cell $\alpha : nu\Rightarrow gm $,
$$\widehat{E\left( \overline{\alpha } ^{f\dashv g}_{l\dashv u}\right)} = \overline{\widehat{E(\alpha)} } ^{E(f)\dashv E(g)}_{E(l)\dashv E(u)},$$
in which
$\widehat{E\left( \overline{\alpha } ^{f\dashv g}_{l\dashv u}\right)} : = 
\mathfrak{e}_{{}_{ml}} ^{-1}
\cdot
E\left( \overline{\alpha } ^{f\dashv g}_{l\dashv u}\right)\cdot \mathfrak{e}_{{}_{fn}}$ and 
$\widehat{E(\alpha)}  : = \mathfrak{e}_{{}_{gm}} ^{-1}\cdot E(\alpha)\cdot  \mathfrak{e}_{{}_{nu}} $.
In other words, $$\mathfrak{e}_{{}_{ml}} ^{-1}
\cdot
E\left( \overline{\alpha } ^{f\dashv g}_{l\dashv u}\right)\cdot \mathfrak{e}_{{}_{fn}}$$ is the 
mate of $\mathfrak{e}_{{}_{gm}} ^{-1}\cdot E(\alpha)\cdot  \mathfrak{e}_{{}_{nu}}$ under $E(l)\dashv E(u) $ and
$E(f)\dashv E(g) $.
\end{lem}

We also have an important result relating mates and pseudonatural transformations:

\begin{lem}\label{Pseudonatural MATE PhD}
Let  $\lambda : E\longrightarrow L $ be a pseudonatural 
transformation between pseudofunctors. If $(f\dashv g , \epsilon , \eta ): Y\to Z  $ is an adjunction of $\aaa $, then 
the mate of
$$\xymatrix@=3em{
&
E(Y)
\ar[ld]_-{E(f)}
\ar[d]^-{\lambda _{{}_{Y}} }
\\
E(Z)
\ar[d]_-{\lambda _ {{}_{Z}} }
\ar@{}[r]|-{\xLeftarrow{\enspace\lambda _ {{}_{f}}\enspace  } }
&
L(Y)\ar[ld]^-{L(f) }
\\
L(Z)
&
}$$
under the adjunction   
$(E(f)\dashv E(g), 
\mathfrak{e}_{{}_{Z}}^{-1}\cdot E(\epsilon )\cdot \mathfrak{e}_{{}_{fg}} , 
\mathfrak{e}_{{}_{gf}}^{-1}\cdot E(\eta )\cdot \mathfrak{e}_{{}_{Y}})$ 
and $(L(f)\dashv L(g), 
\mathfrak{l}_{{}_{Z}}^{-1}\cdot L(\epsilon )\cdot \mathfrak{l}_{{}_{fg}} , 
\mathfrak{l}_{{}_{gf}}^{-1}\cdot L(\eta )\cdot \mathfrak{l}_{{}_{Y}})$ is equal
to $\lambda _{{}_{g}}^{-1} $.
\end{lem}

\begin{proof}
In order to simplify the terminology, we prove it for a pseudonatural 
transformation $\lambda : E\longrightarrow L $ between $2$-functors.
The proof in the case of pseudofunctors is analogous. 

The proof consists in verifying that the mate $\underline{\lambda _ {{}_{f}}  } ^{L(f)\dashv L(g)}_{E(f)\dashv E(g)} $ composed with $\lambda _ {{}_{g}} $ is equal to the identity $L(g)\lambda _ {{}_{Z}}\Rightarrow L(g)\lambda _ {{}_{Z}} $. Firstly, observe that this composition is equal to
$$
\xymatrix@C=4em@R=3em{
& 
E(Z)
\ar[rd]^-{\lambda _ {{}_{Z}} }
\ar@{=}[ld]
\ar[d]|-{E(g) }
\ar@{{}{ }{}}@/^0.8pc/[ld]|-{\xLeftarrow{ E(\epsilon ) } }
&
\\
E(Z)
\ar[dd]|-{\lambda _ {{}_{Z}} }
&
E(Y)
\ar[d]|-{\lambda _ {{}_{Y}} }
\ar[l]|-{E(f)}
&
L(Z)
\ar[ld]|-{L(g)}
\ar@{}[l]|-{\xLeftarrow{\enspace\lambda _ {{}_{g}}\enspace } }
\\
&
L(Y)\ar@{{}{ }{}}@/^1pc/[ld]|-{\xLeftarrow{ L(\eta ) } }
\ar@{}[l]|-{\xLeftarrow{\enspace\lambda _ {{}_{f}}\enspace } }
\ar[ld]|-{L(f)}
\ar@{=}[d]
&
\\
L(Z)
\ar[r]|-{L(g)}
&
L(Y)
&
}$$ 
Since $\lambda _ {{}_{\id _ {{}_{Z}} }} = \id _ {{}_{\lambda _ {{}_{Z}} }} $,  by Equations 1 and 3 
 of Definition 4.2.1~\cite{MR3491845}
we get that this composition is equal to
$$\xymatrix@C=4em@R=3em{
E(Z)\ar@{=}[dd]
\ar[r]|-{\lambda _ {{}_{Z}} }
\ar[rd]|-{E(g)}
&
L(Z)
\ar[rd]|-{L(g)}
&
\ar@{}[ddr]|-{=}
&
E(Z)
\ar[r]|-{\lambda _ {{}_{Z}} }
\ar@{=}[dd]
\ar@{}[ddr]|-{\xLeftrightarrow{\id _ {{}_{\lambda _{{}_{Z}} }}} }
&
L(Z)
\ar@{=}[dd]
\ar[dr]|-{L(g)}
&
\\
&
E(Y)
\ar@{}[l]|-{\xLeftarrow{E(\epsilon ) } }
\ar@{}[r]|-{\xLeftarrow{\enspace\lambda _ {{}_{fg}}\enspace } }
\ar[ld]|-{E(f)}
&
L(Y)
\ar@{=}[d]
\ar[ld]|-{L(f)}
\ar@{{}{ }{}}@/^1pc/[ld]|-{\xLeftarrow{ L(\eta ) } }
&
&
&
L(Y)
\ar@{}[l]|-{\xLeftarrow{ L(\epsilon ) } }
\ar[ld]|-{L(f)}
\ar@{=}[d]
\ar@{{}{ }{}}@/^1pc/[ld]|-{\xLeftarrow{ L(\eta ) } }
\\
E(Z)
\ar[r]|-{\lambda _ {{}_{Z}} }
&
L(Z)\ar[r]|-{L(g)}
&
L(Y)
&
E(Z)
\ar[r]|-{\lambda _ {{}_{Z}} }
&
L(Z)
\ar[r]|-{L(g)}
&
L(Y)
}$$
which is clearly equal to the identity $L(g)\lambda _ {{}_{Z}}\Rightarrow L(g)\lambda _ {{}_{Z}} $, since $(L(f)\dashv L(g), L(\epsilon ), L(\eta )) $ is an adjunction.
\end{proof}

\begin{defi}[Beck-Chevalley condition]\label{Beck-Chevalley Obsoleto PhD}
Let $(f\dashv g) := (f\dashv g , \epsilon  , \eta  ): Z\to Y   $
and $(l\dashv u) := (l\dashv u , \mu , \rho ): W\to X   $ be adjunctions in a $2$-category $\aaa $. Assume that $\alpha : nu \Rightarrow gm $ is a $2$-cell
of $\aaa $. We say that
$$\xymatrix@=2.5em{
&
X
\ar[ld]_-{u}
\ar[d]^-{m}
\\
W
\ar[d]_-{n}
\ar@{}[r]|-{\xRightarrow{\alpha } }
&
Y\ar[ld]^-{g}
\\
Z
&
}$$
satisfies the Beck-Chevalley condition if the mate of $\alpha $ 
under $l\dashv u $ and $f\dashv g $ is an invertible $2$-cell.
\end{defi}

Outside any context, the meaning of the Beck-Chevalley condition might seem vacuous. Even when some context is provided, it is many times considered as an isolated
 technical condition. In this thesis, however, this condition is always applied in the context of \textit{doctrinal adjunction}. More precisely, our informal perspective is that, whenever the Beck-Chevalley condition plays an important
role, we can frame our problem in terms of $2$-dimensional monad theory, getting a  problem of lifting of adjunctions of the base $2$-category
 to the $2$-category of pseudoalgebras. 
The Beck-Chevalley condition is precisely the obstruction condition to this lifting. 
The most important example of this approach in this thesis 
is Chapter 3 or, more specifically, the proof of the B\'{e}nabou-Robaud Theorem presented therein.

Below, we briefly explain the Beck-Chevalley condition within the context of $2$-dimensional monad theory. This section can be considered, then, as prerequisite to 
the understanding of Section 7 of Chapter 3~\cite{2016arXiv160604999L}, since herein we do not assume familiarity with the doctrinal adjunction. We also show in \ref{KZ-PSEUDOMONADA CAPITULO 1 PhD}
 how, within our context,
Kock-Z\"{o}berlein pseudomonads encompass the
 situation
of ``the Beck-Chevalley condition always holding''.

We start by showing the most elementary version of an important bijection between colax and lax $\TTTTT$-structures in adjoint morphisms.
Again, the mate correspondence is the basic technique to 
introduce this bijection. 

Let 
$\TTTTT $ be a pseudomonad
 on a $2$-category $\bbb $ and $g: Z\to Y $ a morphism of $\bbb $. Given lax $\TTTTT$-algebras
$\mathtt{y}=(Y, \alg _ {{}_{\mathtt{y}}}, \overline{{\underline{\mathtt{y}}}}, \overline{{\underline{\mathtt{y}}}}_0 )$ and $\mathtt{z}= (Z, \alg _ {{}_{\mathtt{z}}}, \overline{{\underline{\mathtt{z}}}}, \overline{{\underline{\mathtt{z}}}}_0 )$, the \textit{collection of lax $\TTTTT$-structures for $g:Z\to Y$ w.r.t. $\mathtt{z}$ and 
$\mathtt{y}$}, denoted by 
$\mathsf{Lax}\textrm{-}\TTTTT\textrm{-}\Alg _{\ell }(\mathtt{z}, \mathtt{y})_g $, is the pullback of
the inclusion of $g $ in the category of morphisms $\bbb (Z, Y) $, $\mathsf{1}\to \bbb (Z,Y) $, along
the functor $\mathsf{Lax}\textrm{-}\TTTTT\textrm{-}\Alg _{\ell } (\mathtt{z}, \mathtt{y})\to \bbb (Z,Y)  $ induced by the forgetful $2$-functor. Analogously, given a morphism $f: Y\to Z $ of $\bbb$,  $\mathsf{Lax}\textrm{-}\TTTTT\textrm{-}\Alg _{\mathsf{c}\ell}(\mathtt{y}, \mathtt{z})_f $ is the pullback of the inclusion of $f$ into $\bbb (Y,Z) $ along 
the forgetful functor $\mathsf{Lax}\textrm{-}\TTTTT\textrm{-}\Alg _{\mathsf{c}\ell } (\mathtt{y}, \mathtt{z})\to \bbb (Y,Z)  $.

It is clear, then, that
 a lax $\TTTTT $-morphism in $\mathsf{Lax}\textrm{-}\TTTTT\textrm{-}\Alg _{\ell}(\mathtt{z}, \mathtt{y})_g $ corresponds to a $2$-cell
$\left\langle \overline{\mathsf{g}}\right\rangle :\alg _ {{}_{\mathtt{y}}}\TTTTT (g)
\Rightarrow g\alg _ {{}_{\mathtt{z}}}    $ of $\bbb $
satisfying the axioms of Definition 5.4.1~\cite{2016arXiv160703087L}, while a colax $\TTTTT $-morphism in
$\mathsf{Lax}\textrm{-}\TTTTT\textrm{-}\Alg _{\mathsf{c}\ell}(\mathtt{y}, \mathtt{z})_f $
corresponds to a $2$-cell
$\left\langle   \overline{\mathsf{f}}\right\rangle :  f\alg _ {{}_{\mathtt{y}}}\Rightarrow\alg _ {{}_{\mathtt{z}}}\TTTTT (f)$ of $\bbb $ satisfying the axioms of Definition \ref{LAXALGEBRASCOLAX}. 

Moreover, we can consider the \textit{category of lax $\TTTTT$-structures for $f$ w.r.t. lax $\TTTTT $-algebras} which is the pullback of the inclusion of the morphism $f$ into $\bbb $,  $\mathsf{2}\to \bbb $, along the forgetful $2$-functor  $\mathsf{Lax}\textrm{-}\TTTTT\textrm{-}\Alg _{\ell } \to \bbb  $.  Finally, the 
 \textit{category of colax $\TTTTT$-structures for $g$ w.r.t. lax $\TTTTT $-algebras} is the pullback of the inclusion of the morphism $g$ into $\bbb $ along $\mathsf{Lax}\textrm{-}\TTTTT\textrm{-}\Alg _{\mathsf{c}\ell } \to \bbb  $.

\begin{theo}[Colax and lax structures in adjoints]\label{Main Theorem Doctrinal Correspondence PhD}
Let $\TTTTT $ be a pseudomonad on $\bbb $ and $$(f\dashv g , \epsilon  , \rho  ): Y\to Z   $$ an adjunction in $\bbb $. Given lax $\TTTTT$-algebras $\mathtt{y}=(Y, \alg _ {{}_{\mathtt{y}}}, \overline{{\underline{\mathtt{y}}}}, \overline{{\underline{\mathtt{y}}}}_0 )$ and $\mathtt{z}= (Z, \alg _ {{}_{\mathtt{z}}}, \overline{{\underline{\mathtt{z}}}}, \overline{{\underline{\mathtt{z}}}}_0 )$,
the mate correspondence 
under the adjunction  
$(\TTTTT (f)\dashv \TTTTT (g), 
\mathfrak{t}_{{}_{Z}}^{-1}\cdot \TTTTT (\epsilon )\cdot \mathfrak{t}_{{}_{fg}} , 
\mathfrak{t}_{{}_{gf}}^{-1}\cdot \TTTTT (\rho )\cdot \mathfrak{t}_{{}_{Y}})$
and the adjunction $(f\dashv g , \epsilon  , \rho  )$
induces a bijection
$$\diamond : \mathsf{Lax}\textrm{-}\TTTTT\textrm{-}\Alg _{\ell }(\mathtt{z}, \mathtt{y})_g
\cong \mathsf{Lax}\textrm{-}\TTTTT\textrm{-}\Alg _{\mathsf{c}\ell }(\mathtt{y}, \mathtt{z})_f. $$
These bijections induce an isomorphism between the category of lax $\TTTTT$-structures for $g: Z\to Y$
and category of colax $\TTTTT$-structures for $f: Y\to Z$ w.r.t. lax $\TTTTT $-algebras. 
\end{theo}

\begin{proof}
Assume that $(f\dashv g , \epsilon  , \rho  ): Y\to Z   $ is an adjunction in $\bbb $ and
$(\TTTTT,  m  , \eta , \mu,   \iota, \tau ) $ is a pseudomonad on $\bbb $ as in the hypothesis. 
Given $2$-cells $\left\langle \overline{\mathsf{g}}\right\rangle :\alg _ {{}_{\mathtt{z}}}\TTTTT (g)
\Rightarrow g\alg _ {{}_{\mathtt{y}}}    $ and $\left\langle   \overline{\mathsf{f}}\right\rangle :  f\alg _ {{}_{\mathtt{y}}}\Rightarrow\alg _ {{}_{\mathtt{z}}}\TTTTT (f)$ that are mates under 
the adjunctions $(\TTTTT (f)\dashv \TTTTT (g), 
\mathfrak{t}_{{}_{Z}}^{-1}\cdot \TTTTT (\epsilon )\cdot \mathfrak{t}_{{}_{fg}} , 
\mathfrak{t}_{{}_{gf}}^{-1}\cdot \TTTTT (\rho )\cdot \mathfrak{t}_{{}_{Y}})$
and $(f\dashv g , \epsilon  , \rho  )$, we have that: 
\begin{enumerate}
\item By Lemmas \ref{IMAGE OF MATE PSEUDOFUNCTOR} and \ref{Pseudonatural MATE PhD}, the $2$-cells
$$\xymatrix@=2.8em{
&
\TTTTT Z\ar[r]^{\alg _ {{}_{\mathtt{z}}}}\ar@{}[dd]|-{\xLeftarrow{ m_{{}_{f}}  }}
&
Z\ar@{}[d]|-{\xLeftarrow{\left\langle\overline{\mathsf{f}}\right\rangle } }
&
&
&
\TTTTT ^2 Y\ar@{}[d]|-{\xRightarrow{\widehat{\TTTTT (\left\langle   \overline{\mathsf{f}}\right\rangle)}             } }
\ar[ld]_ -{\TTTTT (\alg _ {{}_{\mathtt{y}}}) }\ar[r]^-{\TTTTT ^2(f) }
&
\TTTTT ^2 Z\ar[rd]^-{m _ {{}_{Z}} }\ar[ld]|-{\TTTTT (\alg _ {{}_{\mathtt{z}}}) }\ar@{}[dd]|-{\xRightarrow{\enspace\overline{{\underline{\mathtt{z}}} }\enspace} }
&
\\
\TTTTT ^2 Z\ar[ru]^-{m _ {{}_{Z}} }
&
&
\TTTTT Y\ar[lu]|-{\TTTTT (f) }\ar[r]^{\alg _ {{}_{\mathtt{y}}} }\ar@{}[d]|-{\xLeftarrow{\overline{{\underline{\mathtt{y}}} } } }
&
Y\ar[lu]_ {f}
&
\TTTTT Y\ar[r]_{\TTTTT (f) }\ar[rd]_{\alg _ {{}_{\mathtt{y}}}}
&
\TTTTT Z\ar[rd]|-{\alg _ {{}_{\mathtt{z}}} }\ar@{}[d]|-{\xRightarrow{\left\langle   \overline{\mathsf{f}}\right\rangle} }
&
&
\TTTTT Z\ar[ld]^{\alg _ {{}_{\mathtt{z}}} }
\\
&
\TTTTT ^2Y\ar[ru]|-{m _ {{}_{Y}} }\ar[r]_{\TTTTT (\alg _ {{}_{\mathtt{y}}}) }\ar[lu]^{\TTTTT ^2(f)}
&
\TTTTT Y\ar[ru]_-{\alg _ {{}_{\mathtt{y}}} }
&
&
&
Y\ar[r]_{f}
&
Z
&
}$$
are respectively the mates of
$$\xymatrix@=2.8em{
&
\TTTTT ^2 Z\ar@{}[d]|-{\xLeftarrow{m_{{}_{g}}^{-1} } }\ar[ld]_ -{m _ {{}_{Z}} }\ar[r]^-{\TTTTT ^2(g) }
&
\TTTTT ^2 Y\ar[rd]^-{\TTTTT (\alg _ {{}_{\mathtt{y}}}) }\ar[ld]|-{m _ {{}_{Y}} }
\ar@{}[dd]|-{\xLeftarrow{\quad\overline{{\underline{\mathtt{y}}} }\quad } }
&
&
&
\TTTTT Y
\ar[r]^{\alg _ {{}_{\mathtt{y}}}}
\ar@{}[dd]|-{\xRightarrow{\widehat{\TTTTT (\left\langle   \overline{\mathsf{g}}\right\rangle)}}}
&
Y
\ar@{}[d]|-{\xRightarrow{\left\langle\overline{\mathsf{g}}\right\rangle } }
&
\\
\TTTTT Z
\ar[r]_-{\TTTTT (g) }\ar[rd]_{\alg _ {{}_{\mathtt{z}}}}
&
\TTTTT Y\ar[rd]|-{\alg _ {{}_{\mathtt{y}}} }
\ar@{}[d]|-{\xLeftarrow{\left\langle   \overline{\mathsf{g}}\right\rangle} }
&
&
\TTTTT Y
\ar[ld]^{\alg _ {{}_{\mathtt{y}}} }
&
\TTTTT ^2 Y\ar[ru]^-{\TTTTT (\alg _ {{}_{\mathtt{y}}} )}
&
&
\TTTTT Z\ar[lu]|-{\TTTTT (g) }\ar[r]^{\alg _ {{}_{\mathtt{z}}} }
\ar@{}[d]|-{\xRightarrow{\overline{{\underline{\mathtt{z}}} } } }
&
Z\ar[lu]_ {g}
\\
&
Z\ar[r]_{g}
&
Y
&
&
&
\TTTTT ^2 Z\ar[ru]|-{\TTTTT (\alg _ {{}_{\mathtt{z}}})}
\ar[r]_{m _ {{}_{Z}} }\ar[lu]^{\TTTTT ^2 (g)}
&
\TTTTT Z\ar[ru]_-{\alg _ {{}_{\mathtt{z}}} }
&
}$$
under the adjunctions $\TTTTT ^2 (f)\dashv \TTTTT ^2 (g)$ and $f\dashv g $.

\item By Lemma \ref{Pseudonatural MATE PhD}, the $2$-cells
$$\xymatrix@=2.8em{ 
Z\ar[d]_{\eta _ {{}_{Z}} }\ar@{}[rd]|-{\xLeftarrow{\eta _ {{}_{f}}} }
&
&
Y\ar[ld]|-{\eta _ {{}_{Y}} }\ar@{=}[dd]\ar[ll]_-{f}
&
&
Z\ar[ld]_{\eta _ {{}_{Z}} }\ar@{=}[dd]
&
\\
\TTTTT Z\ar[d]_-{\alg _ {{}_{\mathtt{z}}} }\ar@{}[rd]|-{\xLeftarrow{\left\langle   \overline{\mathsf{f}}\right\rangle}}
&
\TTTTT Y\ar@{}[r]|-{\xLeftarrow{\overline{{\underline{\mathtt{y}}}}_0} }\ar[rd]|-{\alg _ {{}_{\mathtt{y}}} }\ar[l]^-{\TTTTT (f)}
&
&
\TTTTT Z\ar@{}[r]|-{\xLeftarrow{\overline{{\underline{\mathtt{z}}}}_0 } }\ar[rd]_-{\alg _ {{}_{\mathtt{z}}}}
&
&
\\
Z
&
&
Y\ar[ll]^-{f}
&
&
Z
&
Y\ar[l]_{f}
}$$
are respectively the mates of
$$\xymatrix@=2.8em{ 
Z\ar[rr]^{g}\ar[d]_{\eta _ {{}_{Z}} }\ar@{}[rd]|-{\xLeftarrow{\eta _ {{}_{g}}^{-1}} }
&
&
Y\ar[ld]_-{\eta _ {{}_{Y}} }\ar@{=}[dd]
&
&
Z\ar[ld]_{\eta _ {{}_{Z}} }\ar@{=}[dd]
&
\\
\TTTTT Z\ar[r]_{\TTTTT (g)}\ar[d]_-{\alg _ {{}_{\mathtt{z}}} }\ar@{}[rd]|-{\xLeftarrow{\left\langle   \overline{\mathsf{g}}\right\rangle}}
&
\TTTTT Y\ar@{}[r]|-{\xLeftarrow{\overline{{\underline{\mathtt{y}}}}_0} }\ar[rd]_-{\alg _ {{}_{\mathtt{y}}} }
&
&
\TTTTT Z\ar@{}[r]|-{\xLeftarrow{\overline{{\underline{\mathtt{z}}}}_0 } }
\ar[rd]_{\alg _ {{}_{\mathtt{z}}}}
&
&
\\
Z\ar[rr]_-{g}
&
&
Y
&
&
Z\ar[r]_{g}
&
Y
}$$
under $f\dashv g$ and itself.
\end{enumerate}
Therefore $\left\langle \overline{\mathsf{g}}\right\rangle :\alg _ {{}_{\mathtt{z}}}\TTTTT (g)
\Rightarrow g\alg _ {{}_{\mathtt{y}}}    $ 
corresponds to a lax $\TTTTT $-morphism in $\mathsf{Lax}\textrm{-}\TTTTT\textrm{-}\Alg _{\ell}(\mathtt{z}, \mathtt{y})_g $ if and only if $\left\langle   \overline{\mathsf{f}}\right\rangle :  f\alg _ {{}_{\mathtt{y}}}\Rightarrow\alg _ {{}_{\mathtt{z}}}\TTTTT (f)$ corresponds to a colax $\TTTTT $-morphism
in $\mathsf{Lax}\textrm{-}\TTTTT\textrm{-}\Alg _{\mathsf{c}\ell}(\mathtt{y}, \mathtt{z})_f $.
\end{proof}

\begin{rem}
Clearly, we have dual results. For instance, given colax $\TTTTT $-algebras $\mathtt{z}, \mathtt{y}$ and an adjunction $(f\dashv g, \epsilon , \rho ): Y\to Z $
 in the
base $2$-category $\bbb $, 
we can analogously define the collections $\mathsf{Colax}\textrm{-}\TTTTT\textrm{-}\Alg _{\mathsf{c}\ell}(\mathtt{y}, \mathtt{z})_f$
and $\mathsf{Colax}\textrm{-}\TTTTT\textrm{-}\Alg _{\ell}(\mathtt{z}, \mathtt{y})_g$ of colax $\TTTTT $-structures for $f$ and lax $\TTTTT $-structures for 
$g$. The mate correspondence induces a bijection between such collections.
\end{rem}

Now, we can introduce the Beck-Chevalley condition in the context of lax $\TTTTT $-morphisms, being a particular case of Definition \ref{Beck-Chevalley Obsoleto PhD}. The relevance of this case is demonstrated in Theorem \ref{Doctrinal Adjunction Theorem PhD}.

\begin{defi}[Beck-Chevalley within $2$-dimensional monad theory]
Let $\TTTTT $ be a pseudomonad on $\bbb $ and $(f\dashv g , \epsilon  , \rho  ): Y\to Z   $ an adjunction in $\bbb $. Assume that $\mathtt{g} = (g, \left\langle \overline{\mathsf{g}}\right\rangle ): \mathtt{z}\to \mathtt{y}  $ is a lax $\TTTTT $-morphism between the
lax $\TTTTT$-algebras $\mathtt{z} = (Z, \alg _ {{}_{\mathtt{z}}}, \overline{{\underline{\mathtt{z}}}}, \overline{{\underline{\mathtt{z}}}}_0 )$
and $\mathtt{y} = (Y, \alg _ {{}_{\mathtt{y}}}, \overline{{\underline{\mathtt{y}}}}, \overline{{\underline{\mathtt{y}}}}_0 )$. 

We say that $\mathtt{g}$ \textit{satisfies the Beck-Chevalley condition} if the
corresponding colax $\TTTTT $-morphism $(f, \diamond \left\langle \overline{\mathsf{g}}\right\rangle  ) : \mathtt{z}\to\mathtt{y} $ via the bijection of Theorem \ref{Main Theorem Doctrinal Correspondence PhD} is a $\TTTTT $-pseudomorphism. That is to say, $\mathtt{g}$ satisfies the Beck-Chevalley condition if
\begin{equation*}
\diamond \left\langle \overline{\mathsf{g}}\right\rangle  =
\vcenter{\xymatrix@C=7em@R=3em{
\TTTTT Y
\ar[r]^-{\TTTTT (f )}
\ar@{=}[dd]
&
\TTTTT Z\ar@{{}{ }{}}@/_0.4pc/[ld]|-{\xRightarrow{ \mathfrak{t}_{{}_{gf}}^{-1}\cdot \TTTTT (\rho )\cdot \mathfrak{t}_{{}_{Y}} } }
\ar[ldd]|-{\TTTTT (g) }
\ar[dd]^-{\alg _ {{}_{\mathtt{z}}}  }
\\
&
\\
\TTTTT Y
\ar[dd]_-{\alg _ {{}_{\mathtt{y}}} }
\ar@{}[r]|-{\xRightarrow{\enspace\left\langle \overline{\mathsf{g}}\right\rangle\enspace} }
&
Z
\ar[ldd]|-{g}
\ar@{=}[dd]
\\
&
\\
Z\ar@{{}{ }{}}[ru]|-{\xRightarrow{\enspace\epsilon\enspace } }
\ar[r]_-{f}
&
Y
}}
\end{equation*}
is an invertible $2$-cell.
\end{defi}

Given a $2$-monad $\TTTTT $ on a $2$-category $\bbb $, while the forgetful $2$-functors  $\TTTTT\textrm{-}\Alg _{\textrm{s} }\to \bbb $ and
$ \mathsf{Ps}\textrm{-}\TTTTT\textrm{-}\Alg\to\bbb $ respectively reflect isomorphisms and equivalences, the forgetful $2$-functor 
$$\mathsf{Lax}\textrm{-}\TTTTT\textrm{-}\Alg _{\ell }\to \bbb $$ reflects right adjoints that satisfy the Beck-Chevalley condition. More generally, 
the \textit{Doctrinal Adjunction} characterizes when the unit and the counit of an adjunction satisfy the condition of being a $\TTTTT $-transformation (given in
in Eq. 3 of Definition 5.4.1~\cite{2016arXiv160703087L}).

\begin{theo}[Doctrinal Adjunction]\label{Doctrinal Adjunction Theorem PhD}
Let $\TTTTT $ be a pseudomonad on $\bbb $ and $$\mathtt{g} = (g, \left\langle \overline{\mathsf{g}}\right\rangle ): \mathtt{z}\to \mathtt{y}, \mathtt{f} = (f, \left\langle \overline{\mathsf{f}}\right\rangle ): \mathtt{y}\to \mathtt{z}  $$ lax $\TTTTT$-morphisms. Assume
that $(f\dashv g , \epsilon  , \rho  ): Y\to Z   $ is an adjunction in $\bbb $. The $2$-cells $\epsilon , \rho $ give $\TTTTT $-transformations 
$\widetilde{\epsilon } : \mathtt{f}\mathtt{g}\Rightarrow \id _ {{}_{\mathtt{z} }}$  and  $\widetilde {\rho }: \id _ {{}_{\mathtt{y}}}\Rightarrow \mathtt{g}\mathtt{f} $
if and only if  $\left\langle \overline{\mathsf{f}}\right\rangle $ is invertible and $\diamond \left\langle \overline{\mathsf{g}}\right\rangle  = \left\langle \overline{\mathsf{f}}\right\rangle ^{-1}$. In this case, $(\mathtt{f}\dashv \mathtt{g} , \widetilde{\epsilon }  , \widetilde{\rho }  )$ is an adjunction in $\mathsf{Lax}\textrm{-}\TTTTT\textrm{-}\Alg _{\ell } $.
\end{theo}

\begin{proof}
Assume that $\epsilon , \rho $ give $\TTTTT $-transformations 
$\widetilde{\epsilon } : \mathtt{f}\mathtt{g}\Rightarrow \id _ {{}_{\mathtt{z} }}$  and  $\widetilde {\rho }: \id _ {{}_{\mathtt{y}}}\Rightarrow \mathtt{g}\mathtt{f} $ in $\mathsf{Lax}\textrm{-}\TTTTT\textrm{-}\Alg _{\ell } $. That is to say, by hypothesis, $\epsilon , \rho $ satisfy Eq. 3  of Definition 5.4.1~\cite{2016arXiv160703087L}. 
Denoting $\mathfrak{t}_{{}_{gf}}^{-1}\cdot \TTTTT (\rho )\cdot \mathfrak{t}_{{}_{Y}}$ by $\widehat{\TTTTT (\rho )} $, we have that $\left\langle\overline{\mathsf{f}}\right\rangle \cdot \left( \diamond \left\langle \overline{\mathsf{g}}\right\rangle\right) $ and 
$\left( \diamond \left\langle \overline{\mathsf{g}}\right\rangle\right)\cdot \left\langle\overline{\mathsf{f}}\right\rangle $ are respectively equal to
\begin{equation*}
\vcenter{
\xymatrix@C=4.8em@R=5em{
\TTTTT Y
\ar[dr]|-{\TTTTT (f) }
\ar[d]|-{\alg _ {{}_{\mathtt{y}}} }
&
&
\TTTTT Y
\ar[ll]|-{\id _ {{}_{\TTTTT Y }} }
\ar@{}[ld]|-{\xLeftarrow{\widehat{\TTTTT (\rho )} } }
\ar@{=}[dd]
\\
Y
\ar@{}[r]|-{\xLeftarrow{\left\langle \overline{\mathsf{f}}\right\rangle } }
\ar[dr]|-{f}
&
\TTTTT Z
\ar[dr]|-{\TTTTT (g) }
\ar[d]|-{\alg _ {{}_{\mathtt{z}}} }
&
\\
&
Z
\ar@{{}{ }{}}@/_1pc/[dr]|-{ \xLeftarrow{\epsilon }}
\ar@{=}[d]
\ar@{}[r]|-{ \xLeftarrow{\left\langle \overline{\mathsf{g}}\right\rangle }  }
\ar[dr]|-{g}
&
\TTTTT Y
\ar[d]|-{\alg _ {{}_{\mathtt{y}}}}
\\
&
Z
&
Y\ar[l]|-{f}
}}
\quad\mbox{and}\quad
\vcenter{
\xymatrix@C=4.8em@R=5em{
\TTTTT Z
\ar[dr]|-{\TTTTT (g) }
\ar[d]|-{\alg _ {{}_{\mathtt{z}}} }
\ar@{{}{ }{}}@/^1.5pc/[dr]|-{ \xLeftarrow{\widehat{\TTTTT (\rho )} } }
&
\TTTTT Y\ar@{=}[d]
\ar[l]|-{\TTTTT (f)}
&
\\
Z
\ar@{}[r]|-{\xLeftarrow{\left\langle \overline{\mathsf{g}}\right\rangle } }
\ar[dr]|-{g}
\ar@{=}[dd]
&
\TTTTT Y
\ar[d]|-{\alg _ {{}_{\mathtt{y}}}}
\ar[dr]|-{\TTTTT (f)}
&
\\
&
Y
\ar[dr]|-{f}
\ar@{}[r]|-{\xLeftarrow{\left\langle \overline{\mathsf{f}}\right\rangle } }
&
\TTTTT Z
\ar[d]|-{\alg _ {{}_{\mathtt{z}}}}
\\
Z
\ar@{}[ur]|-{\xLeftarrow{\quad\epsilon\quad } }
&
&
Z\ar[ll]|-{\id _ {{}_{Z}} }
}
}
\end{equation*}
which, by Eq. 3  of Definition 5.4.1 and the triangle identities (of Definition \ref{Adjuncao Capitulo 1 PhD}), are respectively equal to the identities $f\alg _ {{}_{\mathtt{y}}}\Rightarrow f\alg _ {{}_{\mathtt{y}}}$ and $\alg _ {{}_{\mathtt{z}}}\TTTTT (f)\Rightarrow\alg _ {{}_{\mathtt{z}}}\TTTTT (f)$. Therefore the proof of the first part is complete.

Reciprocally, assume now that $\diamond \left\langle \overline{\mathsf{g}}\right\rangle  = \left\langle \overline{\mathsf{f}}\right\rangle ^{-1}$. 
We prove below that $\rho $ gives a $\TTTTT $-transformation $\widetilde {\rho }: \id _ {{}_{\mathtt{y}}}\Rightarrow \mathtt{g}\mathtt{f} $. On one hand,
the mate of the $2$-cell
\begin{equation*}
\left(
\vcenter{
\xymatrix@C=5em@R=4em{
\TTTTT Y
\ar@/_2pc/[dd]|-{\TTTTT (\id _ {{}_{Y}} ) }
\ar@/^2pc/[dd]|-{\TTTTT (gf)}
\ar@{}[dd]|-{\xRightarrow{\TTTTT (\rho ) } }
\ar[r]|-{\alg _ {{}_{\mathtt{y}}} }
\ar@{{}{ }{}}@/^3.8pc/[dd]|-{\xRightarrow{\left\langle \overline{\mathsf{g}\mathsf{f}}\right\rangle} }
&
Y
\ar[dd]|-{gf  }
\\
&
\\
\TTTTT Y
\ar[r]|-{\alg _ {{}_{\mathtt{y}}}  }
&
Y
}
}
\right)\cdot \left( \id _ {{}_{\alg _ {{}_{\mathtt{y}}}}}\ast \mathfrak{t} _{{}_{Y}} \right)\enspace =\enspace
\vcenter{
\xymatrix@C=5em@R=4em{
\TTTTT Y
\ar@{=}[dd]
\ar[dr]|-{\TTTTT (f) }
\ar@{{}{ }{}}@/^2pc/[dd]|-{\xRightarrow{\widehat{\TTTTT (\rho )} } }
\ar[rr]|-{\alg _ {{}_{\mathtt{y}}} }
&
&
Y
\ar[d]|-{f  }
\ar@{}[ld]|-{\xRightarrow{ \left\langle \overline{\mathsf{f}}\right\rangle }  }
\\
&
\TTTTT Z
\ar@{}[rd]|-{\xRightarrow{\left\langle \overline{\mathsf{g}}\right\rangle } }
\ar[r]|-{\alg _ {{}_{\mathtt{z}}} }
\ar[ld]|-{\TTTTT (g) }
&
Z
\ar[d]|-{g}
\\
\TTTTT Y
\ar[rr]|-{\alg _ {{}_{\mathtt{y}}}  }
&
&
Y
}
}
\end{equation*}
under the adjunction of identities and the adjunction $(f\dashv g, \epsilon , \rho ) $ is equal to $ \left\langle \overline{\mathsf{f}}\right\rangle
\cdot \left(\diamond \left\langle \overline{\mathsf{g}}\right\rangle\right)  $ which by hypothesis is equal to $\id _{{}_{f\alg _ {{}_{\mathtt{y}}} }}$. On the other hand, the mate of $\rho \ast \id _{{}_{\alg _ {{}_{\mathtt{y}}}}} $ under the adjunction of identities and the adjunction $(f\dashv g, \epsilon , \rho ) $
is also equal to  $\id _{{}_{f\alg _ {{}_{\mathtt{y}}} }}$ by the triangle identity. Therefore, by the mate correspondence (Theorem \ref{Mate Correspondence Theorem PhD}), we conclude that the left side of the equation above is equal to   $\rho \ast \id _{{}_{\alg _ {{}_{\mathtt{y}}}}} $. Therefore
$$\xymatrix@=3em{  \TTTTT Y\ar@/_5ex/[dd]|-{\TTTTT (\id _ {{}_{Y}} ) }
                    \ar@{}[dd]|{\xRightarrow{\TTTTT(\rho ) } }
                    \ar@/^5ex/[dd]|-{\TTTTT (gf) }
										\ar[rr]^{\alg _ {{}_{\mathtt{y}}} } && 
										 Y\ar[dd]^{gf }    
&&
\TTTTT Y \ar[rr]^{ \alg _ {{}_{\mathtt{y}}} }\ar@/_3.3ex/[dd]|-{\TTTTT (\id _ {{}_{Y}} ) }  &&
Y\ar@/_3.3ex/[dd]|-{\id _ {{}_{Y}}}
                    \ar@{}[dd]|{\xRightarrow{\rho  } }
                    \ar@/^3.3ex/[dd]|-{gf }
\\
&\ar@{}[r]|{\xRightarrow{ \left\langle\overline{\mathsf{g}\mathsf{f}}\right\rangle  } }   &
 &=& 
&\ar@{}[l]|{\xRightarrow{\id _{{}_{\alg _ {{}_{\mathtt{y}}}}}\ast \mathfrak{t} _{{}_{Y}}^{-1} } }  & 
\\
 \TTTTT Y\ar[rr]_ {\alg _ {{}_{\mathtt{y}}} } &&  Y
 &&
\TTTTT Y\ar[rr]_ {\alg _ {{}_{\mathtt{y}}} } && Y	 }$$ 
which, by Remark \ref{Identidade de Lax Algebra PhD}, shows that $\rho $ satisfies Eq. 3 of Definition 5.4.1~\cite{2016arXiv160703087L}. This proves that indeed $\rho $ gives a $\TTTTT $-transformation $\widetilde {\rho }: \id _ {{}_{\mathtt{y}}}\Rightarrow \mathtt{g}\mathtt{f} $.  The proof for $\epsilon $ is analogous.
\end{proof}

\begin{coro}\label{BeckChevalley 2monadas Corolario 1 Capitulo 1 PhD}
Let $U:\mathsf{Lax}\textrm{-}\TTTTT\textrm{-}\Alg _{\ell }\to \bbb $ be the forgetful $2$-functor. Given a lax $\TTTTT $-morphism
$\mathtt{f}: \mathtt{y}\to \mathtt{z}  $:
\begin{itemize}\renewcommand\labelitemi{--}
\item $\mathtt{f} $ is left adjoint in $\mathsf{Lax}\textrm{-}\TTTTT\textrm{-}\Alg _{\ell } $ if and only if $U(\mathtt{f} ) $ is left adjoint
and $\mathtt{f} $ is a $\TTTTT $-pseudomorphism;
\item $\mathtt{f} $ is right adjoint if and only if $U(\mathtt{f} ) $ is right adjoint
and $\mathtt{f} $ satisfies the Beck-Chevalley condition.
\end{itemize}
\end{coro}

In the case of pseudomorphisms, the second condition remains equally, but, for the case of lifting of left adjoints, we still need to assure
that the right adjoint is going to be a pseudomorphism. More precisely:

\begin{coro}\label{BeckChevalley 2monadas Corolario 2 Capitulo 1 PhD}
Let $U:\mathsf{Lax}\textrm{-}\TTTTT\textrm{-}\Alg _\ell \to \bbb $ be the forgetful $2$-functor. Given a $\TTTTT $-pseudomorphism
$\mathtt{f}= (f, \left\langle \overline{\mathsf{f}}\right\rangle ): \mathtt{y}\to \mathtt{z}  $:
\begin{itemize}\renewcommand\labelitemi{--}
\item $\mathtt{f} $ is left adjoint in $\mathsf{Lax}\textrm{-}\TTTTT\textrm{-}\Alg _\ell  $ if and only if $U(\mathtt{f} ) $ is left adjoint
and $\diamond ^{-1}\left\langle \overline{\mathsf{f}}\right\rangle  $ is an invertible $2$-cell;
\item $\mathtt{f} $ is right adjoint if and only if $U(\mathtt{f} ) $ is right adjoint
and $\mathtt{f} $ satisfies the Beck-Chevalley condition.
\end{itemize}
\end{coro}

\subsection{Kock-Z\"{o}berlein pseudomonads}\label{KZ-PSEUDOMONADA CAPITULO 1 PhD}
The concept of Kock-Z\"{o}berlein doctrine was originally introduced by Kock~\cite{MR1359690} and Z\"{o}berlein~\cite{ManuscritoZoberleinKZdoctrines}. We adopt the natural extended notion of Kock-Z\"{o}berlein pseudomonad~\cite{MR1432190}, called herein lax idempotent pseudomonad.  

Furthermore, since, in our context, the most important property of a lax idempotent pseudomonad $\TTTTT$ is the fact that the forgetful
$2$-functor $\mathsf{Lax}\textrm{-}\TTTTT\textrm{-}\Alg _\ell\to\bbb $ is fully faithful,
we get a shortcut, defining lax idempotent pseudomonads via this property: 

\begin{defi}\label{Our definition of KZ}
A pseudomonad $\TTTTT $ on a $2$-category $\bbb $ is called a \textit{lax idempotent} if the forgetful $2$-functor  $\mathsf{Lax}\textrm{-}\TTTTT\textrm{-}\Alg _\ell\to\bbb $ is fully faithful (meaning that it is locally an isomorphism). 
\end{defi}

It should be noted that our definition is actually equivalent to the usual Kock-Z\"{o}berlein adjoint property as stated below: the proof of this fact for the strict case 
is originally given in \cite{MR1476422}.
Since the Kock-Z\"{o}berlein adjoint property has no important role in our observation, we avoid the proof.

\begin{prop}
A pseudomonad $(\TTTTT,  m  , \eta , \mu,   \iota, \tau )$ is lax idempotent if and only if it satisfies the \textit{Kock-Z\"{o}berlein adjoint structure property}: that is to say, there 
 is a modification $\gamma : \Id _ {{}_{\TTTTT ^2 }}\Longrightarrow (\eta \TTTTT )(m)  $ such that $(m\dashv \eta \TTTTT , \iota , \gamma ) $ is an adjunction. 
\end{prop}

In some situations, it can be easier to verify whether a pseudomonad $\TTTTT$ satisfies the Kock-Z\"{o}berlein adjoint property than to verify whether the
forgetful $2$-functor $\mathsf{Lax}\textrm{-}\TTTTT\textrm{-}\Alg _\ell\to\bbb $ is fully faithful. 
However our main observation on lax idempotent pseudomonads relies on the last property. More precisely:

\begin{rem}\label{Obvio KZPSEUDOMONADAS Capitulo 1 PhD}
Given a lax idempotent pseudomonad $\TTTTT$, with a forgetful $2$-functor $U: \mathsf{Lax}\textrm{-}\TTTTT\textrm{-}\Alg _\ell \to\bbb$,  it is clear that:
\begin{itemize}\renewcommand\labelitemi{--}
\item The forgetful $2$-functor $\mathsf{Ps}\textrm{-}\TTTTT\textrm{-}\Alg _\ell \to\bbb$ is fully faithful as well;
\item Given an object $Z$ of $\bbb $, if there is a lax $\TTTTT$-algebra $\mathtt{z}$ such that $U(\mathtt{z} ) = Z $, it is unique up to isomorphism; 
\item For every adjunction $(f\dashv U(\mathtt{g} ), \epsilon  , \rho  ): Y\to Z   $ in $\bbb $, there is an adjunction 
$(\mathtt{f}\dashv \mathtt{g} , \widetilde{\epsilon }  , \widetilde{\rho }  )$ in $\mathsf{Lax}\textrm{-}\TTTTT\textrm{-}\Alg$
such that $U(\mathtt{f} ) = f$;
\item For every adjunction $(U(\mathtt{f} )\dashv g, \epsilon  , \rho  ): Y\to Z   $ in $\bbb $, there is an adjunction 
$(\mathtt{f}\dashv \mathtt{g} , \widetilde{\epsilon }  , \widetilde{\rho }  )$ in $\mathsf{Lax}\textrm{-}\TTTTT\textrm{-}\Alg$
such that $U(\mathtt{g} ) = g$.
\end{itemize}
\end{rem}

By Corollaries \ref{BeckChevalley 2monadas Corolario 1 Capitulo 1 PhD},\ref{BeckChevalley 2monadas Corolario 2 Capitulo 1 PhD} and, by Remark \ref{Obvio KZPSEUDOMONADAS Capitulo 1 PhD}, we get that, for every lax $\TTTTT $-morphism $\mathtt{g}: \mathtt{z}\to\mathtt{y} $ such that $U(\mathtt{g}) $ is right adjoint,
$\mathtt{g}$ satisfies   the Beck-Chevalley condition. More precisely:

\begin{coro}\label{KZpseudomonadasBECKchevalley Capitulo 1 PhD}
Assume that $\TTTTT$ is a lax idempotent pseudomonad, $U:\mathsf{Lax}\textrm{-}\TTTTT\textrm{-}\Alg _{\ell }\to \bbb $ is the forgetful $2$-functor
and  $(f\dashv g, \epsilon , \rho ): U(\mathtt{y})\to U(\mathtt{z}) $ is an adjunction in $\bbb $. 
\begin{itemize}\renewcommand\labelitemi{--}
\item There is only one lax $\TTTTT $-morphism $\mathtt{f}: \mathtt{y}\to \mathtt{z} $ such that $U(\mathtt{f}) = f $. Furthermore,
$\mathtt{f} $ is a $\TTTTT $-pseudomorphism which is left adjoint in $\mathsf{Lax}\textrm{-}\TTTTT\textrm{-}\Alg _{\ell }$;
\item There is only one lax $\TTTTT $-morphism $\mathtt{g}: \mathtt{z}\to \mathtt{y} $ such that $U(\mathtt{g}) = g $. Furthermore,
$\mathtt{g} $ is right adjoint to $\mathtt{f} $ in $\mathsf{Lax}\textrm{-}\TTTTT\textrm{-}\Alg _{\ell } $ and $\mathtt{g} $ satisfies the Beck-Chevalley condition. 
\end{itemize}
\end{coro}

This shows how Kock-Z\"{o}berlein pseudomonads encompass situations when ``the Beck-Chevalley conditions always hold''. In other words, given
such a lax idempotent pseudomonad, whenever $g$ is a right
adjoint between objects in the base $2$-category that can be endowed with lax $\TTTTT $-algebra structure, the unique lax $\TTTTT $-structure for $g$ 
always satisfies the Beck-Chevalley condition.

We can now work 
on examples of lax idempotent pseudomonads. 
Besides the idempotent pseudomonads of Chapter 3~\cite{2016arXiv160604999L}, the examples we should
mention are the cocompletion pseudomonads~\cite{MR1776428, MR2211425}, which motivated the definition of Kock-Z\"{o}berlein doctrines.
Our aim is to show how the elementary result that says that left adjoints preserve colimits can
be stated in our context.

The definition of cocompletion $2$-monads for enriched categories is given in \cite{MR1776428} and, by direct verification, via 
Definition \ref{Our definition of KZ} or via the Kock-Z\"{o}berlein adjoint property, one can see that cocompletion $2$-monads are lax
idempotent pseudomonads. 
We are more interested on the particular case of 
 the Kock-Z\"{o}berlein pseudomonad of cocompletion on $\CAT $. 
In order to work out our example, 
we need the list of well known  properties of the cocompletion pseudomonad below. 
Although we do not present proofs, they can be found in any of the main references~\cite{MR651714}.

\begin{enumerate}
\item There exists a lax idempotent pseudomonad $(\PPPPP,  m  , \eta , \mu,   \iota, \tau )$ on $\CAT $ such that $\PPPPP X $ is the free cocompletion of $X$ for every
category $X$;
\item The $\PPPPP $-pseudoalgebras (as the lax $\PPPPP$-algebras) are the cocomplete categories;
\item Clearly, since $\PPPPP $ is lax idempotent, the lax $\PPPPP $-morphisms are just functors between cocomplete categories;
\item If $\mathtt{f}= (f, \left\langle \overline{\mathsf{f}}\right\rangle ): \mathtt{y}\to \mathtt{z}  $ is a lax $\PPPPP $-morphism, $ \left\langle \overline{\mathsf{f}}\right\rangle $ is given by the natural comparisons of the colimit of the image and the image of the colimit of diagrams. In particular, the  $\PPPPP $-pseudomorphisms are exactly the cocontinuous functors (which means that they preserve all the colimits).
\end{enumerate}

By Corollary \ref{KZpseudomonadasBECKchevalley Capitulo 1 PhD}, we get that right adjoint functors between cocomplete categories always satisfy the
Beck-Chevalley condition w.r.t. the pseudomonad $\PPPPP $. Or, in other words, left adjoints between cocomplete categories always induce $\PPPPP $-pseudomorphisms.

\begin{coro}\label{Left Adjoint Capitulo 1 PhD}
Left adjoint functors between cocomplete categories are cocontinuous.
\end{coro} 
 
This shows how our result on Beck-Chevalley condition for Kock-Z\"{o}berlein pseudomonads can be seen as a generalization of this elementary result.
Of course, left adjoints in general preserve colimits. But we can see this fact as a consequence of Corollary \ref{Left Adjoint Capitulo 1 PhD}.
More precisely,
given an adjunction $(\widetilde{f}\dashv \widetilde{g} ) : Y\to Z $ in $\CAT $, by Yoneda embedding, there is  
an extension  $$(\underline{\widetilde{f}}\dashv \underline{\widetilde{g} } ): \left( \CAT \left[ Y, \SET \right]\right) ^\op\to 
\left( \CAT \left[ Z, \SET \right]\right) ^\op $$ which commutes up to isomorphism with the Yoneda embeddings. This means in particular
that we have an invertible natural transformation
\begin{equation*}
\xymatrix@=3em{
\left( \CAT \left[Y, \SET \right]\right) ^\op
\ar[rr]^-{\underline{\widetilde{f}}  }
&\ar@{}[d]|-{\cong }&
\left( \CAT \left[ Z , \SET \right]\right) ^\op
\\
Y
\ar[rr]_-{\widetilde{f} }
\ar[u]^-{\YYYY _{Y^\op } ^\op }
&&
Z
\ar[u]_-{\YYYY _{Z^\op } ^\op }
}
\end{equation*}
 which completes the argument that $\widetilde{f}$ preserves colimits, since the Yoneda embeddings $\YYYY _{Y^\op } ^\op, \YYYY _{Z^\op } ^\op $ preserve and reflect colimits and 
$\underline{\widetilde{f}} $ preserves colimits by Corollary \ref{Left Adjoint Capitulo 1 PhD}.

Of course, the usual argument works very similarly. That is to say, by Remark \ref{Adjuncoes em CAT Capitulo 1 PhD}, 
$$\CAT \left[ \widetilde{f} ^\op, \SET \right] \dashv  \CAT \left[ \widetilde{g} ^\op, \SET \right] .$$
We have that $\CAT \left[ \widetilde{f} ^\op , \SET \right] $ preserves limits, since the the limits of  
$\CAT \left[ Y ^\op , \SET \right], \CAT \left[ Z^\op, \SET \right] $ are constructed pointwise.
By the natural isomorphism \ref{phi PhD} of Remark \ref{Adjuncoes em CAT Capitulo 1 PhD}, since 
 $\YYYY _{Y }, \YYYY _{Z }  $
preserve and reflect limits, we conclude that $\widetilde{g} $ preserves limits as well.

\section{Pseudo-Kan Extensions and Descent Theory}\label{Capitulo3 PhD}

Chapter 3  is the article  
\textit{Pseudo-Kan Extensions and descent theory}~\cite{2016arXiv160604999L}, under review.
We give a formal approach to descent theory, framing classical descent theory in the context of idempotent pseudomonads. Within this perspective, we recast
and generalize 
most of the classical results of the context of \cite{MR1466540, MR1285884}, including transfer results, embedding results and the B\'{e}nabou-Roubaud Theorem.

The chapter starts by giving an outline of the setting, presenting basic problems and results of the classical context of descent theory. 
We give an outline of the classical results that are proved and generalized in that paper, including the results mentioned above.

In Section 2 of Chapter 3~\cite{2016arXiv160604999L}, we prove 
theorems on pseudoalgebra structures and biadjoint triangles, always focusing in the case of idempotent pseudomonads. The main advantages
on focusing our study on idempotent pseudomonads are the following: the pseudoalgebra structures w.r.t. an idempotent pseudomonad 
 are easier to study. In this case, 
if a pseudoalgebra structure over an object $X$ exists, it is unique  (up to isomorphism) and, moreover, the pseudoalgebra structure over an object $X$
exists if and only if the unit of the pseudomonad on  $X$ is an equivalence. This fact allows us to study situations when we ``almost have'' a pseudoalgebra structure over an object $X$, which
correspond to the situations when the component of the unit on $X$ is faithful or fully faithful. This is important, later, to study
descent and almost descent morphisms.

The results on pseudoalgebra structures and biadjoint triangles give the formal account
to study descent theory. In order to study classical descent theory in the context of \cite{MR1466540}, the first step 
was to give results on commutativity of bilimits. More precisely, we firstly give a direct approach to prove 
an analogue of Fubini's Theorem
for weighted bilimits. This allows us to construct pointwise pseudo-Kan extensions and prove the basic results about them. Secondly, 
since  we prove that
$$\left[\t , \bbb\right]_{PS}: \left[\dot{\aaa } , \bbb\right]_{PS}\to \left[ \aaa , \bbb \right]  $$
is pseudomonadic and induces an idempotent pseudomonad whenever $\t $ is locally fully faithful and $\dot{\aaa } $ is a small $2$-category,
we are able to get results on commutativity of bilimits as direct consequences of our results  
on pseudomonadic pseudofunctors. 

Section 4 of Chapter 3~\cite{2016arXiv160604999L} introduces the descent objects, giving key results to finally frame the classical context of descent theory. The main 
result of this section is that the conical bilimit of a pseudocosimplicial object is its descent object. In other words, it shows that our definition
of descent object coincides with the usual definition (as, for instance, given in \cite{MR574662}). More concisely, within the language of pseudo-Kan extensions,
we prove that $$\Ps\Ran _ \j \AAA (\mathsf{0}) \simeq \Ps\Ran _ {\j _{{}_{3}} }\AAA \t _ {{}_3} $$ in which 
$\AAA : \Delta \to\hhh $ is a pseudofunctor, $\hhh $ is a bicategorically complete $2$-category, $\j : \Delta \to\dot{\Delta } $ is the full inclusion of the category of finite nonempty ordinals 
into the category of finite ordinals and order preserving functions, $\t _ {{}_3} : \Delta _ {{}_{3}} \to \Delta  $ 
is the inclusion of the $2$-category given by the faces and degeneracies 
$$\xymatrix{  \mathsf{1}\ar@<1.7 ex>[rrr]^-{d^0}\ar@<-1.7ex>[rrr]_-{d^1} &&& \mathsf{2}\ar[lll]|-{s^0}
\ar@<1.7 ex>[rrr]^{\partial ^0}\ar[rrr]|-{\partial ^1}\ar@<-1.7ex>[rrr]_{\partial ^2} &&& \mathsf{3} }$$
into $\Delta $, and $\j _{{}_{3}}$ is the inclusion of  $\Delta _ {{}_{3}}$ into the $2$-category
$$\xymatrix{  \mathsf{0} \ar[rr]^-d && \mathsf{1}\ar@<1.7 ex>[rrr]^-{d^0}\ar@<-1.7ex>[rrr]_-{d^1} &&& \mathsf{2}\ar[lll]|-{s^0}
\ar@<1.7 ex>[rrr]^{\partial ^0}\ar[rrr]|-{\partial ^1}\ar@<-1.7ex>[rrr]_{\partial ^2} &&& \mathsf{3}}$$
which is the $2$-category obtained from the addition of an initial object to $\Delta _ {{}_{3}}$.
This proves in particular that the definition of descent category via biased descent data on objects, which corresponds to $\Ps\Ran _ {\j _{{}_{3}} }\AAA \t _ {{}_3}$,  is equivalent to the definition of the descent category via unbiased descent data on objects, which corresponds to the case $\Ps\Ran _ \j \AAA (\mathsf{0})$. But the main points of this result are (1) this gives a very simple universal property of the descent category/object and (2) this gives a way of easily fitting 
the descent object in our language.

After this detailed work on descent objects, we turn to elementary and known examples. 
The Eilenberg-Moore objects and the monadicity of functors also
fit easily in our context of weighted bilimits/pseudoalgebra structures, once we follow the ideas of \cite{MR845410}. This is explained in Section 6 of Chapter 3.

Finally, in Section 7 of Chapter 3, we show how our perspective on the Beck-Chevalley condition (as explained in Section \ref{Doctrinal Adjunction PhD}) 
allows us to get results on 
pseudoalgebra structures/commmutativity of bilimits and monadicity. This leads to our first result of the type of B\'{e}nabou-Roubaud Theorem. After that, 
we finally show how our results work in the context of classical descent theory. We recast and generalize classical results as direct consequences of our
previous work.

We then give refinements of our results on commutativity of bilimits. In the context of descent theory, this allows us to give better results
on effective descent morphisms of weighted bilimits of $2$-categories. It also gives the Galois result of \cite{MR1245796} as a direct consequence.

One particular result obtained from our setting of commutativity of bilimits  is the pseudopullback theorem. It gives conditions to
get effective descent morphisms (w.r.t. basic fibration) of well behaved pseudopullbacks of categories.
We finish this chapter applying this result to 
detect effective descent morphisms
in categories of enriched categories. Firstly, we prove that, for suitable cartesian categories $V$, we have an embedding $V\textrm{-}\Cat\to \Cat (V) $
that is actually induced by a pseudopullback of categories. Then, using the pseudopullback theorem, we prove that such embedding 
reflects effective descent morphisms.

\section{Biadjoint Triangles and Lifting of Biadjoints}\label{Capitulo4 PhD}
Chapter 4 corresponds to the 
article \textit{On Biadjoint Triangles}~\cite{MR3491845},  published in \textit{Theory and Applications of Categories, Vol 31, N. 9} (2016).
The main contributions are the biadjoint triangle theorems, which  can be seen as $2$-dimensional analogues of the adjoint triangle theorem of \cite{MR0233864}.
As mentioned in Section \ref{Overview PhD},  in order to prove the main results, we use the fact that the category of pseudomorphisms between two pseudoalgebras has the universal property of the descent object.  
More precisely, assuming that
\[\xymatrix{  \aaa\ar[rr]^-{J}\ar[dr]_-{R}&&\bbb\ar[dl]^-{U}\\
&\ccc\ar@{}[u]|-\simeq & }\] 
is a pseudonatural equivalence, in which $R$, $J$, $U $ are pseudofunctors, $U$ is pseudopremonadic and $R$ has a right biadjoint. We prove that $J$ has a left biadjoint $G$, provided that $\aaa$ has some needed codescent objects. We also study the unit and the counit of the obtained biadjunction: we give sufficient 
conditions under which the unit and counit  are pseudonatural equivalences. Finally, we show that, under suitable conditions, it is possible to construct a (strict) left 2-adjoint.

Similarly to the case of adjoint triangles in $1$-dimensional category theory, the biadjoint triangles 
have many 
applications in 
$2$-dimensional category theory. 
Examples of which are
given in this same paper: 
\begin{itemize} \renewcommand\labelitemi{--}
 \item \textit{Pseudomonadicity characterization}: without avoiding pseudofunctors,
 using the biadjoint triangle theorem and the results on counit and unit, 
 we give a explicit proof of the pseudomonadicity 
 characterization 
 due to  
 Le Creurer, Marmolejo and Vitale in \cite{MR1916482};
 \item \textit{$2$-monadic approach to coherence}: as immediate consequences of our main theorems,
 we recast and improve results on the $2$-monadic approach 
 to coherence developed in \cite{MR1916482, MR985657, MR1007911, MR1935980}. 
 More precisely, we characterize when there is a left $2$-adjoint 
 (and biadjoint)
 to the inclusion of the $2$-category of strict algebras into the $2$-category of pseudoalgebras, 
$$\TTTTT\textrm{-}\Alg _ {\textrm{s}}\to
\mathsf{Ps}\textrm{-}\TTTTT\textrm{-}\Alg $$ of 
 a given $2$-monad $\TTTTT$. We also characterize when the unit of such biadjunction is a pseudonatural equivalence;
 \item \textit{Lifting of biadjoints}: the biadjoint triangle theorem gives biadjoints
 to algebraic pseudofunctors, that is to say, lifting of biadjoints. These results
 recover and generalize, for instance, results of 
 Blackwell, Kelly and Power~\cite{MR1007911}. 
 \item \textit{Pointwise pseudo-Kan extension}: We originally introduce the notion of pseudo-Kan extension
 and,
 using the results on lifting of biadjoints, in the presence of weighted bilimits, we
 construct pseudo-Kan extensions with them. 
This result, hence, gives the notion of pointwise 
 pseudo-Kan extension. It also gives a way of recovering the construction of weighted bilimits via
 descent objects, cotensor (bi)products and (bi)products given originally in 
\cite{MR574662}, if we assume the construction of the pointwise pseudo-Kan extension via Fubini's Theorem for weighted bilimits given in Chapter 3.

Similarly to the pointwise Kan extension in $1$-dimensional category theory, 
the concept of pointwise pseudo-Kan extension plays a relevant  role in $2$-dimensional category theory.
As mentioned in Section \ref{Capitulo3 PhD}, one instance of application of this concept given in this thesis is within the
 study of commutativity of weighted bilimits of Chapter 3. Other examples are
within the study of $2$-dimensional flat pseudofunctors of \cite{2016arXiv161009429D} and within the study of formal aspects of $2$-dimensional category theory via $\mathsf{Gray}$-categories~\cite{2017arXiv170704074D}.
\end{itemize}

Chapter 5 corresponds to the 
article \textit{On lifting of biadjoints and lax algebras}~\cite{2016arXiv160703087L}, to appear in  
\textit{Categories and General Algebraic Structures with Applications}. 
It gives further theorems on lifting of biadjoints 
provided that we can describe the categories
of morphisms of a certain domain  $2$-categories in terms of weighted (bi)limits.
This gives an abstract account of the main idea of some proofs of Chapter 4. Still,
we show that this setting allows us to get results outside of the context of Chapter 4.

In particular, this approach, together with 
results on lax descent objects and lax algebras, allows us to give results on lifting of biadjoints
involving (full) sub-$2$-categories of the $2$-category of lax algebras. This gives biadjoint triangle theorems
involving the $2$-category of lax algebras.
As a immediate consequence, we complete our
treatment of the $2$-monadic approach to coherence via biadjoint triangle theorems.  

\begin{rem}\label{Faltando counidade}
Unlike in the case of Chapter 5, our results on biadjoint triangles involving the $2$-category of lax algebras, Theorems 5.5.2 and .5.5.3~\cite{2016arXiv160703087L}, lack
 the study of the counit and the unit of the obtained 
biadjunctions. The study of counit and the unit in this setting could
lead to new applications. One example is given in Remark \ref{Eu nao pretendo fazer mas deixo como trabalho futuro}.
\end{rem}

\section{Lifting of Biadjoints and Formal Theory of Monads}\label{Teoria Formal das Monadas e levantamentos PhD}

In this section, we talk about applications of the results of Chapter 5 in the 
context of the formal theory of monads. In order to do so, we assume most of the prerequisites of that chapter, including 
the concept of weighted limits and colimits in a $2$-category w.r.t. the $\Cat $-enrichment, usually called $2$-limits and $2$-colimits. In this direction, we adopt the terminology and definitions
of Section 1 of Chapter 2~\cite{2017arXiv170404474L}.

Every adjunction induces a monad. This was originally shown in \cite{MR0150184} for the $2$-category $\Cat$. However it works 
for any $2$-category. Indeed, given an adjunction $(f\dashv g , \epsilon , \eta ): Y\to Z $ in a $2$-category $\bbb $,
$(Y, gf , \id _ {{}_{g}}\ast \epsilon\ast \id _ {{}_{f}} , \eta ) $ is a monad on $Y$, that is to say, a lax $\Id _ {{}_{\bbb }} $-algebra structure on $Y$.

\begin{rem}
Given a monad $\mathtt{y}= (Y, \alg _ {{}_{\mathtt{y}}}, \overline{{\underline{\mathtt{y}}}}, 
\overline{{\underline{\mathtt{y}}}}_0 )$ in a $2$-category $\bbb $, we define the category of $\mathtt{y}$-adjunctions $\mathtt{y}\textrm{-}\mathsf{adj} (\bbb ) $ as follows:
\begin{itemize}\renewcommand\labelitemi{--}
	\item The objects of $\mathtt{y}\textrm{-}\mathsf{adj}(\bbb ) $  are adjunctions $(f\dashv g , \epsilon , \eta ): Y\to Z $ that induce $\mathtt{y}$;
	\item A morphism of $\mathtt{y}\textrm{-}\mathsf{adj}(\bbb ) $ between two adjunctions $(f\dashv g , \epsilon , \eta ): Y\to Z $ and $(\widetilde{f}\dashv \widetilde{g} , \widetilde{\epsilon } , \eta ): 
	Y\to \widetilde{Z} $ is a morphism $j : Z\to \widetilde{Z} $ such that  $jf = \widetilde{f}  $ and $g = \widetilde{g} j $.
\end{itemize}
If there exists, the terminal object of $\mathtt{y}\textrm{-}\mathsf{adj}(\bbb ) $ is called an \textit{Eilenberg-Moore adjunction} for the monad $\mathtt{y} $.
In this case, the domain of the
right adjoint of such adjunction is called the \textit{Eilenberg-Moore object of  $\mathtt{y}$} and denoted by  $Y^\mathtt{y}$. 
Dually, the initial object (if it exists) is called the Kleisli adjunction for the monad $\mathtt{y} $. In this case, 
the domain of the right adjoint of such adjunction is called the \textit{Kleisli object of  $\mathtt{y}$} and denoted by  $Y_\mathtt{y}$. 
\end{rem}

There is an Eilenberg-Moore adjunction and a Kleisli adjunction for each monad in $\Cat $.
These results were shown respectively in \cite{MR0184984} and \cite{MR0177024}. 
However, it is easy to construct counterexamples of $2$-categories not having all the Eilenberg-Moore (or Kleisli) adjunctions.
In order to give  a non-artificial 
easy to check example, we consider bicategories. More precisely, we can consider
the suspension of the monoidal cartesian category $(\Set , \times , 1 ) $, that is to say, we see such a monoidal category as a bicategory with only one object $\triangle $
as in Remark \ref{Suspension Categorias Monoidais PhD Capitulo 1}. A monad
in such bicategory
is the same as a (classical) monoid. There are plenty nontrivial monoids, while the suspension of $(\Set , \times , 1 ) $ has only the trivial adjunction.

\begin{rem}
Clearly there is a bijection between monads  in $\bbb $ and monads
in $\bbb ^\op $. More precisely, the contraviariant $2$-functor $\bbb\to\bbb ^\op $ takes monads in $\bbb $ to monads in $\bbb ^\op $.
So, by abuse of language, if $\mathtt{y} $ is a monad in $\bbb $, we denote by $\mathtt{y} $ the corresponding monad in $\bbb ^\op $.

We can, then, give precise meaning to the fact that the notion of Eilenberg-Moore objects is dual to the notion of Kleisli objects.
Indeed, the Kleisli object for a monad $\mathtt{y} $ in a $2$-category
$\bbb $ is, if it exists,  the Eilenberg-Moore object of $\mathtt{y} $ in $\bbb ^\op $. 

\end{rem}

In  \cite{MR0347936}, it is observed that the Eilenberg-Moore object has a concise universal property. 
Namely, given a $2$-category $\bbb $,
there is an inclusion $2$-functor $\bbb\to \mathsf{Mnd} (\bbb )    $  
which takes each object $Z$ to the monad $(Z, \id _ {{}_{Z}}, \id _ {{}_{\id _ {{}_{Z}} }}, 
\id _ {{}_{\id _ {{}_{Z}} }} )$. The Eilenberg-Moore object of a 
monad  $\mathtt{y}= (Y, \alg _ {{}_{\mathtt{y}}}, \overline{{\underline{\mathtt{y}}}}, 
\overline{{\underline{\mathtt{y}}}}_0 )$ is given by the right $2$-reflection of $\mathtt{y} $ along $\bbb\to \mathsf{Mnd} (\bbb )$, if it exists.
In particular,  a morphism $f: X\to Y ^\mathtt{y} $ corresponds to a pair  $\mathtt{f} = (\widetilde{f}, \left\langle \overline{\mathsf{f}}\right\rangle ) $
in which $ \widetilde{f}: X\to Y $ is a morphism and $\left\langle   \overline{\mathsf{f}}\right\rangle : \alg _ {{}_{\mathtt{y}}}\widetilde{f} \Rightarrow \widetilde{f}   $ 
is a $2$-cell, such that the equations
\begin{equation*}
\vcenter{\xymatrix@R=3em{
Y
\ar[rr]|-{\alg _ {{}_{\mathtt{y}}} }
&&
Y
\\
\ar@{}[ru]|-{\xRightarrow{\overline{{\underline{\mathtt{y}}}} }}&&\ar@{}[ld]|-{\xRightarrow{\left\langle   \overline{\mathsf{f}}\right\rangle }}
\\
Y
\ar[rruu]|-{\alg _ {{}_{\mathtt{y}}}}
\ar[uu]|-{\alg _ {{}_{\mathtt{y}}}}
&&
X
\ar[ll]|-{\widetilde{f}}
\ar[uu]|-{\widetilde{f}}
}} 
\, =\,    
\vcenter{\xymatrix@=3em{
Y
\ar[rr]|-{\alg _ {{}_{\mathtt{y}}}}
&&
Y
\\
\ar@{}[rd]|-{\xRightarrow{\left\langle   \overline{\mathsf{f}}\right\rangle }}&&\ar@{}[lu]|-{\xRightarrow{\left\langle   \overline{\mathsf{f}}\right\rangle }}
\\
Y
\ar[uu]|-{\alg _ {{}_{\mathtt{y}}}}
&&
X
\ar[uu]|-{\widetilde{f}}
\ar[lluu]|-{\widetilde{f} }
\ar[ll]|-{\widetilde{f}}
}}
\quad\mbox{and}\quad
\vcenter{\xymatrix@R=3em{
Y
\ar@{=}[rr]
&&
Y
\\
\ar@{}[ru]|-{\xRightarrow{\overline{{\underline{\mathtt{y}}}}_0 }}&&\ar@{}[ld]|-{\xRightarrow{\left\langle   \overline{\mathsf{f}}\right\rangle }}
\\
Y
\ar[rruu]|-{\alg _ {{}_{\mathtt{y}}}}
\ar@{=}[uu]
&&
X
\ar[ll]|-{\widetilde{f}}
\ar[uu]|-{\widetilde{f}}
}} 
\,
=
\,
\vcenter{
\xymatrix@R=3em{
Y
\ar@{}[dd]|-{\xRightarrow{\id _ {{}_{\widetilde{f} }}} }
\ar@/^3ex/[dd]|-{\widetilde{f} }
\ar@/_3ex/[dd]|-{\widetilde{f} }
\\
\\
X
}
}
\end{equation*}
hold.
Since the Kleisli object of  $\mathtt{y} $
is the 
 Eilenberg-Moore object of $\mathtt{y} $ in $\bbb ^\op $, we have that the Kleisli object is given by the right $2$-reflection
 of $\mathtt{y} $  along  $\bbb ^\op\to \mathsf{Mnd} (\bbb ^\op )    $ 
which is the same as the left $2$-reflection of $\bbb \to \left(\mathsf{Mnd} (\bbb ^\op ) \right) ^\op\cong \mathsf{Mnd} _ \co (\bbb ) $.

Moreover, \cite{MR0347936} generalizes the Eilenberg-Moore and the Kleisli constructions. More precisely, 
if $X$ is a category, \cite{MR0347936} constructs the right $2$-adjoint to 
the inclusion
$\left[ X, \Cat\right]\to \left[ X, \Cat\right] _ {Lax} $ of the $2$-category of lax functors $X\to\Cat $, lax natural transformations and modifications into
the $2$-category of $2$-functors, $2$-natural transformations and modifications.
In \cite{MR0347936}, Street also constructs the left $2$-adjoint 
to the inclusion  $\left[ X, \Cat\right]\to \left[ X, \Cat\right] _ {{Lax}_\mathsf{c}} $,
in which
$\left[ X, \Cat\right] _ {{Lax}_\mathsf{c}} $ denotes the $2$-category of lax functors, colax natural transformations and modifications.

In order to verify that these $2$-adjoints actually are generalizations of the Eilenberg-Moore and Kleisli objects, we should observe that,
considering the inclusions $\Cat\to \mathsf{Mnd} _ \co (\Cat ) $ and $\Cat\to \mathsf{Mnd} (\Cat )    $, 
  we actually have isomorphisms $\mathsf{Mnd} (\Cat )\cong \left[ \mathsf{1}, \Cat\right] _ {Lax} $ and $\mathsf{Mnd}_\co (\Cat )\cong \left[ \mathsf{1}, \Cat\right]_ {{Lax}_\mathsf{c}} $   such that the diagrams 
\begin{equation*}
\vcenter{\xymatrix@R=3.5em{
\Cat\ar[r]|-{\cong }\ar[d]
&\left[ \mathsf{1}, \Cat\right] \ar[d]
\\ 
\mathsf{Mnd} (\Cat )
\ar[r]|-{\cong } 
&\left[ \mathsf{1}, \Cat\right] _ {Lax}
}}
\quad
\vcenter{\xymatrix@R=3.5em{
\Cat\ar[r]|-{\cong }\ar[d]
&\left[ \mathsf{1}, \Cat\right]
\ar[d]
\\ 
\mathsf{Mnd} _ \co (\Cat )
\ar[r]|-{\cong } 
&\left[ \mathsf{1}, \Cat\right] _ {{Lax}_\mathsf{c}}
}}
\end{equation*}
commute. More generally, in our context, this is given by the fact that, given a $2$-category $\bbb $, the diagrams
\begin{equation*}
\vcenter{
\xymatrix@=3.5em{
\mathsf{Lax}\textrm{-}\Id _ {{}_{\bbb }} \textrm{-}\CoAlg _{\mathsf{c}\ell }
\ar[dr]
&
\Id _ {{}_{\bbb }}\textrm{-}\CoAlg _ {\textrm{s}}\cong\Id _ {{}_{\bbb }}\textrm{-}\Alg _ {\textrm{s}}
\ar[d]
\ar[l]
\ar[r]
&
\mathsf{Lax}\textrm{-}\Id _ {{}_{\bbb }} \textrm{-}\Alg _{\mathsf{c}\ell }
\ar[dl]
\\
&\bbb &
}}
\end{equation*}
\begin{equation*}
\vcenter{\xymatrix@=3.5em{
\left[ \mathsf{1}, \bbb\right] _ {Lax}
\ar[dr]
&
\left[ \mathsf{1}, \bbb\right]
\ar[d]
\ar[l]
\ar[r]
&
\left[ \mathsf{1}, \bbb\right] _ {{Lax}_\mathsf{c}}
\ar[dl]
\\
&\bbb &
}}
\end{equation*}
are naturally isomorphic, in which the horizontal arrows are the obvious inclusions while the non-horizontal arrows are the forgetful $2$-functors. In particular, the inclusion $\left[ \mathsf{1}, \bbb\right]\to\left[ \mathsf{1}, \bbb\right] _ {Lax} $   is actually the inclusion 
$\Id _ {{}_{\bbb }}\textrm{-}\CoAlg _ {\textrm{s}}\cong\Id _ {{}_{\bbb }}\textrm{-}\Alg _ {\textrm{s}}\to \mathsf{Lax}\textrm{-}\Id _ {{}_{\bbb }} \textrm{-}\CoAlg _{\mathsf{c}\ell }$.

More generally, if $\aaa $ is any $2$-category, denoting by $\aaa _ 0  $ the discrete $2$-category of objects, 
the inclusion $\aaa _ 0\to \aaa $ induces a restriction $2$-functor $\left[ \aaa , \bbb\right] \to \left[ \aaa _ 0 , \bbb \right] $. 
If $\bbb $ has suitable weighted limits and colimits and $\aaa $ is small, this restriction has right and left $2$-adjoints given by the (global) pointwise right and left Kan extensions.
Assuming that
this restriction has right and left $2$-adjoints, we have a $2$-monad $\overline{\Lan } $ and a $2$-comonad $\overline{\Ran } $ on the $2$-category
 $\left[ \aaa _ 0 , \bbb \right]$. In this case, the diagrams
\begin{equation*}
\vcenter{
\xymatrix@=4em{
\mathsf{Lax}\textrm{-}\overline{\Ran }  \textrm{-}\CoAlg _{\mathsf{c}\ell }
\ar[dr]
&
\overline{\Ran }\textrm{-}\CoAlg _ {\textrm{s}}\cong\overline{\Lan }\textrm{-}\Alg _ {\textrm{s}}
\ar[d]
\ar[l]
\ar[r]
&
\mathsf{Lax}\textrm{-}\overline{\Lan }  \textrm{-}\Alg _{\mathsf{c}\ell }
\ar[dl]
\\
&\left[ \aaa _ 0 , \bbb \right] &
}}
\end{equation*}
\begin{equation*}
\vcenter{\xymatrix@=4em{
\left[ \aaa , \bbb\right] _ {Lax}
\ar[dr]
&
\left[ \aaa , \bbb\right]
\ar[d]
\ar[l]
\ar[r]
&
\left[ \aaa , \bbb\right] _ {{Lax}_\mathsf{c}}
\ar[dl]
\\
&\left[ \aaa _ 0 , \bbb \right] &
}}
\end{equation*}
are naturally isomorphic. These observations immediately show how the results of Chapter 5  generalize the construction: they actually characterize when it is possible to 
get such constructions. More precisely, using the techniques of that chapter, we are able to study the existence of the 
right $2$-adjoint to $\SSSSS\textrm{-}\CoAlg _ {\textrm{s}}\to\mathsf{Lax}\textrm{-}\SSSSS \textrm{-}\CoAlg _{\mathsf{c}\ell }  $ 
for any given $2$-comonad $\SSSSS $, or, equivalently, a left $2$-adjoint to 
$\TTTTT\textrm{-}\Alg _ {\textrm{s}}\to \mathsf{Lax}\textrm{-}\TTTTT  \textrm{-}\Alg _{\mathsf{c}\ell }$ for any given $2$-monad $\TTTTT $. So it is clear that
this generalizes the constructions of \cite{MR0347936}. Since we do not explicitly deal with colax morphisms in Chapter 5, we briefly
describe below how we get the results for our context. We omit most of the proofs, since some of them are slight variations of the proofs on lax morphisms of Chapter 5, while the rest of the proofs follow directly from results of that chapter.

\begin{defi}[$\t ^\mathsf{c} : \Delta _\ell ^\mathsf{c}\to \dot{\Delta } _ \ell ^\mathsf{c}$]\label{deltacolaxmorphisms PhD}
We denote by $\dot{\Delta  }_\ell^\mathsf{c} $ the $2$-category generated by the diagram
$$\xymatrix{  \mathsf{0} \ar[rr]^-d && \mathsf{1}\ar@<1.7 ex>[rrr]^-{d^0 }\ar@<-1.7ex>[rrr]_-{d^1 } &&& \mathsf{2}\ar[lll]|-{s^0}
\ar@<1.7 ex>[rrr]^{\partial ^0 }\ar[rrr]|-{\partial ^1}\ar@<-1.7ex>[rrr]_{\partial ^2} &&& \mathsf{3} }$$ 
with the $2$-cells:
\begin{equation*}
\begin{aligned}
\sigma_{00} &:&  \partial^{0}d^{0}\Rightarrow \partial ^1 d ^0, \\
\sigma_{01} &:&  \partial ^0 d ^1\Rightarrow\partial^{2}d^{0}, \\
\sigma_{21} &:&  \partial ^2 d ^1\Rightarrow\partial^{1}d^{1},
\end{aligned}
\qquad
\begin{aligned}
 n_0        &:&  \id _{{}_\mathsf{1}}\Rightarrow s^0d^0 ,  \\
 n_1        &:&   \id _{{}_\mathsf{1}}\Rightarrow s^{0}d^{1}, \\
 \vartheta        &:&  d^0d\Rightarrow d^1d,
\end{aligned}
\end{equation*}
satisfying
\begin{itemize}
\item Associativity:
\begin{equation*}
\vcenter{
\xymatrix@C=3.2em{  
\mathsf{0}
\ar[r]^-{d}
\ar[dd]_{d}\ar@{}[rdd]|-{\xRightarrow {\vartheta } } 
&
\mathsf{1}\ar[dd]^-{d^1}\ar@{}[rrddd]|{=}
&
&
\mathsf{3}
&
&
\mathsf{2}\ar[ll]_-{\partial ^1}
\ar@{}[ld]|-{\xRightarrow{\sigma_{21} }}
\\
&
&
&
&
\mathsf{2}
\ar@{}[l]|-{\xRightarrow{\sigma_{01} } }
\ar[lu]|-{\partial ^2}
&
\mathsf{1}\ar[l]^-{d^1}
\ar[u]_-{d^0}
\\
\mathsf{1} 
\ar[r]^-{ d^0 }
\ar[d]_{d^0 }\ar@{}[rd]|-{\xRightarrow {\sigma_{00}  } }
&
\mathsf{2}\ar[d]^-{\partial ^1} 
&
&
\mathsf{2}\ar[uu]^{\partial ^0}\ar@{}[rd]|-{\xRightarrow {\vartheta }}
&
\mathsf{1}\ar[l]_-{d^1 }\ar@{}[r]|-{\xRightarrow{\,\,\vartheta\,\, } }\ar[u]_{d^0}
&
\\
 \mathsf{2}\ar[r]_{\partial ^0 }
&
\mathsf{3}
&
&
\mathsf{1}\ar[u]^-{d^0}
&
\mathsf{0}\ar[l]^-{d}\ar[u]_-{d}\ar@/_5ex/[ruu]|-{d}
&
}}
\end{equation*}
\item Identity:
\begin{equation*}
\vcenter{\xymatrix@=3.2em{
\mathsf{0}\ar[r]^{d }\ar[d]_{d}\ar@{}[dr]|-{\xRightarrow{\vartheta  } }
&
\mathsf{1}\ar[d]^{d^1}\ar@{}[ddrr]|-{=}
&
&
\mathsf{0}\ar[d]_{d } 
&
\\
\mathsf{1}\ar[r]_{d^0 }\ar@{ { } }@/_1ex/[rd]|-{\xRightarrow{n_0} }\ar@{=}@/_5ex/[rd]
&
 \mathsf{2} \ar[d]^{ s^0}
&
&
\mathsf{1}\ar[r]_{d^1  }\ar@{ { } }@/_1ex/[rd]|-{\xRightarrow{n_1} }\ar@{=}@/_5ex/[rd]
&
\mathsf{2}\ar[d]^{ s^0}
\\
&
\mathsf{1}
&
&
&
\mathsf{1}
}}
\end{equation*}
\end{itemize}
The $2$-category $\Delta _\ell ^\mathsf{c}$ is, herein, the full sub-$2$-category of $\dot{\Delta } _ \ell ^\mathsf{c}$ with objects $\mathsf{1} $, $\mathsf{2} $ 
and $\mathsf{3} $. We denote the inclusion by $\t ^\mathsf{c} : \Delta _\ell ^\mathsf{c}\to \dot{\Delta } _ \ell ^\mathsf{c}$.
\end{defi}

\begin{rem}[Colax descent object and category]
If $\AAAA : \Delta _\ell ^\mathsf{c}\to \bbb $ and $\BBBB : \left( \Delta _\ell ^\mathsf{c}\right) ^\op \to \bbb $ are $2$-functors, if it exists, the weighted limit  $\left\{ \dot{\Delta } _\ell ^\mathsf{c}(\mathsf{0}, \t ^\mathsf{c} - ) , \AAAA\right\} $ is called the \textit{strict colax descent object} of $\AAAA $, while the
weighted colimit $\dot{\Delta } _\ell ^\mathsf{c}(\mathsf{0}, \t ^\mathsf{c} - ) \Asterisk \BBBB $ is called the \textit{strict colax codescent object} of $\BBBB $ (if it exists).

In the case of strict colax descent categories, we have a result similar to that described by Remark 5.3.4~\cite{2016arXiv160703087L}. More precisely, if
$\DDDD : \Delta _ {\ell }\to\Cat $
is a $2$-functor, then
$$\left\{ \dot{\Delta }_\ell ^\mathsf{c} (\mathsf{0}, \t ^\mathsf{c} -), \DDDD \right\} \cong \left[ \Delta _\ell ^\mathsf{c} , \Cat\right] \left( \dot{\Delta } _\ell ^\mathsf{c} (\mathsf{0}, \t ^\mathsf{c} -), \DDDD \right) .$$
Thereby, we can describe the strict colax descent object of $\DDDD :\Delta _\ell ^\mathsf{c}\to\Cat $ explicitly as follows:
\begin{enumerate}
\item Objects are $2$-natural transformations $\mathtt{f}: \dot{\Delta } _ \ell ^\mathsf{c} (\mathsf{0}, \t ^\mathsf{c} -)\longrightarrow \DDDD $. We have a bijective correspondence between such $2$-natural transformations and pairs $(f,\left\langle  \overline{\mathsf{f}}\right\rangle)$ in which $f$ is an object of $ \DDDD\mathsf{1} $ and $\left\langle  \overline{\mathsf{f}}\right\rangle: \DDDD (d^0)f\to \DDDD (d^1)f $ is a morphism in $ \DDDD\mathsf{2} $ satisfying the following equations:	
\begin{itemize}
\renewcommand\labelitemi{--}
\item Associativity:
\small
$$\left(\DDDD (\sigma _ {{}_{01}} )_ {{}_f}\right)\left(\DDDD (\partial ^0)(\left\langle  \overline{\mathsf{f}}\right\rangle)\right) \left( \DDDD (\sigma _ {{}_{21}}) _ {{}_{f}}\right)\left(\DDDD (\partial ^2)(\left\langle  \overline{\mathsf{f}}\right\rangle )\right) = \left(\DDDD(\partial ^1)(\left\langle  \overline{\mathsf{f}}\right\rangle)\right) \left(\DDDD (\sigma _ {{}_{00}}) _ {{}_{f}}\right)  $$
\normalsize
\item Identity:
\small
$$\left(\DDDD(s^0) (\left\langle  \overline{\mathsf{f}}\right\rangle) \right)\left(\DDDD(n_0) _ {{}_f}\right) = \left(\DDDD(n_1) _ {{}_f}\right) $$
\normalsize
\end{itemize}
If $\mathtt{f}: \dot{\Delta }(\mathsf{0}, -)\longrightarrow \DDDD $ is a $2$-natural transformation, we get such pair by the correspondence 
$\mathtt{f}\mapsto (\mathtt{f} _ {{}_{\mathsf{1} }}(d), \mathtt{f} _ {{}_{\mathsf{2} }}(\vartheta )) $.

\item The morphisms are modifications. In other words, a morphism $\mathfrak{m} : \mathtt{f}\to\mathtt{h} $ is determined by a morphism $\mathfrak{m}: f\to g $ in $\DDDD \mathsf{1} $ such that $\DDDD (d^1)(\mathfrak{m} )\left\langle   \overline{\mathsf{f}}\right\rangle =  \left\langle \overline{\mathsf{h}}\right\rangle\DDDD (d^0)(\mathfrak{m} ) $.
\end{enumerate}

\end{rem}
Similarly to Proposition 4.5.5~\cite{MR3491845} and 5.4.5~\cite{2016arXiv160703087L}, we clearly have: 

\begin{prop}\label{proposicao igual PhD}
Let $\TTTTT = (\TTTTT,  m  , \eta , \mu,   \iota, \tau )$ be a pseudomonad on a $2$-category $\bbb $. Given lax $\TTTTT$-algebras 
$\mathtt{y}=(Y, \alg _ {{}_{\mathtt{y}}}, \overline{{\underline{\mathtt{y}}}}, \overline{{\underline{\mathtt{y}}}}_0 )$, $\mathtt{z}= (Z, \alg _ {{}_{\mathtt{z}}}, \overline{{\underline{\mathtt{z}}}}, \overline{{\underline{\mathtt{z}}}}_0 )$
the category $\mathsf{Lax}\textrm{-}\TTTTT\textrm{-}\Alg _{\mathsf{c}\ell } (\mathtt{y}, \mathtt{z}) $ is the strict colax descent object  of the diagram
$\mathbb{T}_{\mathtt{z}}^{\mathtt{y}}: \Delta _ \ell \to \Cat $
\begin{equation*}
\xymatrix@R=2em{  \bbb (U\mathtt{y}, U\mathtt{z})\ar@<-3ex>[rrr]_-{\bbb(\alg _ {{}_{\mathtt{y}}}, U\mathtt{z} )}\ar@<3ex>[rrr]^-{\bbb(\TTTTT U\mathtt{y}, \alg _ {{}_{\mathtt{z}}} )\circ \hspace{0.2em}\TTTTT _ {{}_{(U\mathtt{y},U\mathtt{z})}} } &&& \bbb (\TTTTT U\mathtt{y},  U\mathtt{z} )\ar[lll]|-{\bbb (\eta _ {{}_{U\mathtt{y}}} , U\mathtt{z})}
\ar@<-3ex>[rrr]_-{\bbb(\TTTTT (\alg _ {{}_{\mathtt{y}}}), U\mathtt{z} )}\ar[rrr]|-{\bbb (m _{{}_{U\mathtt{y}}}, U\mathtt{z} )}\ar@<3ex>[rrr]^-{\bbb(\TTTTT ^2 U\mathtt{y}, \alg _ {{}_{\mathtt{z}}} )\circ \hspace{0.2em}\TTTTT _ {{}_{(\TTTTT U\mathtt{y},U\mathtt{z})}} } &&& \bbb (\TTTTT ^2 U \mathtt{y}, U\mathtt{z} ) }
\end{equation*}
in which $U:\mathsf{Lax}\textrm{-}\TTTTT\textrm{-}\Alg _{\mathsf{c}\ell }\to \bbb $ denotes the forgetful $2$-functor and
\small
\begin{equation*}
\begin{aligned}
&\mathbb{T}_{\mathtt{z}}^{\mathtt{y}}(\sigma _ {01}) _{{}_{f}} &:=& \left( \id _ {{}_{  \alg _ {{}_{\mathtt{z}}} }}\ast \tttt _ {{}_{(f)(\alg _ {{}_{\mathtt{y}}}) }}^{-1}  \right)   \\
&\mathbb{T}_{\mathtt{z}}^{\mathtt{y}}(\sigma _ {21}) _{{}_{f}} &:=&  \left( \id_ {{}_{f}}\ast \overline{{\underline{\mathtt{y}}} }\right)\\
&\mathbb{T}_{\mathtt{z}}^{\mathtt{y}}(n _ {1}) _{{}_{f}} &:=& \left( \id _ {{}_{f}}\ast \overline{{\underline{\mathtt{y}}}}_0 \right)
\end{aligned}
\begin{aligned}
&&&\\
&\mathbb{T}_{\mathtt{z}}^{\mathtt{y}}(\sigma _ {00}) _{{}_{f}} &:=& \left( \id _ {{}_{ \alg _ {{}_{\mathtt{z}}}  }}\ast m _ {{}_{f}}^{-1}\right)\cdot\left( \overline{{\underline{\mathtt{z}}}}\ast\id _ {{}_{\TTTTT ^2(f)  }}\right)\cdot
\left(\id _ {{}_{\alg _ {{}_{\mathtt{z}}} }}\ast \tttt _ {{}_{(\alg _ {{}_{\mathtt{z}}}) (\TTTTT(f)) }}^{-1} \right) \\
&\mathbb{T}_{\mathtt{z}}^{\mathtt{y}}(n _ {0}) _{{}_{f}} &:=&  
\left(\id _{{}_{\alg _ {{}_{\mathtt{z}}}}}\ast\eta _ {{}_{f}}^{-1} \right)\cdot \left(\overline{{\underline{\mathtt{z}}}}_0\ast\id _ {{}_{f}}\right)
\end{aligned}
\end{equation*}
\normalsize
Furthermore, the strict descent object of $\mathbb{T}_{\mathtt{z}}^{\mathtt{y}}$ is $\mathsf{Lax}\textrm{-}\TTTTT\textrm{-}\Alg  (\mathtt{y}, \mathtt{z}) $.
\end{prop}

\begin{rem}\label{NOVO COLAX PhD}
Similarly to the case of Remark 5.4.6~\cite{2016arXiv160703087L}, we can actually define a pseudofunctor $\mathbb{T}^{\mathtt{y}}: \Delta _ \ell ^\mathsf{c} \times \mathsf{Lax}\textrm{-}\TTTTT\textrm{-}\Alg \to\Cat $ in which $\mathbb{T}^{\mathtt{y}}(-,{\mathtt{z}}) : = \mathbb{T}^{\mathtt{y}}_ {\mathtt{z}}$, since the morphisms defined above are actually pseudonatural in $\mathtt{z}$ w.r.t. $\TTTTT $-pseudomorphisms and $\TTTTT $-transformations. More importantly to our context, if $\TTTTT $ is a $2$-monad, the restriction of 
$\mathbb{T}^{\mathtt{y}} $
to $\Delta _ \ell ^\mathsf{c} \times \mathsf{Lax}\textrm{-}\TTTTT\textrm{-}\Alg _\textrm{s}$, in which $\mathsf{Lax}\textrm{-}\TTTTT\textrm{-}\Alg _\textrm{s} $
denotes the locally full sub-$2$-category of lax algebras and (strict) $\TTTTT $-morphisms and lax algebras, is actually a $2$-functor.
\end{rem}

By Remark \ref{NOVO COLAX PhD} and Proposition \ref{proposicao igual PhD}, as a consequence of the results of Chapter 5~\cite{2016arXiv160703087L}, we conclude as a particular case that:

\begin{theo}
Let $\TTTTT = (\TTTTT , m , \eta ) $ be a $2$-monad on $\bbb $ and $(E\dashv R, \epsilon ,\eta ):\bbb \to \TTTTT\textrm{-}\Alg _ \textrm{s} $ the Eilenberg-Moore $2$-adjunction induced by $\TTTTT $. The inclusion $J: \TTTTT\textrm{-}\Alg _ \textrm{s}\to \mathsf{Lax}\textrm{-}\TTTTT\textrm{-}\Alg  _{\mathsf{c}\ell } $
has a left $2$-adjoint if and only if $\TTTTT\textrm{-}\Alg _ \textrm{s}$  has the strict colax codescent object of
\small
\begin{equation*}\tag{$\BBBB_\mathtt{y}$}
\xymatrix{  E U\mathtt{y}\ar[rrrr]|-{E(\eta _{{}_{U\mathtt{y}}})} &&&&  E\TTTTT U\mathtt{y}\ar@<3ex>[llll]^-{E(\alg _ {{}_{\mathtt{y}}})}\ar@<-3ex>[llll]_-{\varepsilon _ {{}_{EU\mathtt{y}}}E( \alg_{{}_{JEU\mathtt{y}}}\TTTTT (\eta _ {{}_{U\mathtt{y} }} ))   } 
 &&&& \ar@<3ex>[llll]^-{E\TTTTT (\alg _{{}_{\mathtt{y}}})}\ar[llll]|-{E(m _ {{}_{U\mathtt{y}}} )}\ar@<-3ex>[llll]_-{ \varepsilon _ {{}_{E\TTTTT U\mathtt{y}}}E( \alg_{{}_{JE\TTTTT U\mathtt{y}}}\TTTTT (\eta _ {{}_{\TTTTT U\mathtt{y} }} )) } E\TTTTT ^2 U\mathtt{y} }
\end{equation*}
\normalsize
(with omitted $2$-cells) in which $U: \mathsf{Lax}\textrm{-}\TTTTT\textrm{-}\Alg  _{\mathsf{c}\ell }\to\bbb $ is the forgetful $2$-functor.
In this case, the left $2$-adjoint is given by $G\mathtt{y} =   \dot{\Delta }_\ell ^\mathsf{c} (\mathsf{0}, \t ^\mathsf{c} -)\Asterisk \BBBB_\mathtt{y}  $.
\end{theo}

As a corollary, if $\TTTTT $ is a $2$-monad on $\bbb $, the inclusion $\TTTTT\textrm{-}\Alg _ \textrm{s}\to \mathsf{Lax}\textrm{-}\TTTTT\textrm{-}\Alg  _{\mathsf{c}\ell } $ has a left $2$-adjoint
if $\bbb $ has and $\TTTTT $ preserves strict colax codescent objects. Dually, if $\SSSSS $ is a $2$-comonad on $\bbb $, $\SSSSS\textrm{-}\CoAlg _ \textrm{s}\to \mathsf{Lax}\textrm{-}\SSSSS\textrm{-}\CoAlg  _{\mathsf{c}\ell } $ has a right $2$-adjoint if $\bbb $ has and $\SSSSS $ preserves strict colax descent objects.

Since the $2$-monad
$\overline{\Lan } $ on $ \left[ \aaa _ 0 , \bbb \right] $ preserves strict colax codescent objects whenever  $\aaa $ is small and $\bbb $ has $2$-colimits, 
our result gives a left $2$-adjoint to the inclusion $\left[ \aaa  , \bbb \right]\to \left[ \aaa  , \bbb \right] _ {Lax _\mathsf{c} } $.
Dually, the $2$-comonad $\overline{\Ran } $ on 
$ \left[ \aaa _ 0 , \bbb \right] $ preserves strict colax descent objects whenever $\bbb $ has $2$-limits, hence our result 
gives the right $2$-adjoint of $\left[ \aaa  , \bbb \right]\to \left[ \aaa  , \bbb \right] _ {Lax }$.
These facts explain how our results generalize greatly the constructions of \cite{MR0347936}. In particular, unlike \cite{MR0347936},
our setting includes the case of
 lax actions of monoidal categories (called graded monads): or, more precisely, the case in which $\aaa $ is a bicategory with only one object
 (see Remark 5.4.3~\cite{2016arXiv160703087L} and \cite{MR3492435}).

Applying to the very special case of the classical theory of monads, we get that the Eilenberg-Moore category of a monad (seen as a lax coalgebra) $\mathtt{y}= (Y, \coalg _ {{}_{\mathtt{y}}}, \overline{{\underline{\mathtt{y}}}}, 
\overline{{\underline{\mathtt{y}}}}_0 )$  is the colax descent category of
\begin{equation*}\tag{$\DDDD $}
\xymatrix{  Y
\ar@<3ex>[rrrr]^-{\coalg _ {{}_{\mathtt{y}}}}\ar@<-3ex>[rrrr]_-{\id _ {{}_{Y}} } &&&&  Y\ar[llll]|-{\id _ {{}_{Y}} }
 \ar@<3ex>[rrrr]^-{\coalg _{{}_{\mathtt{y}}}}\ar[rrrr]|-{\id _ {{}_{Y}} }\ar@<-3ex>[rrrr]_-{ \id _{{}_{Y}} } 
 &&&& Y }
\end{equation*}
in which $\DDDD (\sigma _{00} ) = \overline{{\underline{\mathtt{y}}}} $, $\DDDD (n_0 ) = \overline{{\underline{\mathtt{y}}}}_0  $
and the images of $\sigma _ {01}$, $\sigma _{21} $ and $n_1 $ are the identities.

\begin{rem}\label{Eu nao pretendo fazer mas deixo como trabalho futuro}
The article \cite{MR0347936} also recasts the original universal properties w.r.t. adjunctions. That is to say, it generalizes the (universal
property) of the (generalized) Eilenberg-Moore and Kleisli 
adjunctions to its setting. In order to do so, it relies on the study of the counit and unit of the $2$-adjunctions $\left[ X , \bbb \right]\to \left[ X  , \bbb \right] _ {Lax _\mathsf{c} } $ and  $\left[ X , \bbb \right]\to \left[ X  , \bbb \right] _ {Lax }$ constructed therein. This fact shows that the study
of counit and unit of the obtained biadjunctions (and $2$-adjunctions) in the context of Chapter 5 could be interesting to recast the Eilenberg-Moore and
Kleisli adjunctions in our context, in order to generalize the setting of \cite{MR0347936}.
\end{rem}

\section{Pseudoexponentiability}
There is a vast literature on exponentiability of objects and morphisms within $1$-dimensional category theory~\cite{MR639571, MR1815044}. 
As mentioned in Section \ref{Suspension Categorias Monoidais PhD Capitulo 1}, we could consider exponentiability of objects
w.r.t. other monoidal structures in $V$, but we are particularly interested in the case of exponentiation w.r.t. the cartesian structure.

An object $A$ of a cartesian category $V$ is exponentiable if the functor
 $A\times - : V\to V $  is left adjoint. In this case,
the right adjoint is usually denoted by $\left[ A, -\right] $. A morphism $f: A\to B $ is exponentiable if it is an exponentiable object
of the comma category $V/B $ defined by: 
\begin{itemize}
\item The objects are morphisms with $B$ as codomain;
\item A morphism $f\to g $ is a morphism $h$ of $V$ between the domains of $f$ and $g$ such that
$$\xymatrix{
A\ar[rr]|-{h}
\ar[rd]|-{f}
&
&
A'
\ar[ld]|-{g}\\
&B&
} $$
commutes in $V$;
\item The composition and identities are given by the composition and identities of $V$.
\end{itemize}

The main problem is to characterize objects and morphisms that are 
exponentiable in a given category $V$ of interest. For instance, the characterization of 
the exponentiable morphisms (functors) of $\Cat $ is given in \cite{MR0190142, MR0310033}, 
while in \cite{MR2471250, MR2259813} the exponentiable morphisms (enriched functors)
 between $V$-enriched categories are characterized, for suitable monoidal categories $V$.   

As many concepts of $1$-dimensional category theory, exponentiability  is very strict to most of the cases within bicategory theory. Hence we should consider a weaker version: pseudoexponentiability.
Firstly, we consider bicategorical products instead of products. Secondly, we consider a biadjunction instead of an adjunction. That is to say:

\begin{defi}
An object $Y$ of a $2$-category $\bbb $ with bicategorical products is \textit{pseudoexponentiable} if the pseudofunctor $Y\times - : \bbb \to \bbb $ has a right biadjoint,
while $Y$ is exponentiable if $Y\times - $ is a $2$-functor and it has a right $2$-adjoint.
\end{defi}

As briefly mentioned in Section \ref{Capitulo4 PhD}, the adjoint triangle theorem of \cite{MR0233864} has many applications in $1$-dimensional category theory.
In particular, it is very useful within the study of exponentiable objects and morphisms. For instance, some of the results of the theory developed in 
\cite{MR920516}
can be seen as applications of the adjoint triangle theorem.

This fact suggests the possibility of applying the biadjoint triangle theorem proved in Chapter 4 to develop an analogue
theory for pseudoexponentiable objects and morphisms. We however do not do this here. Instead, we finish the chapter giving what seems to be 
a folklore result on exponentiability of coalgebras. Then, employing the biadjoint triangle theorem, we give the bicategorical analogue that studies the (pseudo)exponentiability
of pseudocoalgebras.

\begin{theo}[Exponentiability of Coalgebras]\label{Exponencial1 PhD}
Let $\SSSSS  $ be a comonad on a finitely complete category $\bbb $. If $\SSSSS $ preserves finite limits, then 
the forgetful functor $L: \SSSSS\textrm{-}\CoAlg \to \bbb $ reflects exponentiable objects.
\end{theo}
\begin{proof}
In this setting, $L$ creates finite limits. In particular, given an $\SSSSS$-coalgebra $\mathtt{y}= (Y, \coalg _ {{}_{\mathtt{y}}} ) $, we get that
$$\xymatrix@=2.8em{
\SSSSS\textrm{-}\CoAlg _ {\textrm{s}}\ar[r]^-{\mathtt{y}\times - }\ar[d]_ {L}
&
\SSSSS\textrm{-}\CoAlg _ {\textrm{s}}
\ar[d]^{L}
\\
\bbb\ar[r]_{(L\mathtt{y})\times - }
&
\bbb
}
$$ 
commutes. If $(L\mathtt{y})\times - \dashv \left[ (L\mathtt{y}), - \right] $, 
we have that $(L\mathtt{y})\times L(-) \dashv U\left[ (L\mathtt{y}), - \right] $. Since $L$ is comonadic, by the adjoint triangle theorem of Dubuc, we conclude
that $\mathtt{y}\times - $ has a right adjoint.
\end{proof}

Within the context of the theorem above, given an $\SSSSS$-coalgebra $\mathtt{y} $, by the Beck theorem, it is
clear that $L$ induces a functor $L/\mathtt{y}: \SSSSS\textrm{-}\CoAlg _ {\textrm{s}}/\mathtt{y}\to \bbb /L(\mathtt{y}) $
which creates limits and is comonadic  as well. Therefore:

\begin{coro}
Let $\SSSSS  $ be a comonad on a finitely complete category $\bbb $. If $\SSSSS $ preserves finite limits, then 
the forgetful functor $L: \SSSSS\textrm{-}\CoAlg _ {\textrm{s}}\to \bbb $ reflects exponentiable morphisms.
\end{coro}

\begin{rem}\label{Categorias de Functores e Exponenciacao Capitulo Intrdutorio PhD}
An elementary application is given by the category of functors. For instance, if $X, Y $ are categories such that $X$ is small
and $Y$ is complete, there exists a (global) pointwise right Kan extension $\Cat \left[ X _ 0, Y\right] \to \Cat\left[ X , Y\right] $.
Since the forgetful functor $\Cat \left[ X, Y\right] \to \Cat\left[ X_0 , Y\right] $ creates equalizers,
 we conclude that $\Cat \left[ X, Y\right] \to \Cat\left[ X_0 , Y\right] $ is comonadic.

By Theorem \ref{Exponencial1 PhD}, we conclude that   $\Cat \left[ X, Y\right] \to \Cat\left[ X_0 , Y\right] $ 
 reflects exponentiable 
objects. More generally, a natural transformation is exponentiable in $\Cat \left[ X, Y\right] $ whenever it is objectwise exponentiable.
\end{rem}

From an argument entirely analogous to that given in the proof of Theorem \ref{Exponencial1 PhD}, 
using the enriched version of the adjoint triangle theorem as presented in Section 4.1~\cite{MR3491845}, we get  that:

\begin{theo}[Exponentiability of Strict Coalgebras]
Let $\SSSSS  $ be a $2$-comonad on a $2$-category $\bbb $. If $\bbb $ has and $\SSSSS $ preserves products and equalizers, then 
the forgetful $2$-functor $L: \SSSSS\textrm{-}\CoAlg _ {\textrm{s}}\to \bbb $ reflects exponentiable objects.
\end{theo}

Recall the definition of (strict) descent objects given in Section 4.3~\cite{MR3491845}.
Employing the strict version of the biadjoint triangle theorem given in Theorem 4.5.10~\cite{MR3491845}, we can study the exponentiability of
pseudocoalgebras.

\begin{theo}[Exponentiability of Pseudocoalgebras]
Let $\SSSSS  $ be a $2$-comonad on a $2$-category $\bbb $. If $\bbb $ has and $\SSSSS $ preserves products and strict descent objects, then 
the forgetful $2$-functor $L: \mathsf{Ps}\textrm{-}\SSSSS\textrm{-}\CoAlg \to \bbb $ reflects exponentiable objects.
\end{theo}

Finally, by the biadjoint triangle theorem given in Theorem 5.10 of Chapter 4~\cite{MR3491845}, we conclude that:

\begin{theo}[Pseudoexponentiability of Pseudocoalgebras]
Let $\SSSSS  $ be a pseudocomonad on a $2$-category $\bbb $. If $\bbb $ has and $\SSSSS $ preserves biproducts and descent objects, then 
the forgetful $2$-functor $L: \mathsf{Ps}\textrm{-}\SSSSS\textrm{-}\CoAlg \to \bbb $ reflects pseudoexponentiable objects.

\end{theo}

\begin{rem}
Similarly to the case of Remark \ref{Categorias de Functores e Exponenciacao Capitulo Intrdutorio PhD}, we can use the results above to 
study exponentiability and pseudoexponentiability of the $2$-category  of $2$-functors $\left[ \aaa , \bbb \right] $ (with $2$-natural transformations and modifications) and in the $2$-category of pseudofunctors
$\left[ \aaa , \bbb \right] _ {PS} $ (with pseudonatural transformations and modifications) as defined in Section 2 of Chapter 4~\cite{MR3491845}. 

For instance, if $\aaa $ is small and $\bbb $ is $2$-complete, we conclude that
a $2$-functor is exponentiable in $\left[ \aaa , \bbb \right] $ if it is objectwise exponentiable.
Moreover, using the pointwise pseudo-Kan extension constructed in 4.9.2~\cite{MR3491845} (or in 3.3.5, Section 3 of Chapter 3), we get an analogous result for pseudofunctors. More precisely, assuming that
$\bbb $ is bicategorically complete and $\aaa $ is small, we get that a pseudofunctor in   $\left[ \aaa , \bbb \right] _ {PS} $ is pseudoexponentiable 
whenever it is objectwise pseudoexponentiable.
\end{rem}




\begin{spacing}{0.9}


\bibliographystyle{apalike}
\cleardoublepage
\bibliography{references} 



\end{spacing}



\end{document}